\documentclass[oneside,english,reqno]{amsart}

\usepackage{xparse}
\let\realItem\item 
\makeatletter
\NewDocumentCommand\myItem{ o }{%
   \IfNoValueTF{#1}%
      {\realItem}
      {\realItem[#1]\def\@currentlabel{#1}}
}
\makeatother

\usepackage{enumitem}    
\setlist[enumerate]{
    before=\let\item\myItem,       
    label=\textnormal{(\arabic*)}, 
    widest=(2')                    
}

\RequirePackage[OT1]{fontenc}
\RequirePackage{amsthm,amsmath}
\RequirePackage[numbers]{natbib}
\usepackage{amssymb}
\usepackage{extarrows}
\usepackage{multirow}
\usepackage{natbib}
\usepackage{stmaryrd}
\usepackage{color}
\usepackage{amsmath}
\usepackage{enumitem}
\usepackage{mathtools}
\usepackage{dsfont} 
\usepackage{comment}
\usepackage{soul}
\usepackage{amsmath,amsthm,amssymb,bbm}
\usepackage{mathrsfs}

\usepackage{amsmath, amscd, amsthm, amssymb, mathrsfs,amsfonts}
\usepackage{graphicx}
\usepackage{rotating}
\usepackage{subfig}

\usepackage[T1]{fontenc}
\usepackage[utf8]{inputenc}

\numberwithin{equation}{section}
\DeclareMathOperator{\tr}{Tr}
\DeclareMathOperator{\ts}{tr}
\DeclareMathOperator{\cov}{Cov}

\DeclareMathOperator{\supp}{Supp}
\DeclareMathOperator{\id}{id}

\DeclareMathOperator{\Ree}{Re}
\DeclareMathOperator{\ev}{ev}

\let\limsup\relax
\DeclareMathOperator*{\limsup}{limsup}

\newcommand{\norm}[1]{\left\Vert #1\right\Vert}

\newcommand{\equaldist}{\stackrel{\lower0.5pt\hbox{$\scriptstyle\mathrm d$}}=}

\theoremstyle{plain}
\newtheorem{theorem}{Theorem}[section]
\newtheorem*{nono-theorem}{Theorem}
\newtheorem{prop}[theorem]{Proposition} 
\newtheorem{defi}[theorem]{Definition} 
\newtheorem{exe}[theorem]{Example} 
\newtheorem{lemma}[theorem]{Lemma} 
\newtheorem{rem}[theorem]{Remark} 
\newtheorem{cor}[theorem]{Corollary} 
\newtheorem{conj}[theorem]{Conjecture}

\newtheorem{notation}[theorem]{Notation}
\newtheorem*{ack}{Acknowledgments}

\theoremstyle{definition}
 \newtheorem{claim}[theorem]{Claim}

\renewcommand\P{\mathbb{P}}
\newcommand\E{\mathbb{E}}
\newcommand\N{\mathbb{N}}
\newcommand\M{\mathbb{M}}

\newcommand\C{\mathbb{C}}
\newcommand\CC{\mathcal{C}}
\newcommand\PP{\mathcal{P}}
\newcommand\R{\mathbb{R}}
\newcommand\A{\mathscr{A}}

\renewcommand\rho{\varrho}
\newcommand{\deq}{\mathrel{\mathop:}=}
\newcommand\Binv{\rotatebox[origin=c]{180}{$B$}}

\newcommand{\ignore}[1]{}

\newcommand\nc\newcommand
\nc\eps{\varepsilon}
\nc\ls{\lesssim}
\nc\gs{\gtrsim}
\nc\nicked[1]{{\color{magenta} #1}}
\nc\wt\widetilde
\nc{\eqd}{\stackrel{\text{\tiny $d$}}{=}}
\nc\cG{\mathcal G}
\nc\cE{\mathcal E}
\def \tran {\mathsf{T}}
\nc{\delW}{Z}
\nc{\proB}{\Gamma}
\nc\bs\boldsymbol

\allowdisplaybreaks

\usepackage[linktocpage,colorlinks,linkcolor=blue,anchorcolor=blue,citecolor=blue,urlcolor=black,pagebackref]{hyperref}

\renewcommand*{\backref}[1]{\ifx#1\relax \else Page #1 \fi}
\renewcommand*{\backrefalt}[4]{
  \ifcase #1 \footnotesize{(Not cited.)}
  \or        \footnotesize{(Cited on page~#2.)}
  \else      \footnotesize{(Cited on pages~#2.)}
  \fi
}

\usepackage{todonotes}

\NewCommandCopy{\dotlessi}{\i}
\renewcommand\i{\mathbf{i}}

\begin{document}

\title{Eigenvalues of Brownian Motions on $\mathrm{GL}(N,\mathbb{C})$}

\author[T. Brailovskaya]{Tatiana I.\ Brailovskaya\textsuperscript{1}}
\address{\textsuperscript{1}Department of Mathematics, Duke University, Durham, NC 27708-0320}
\email{tatiana.brailovskaya@duke.edu}

\author[N. Cook]{Nicholas A.\ Cook\textsuperscript{2}}
\address{\textsuperscript{2}Department of Mathematics, Duke University, Durham, NC 27708-0320}
\email{nicholas.cook@duke.edu}

\author[T. Kemp]{Todd Kemp\textsuperscript{3}}
\address{\textsuperscript{3}Department of Mathematics, UC San Diego, La Jolla, CA 92093-0112}
\email{tkemp@ucsd.edu}
\thanks{\textsuperscript{3}Supported in part by NSF Grant DMS-2400246}

\author[F. Parraud]{F\'elix Parraud\textsuperscript{4}}
\address{\textsuperscript{4}Department of Mathematics, Queen's university, K7L 3N8, Kingston, Canada.}
\email{felix.parraud@gmail.com}




\begin{abstract} 

	We prove that the empirical law of eigenvalues of Brownian motion on the Lie Group $\mathrm{GL}(N,\mathbb{C})$ converges almost surely to a deterministic probability measure, characterized by a free stochastic differential equation.  This fully resolves a conjecture made by Philippe Biane in 1997. Our analysis includes a family $\{B=B_{\rho,\zeta}\colon \rho>0,|\zeta|\le\rho\}$ of diffusion processes on $\mathrm{GL}(N,\mathbb{C})$ whose laws are invariant under unitary conjugation.
	
	The crux of our analysis is a strong quantitative approximation of Brownian motion $B(t)$ on $\mathrm{GL}(N,\mathbb{C})$ for small $t$ by a single increment $I+W(t)$, where $W=W_{\rho,\zeta}$ is an elliptical Brownian motion in the Lie algebra $\mathfrak{gl}(N,\mathbb{C}) = \mathbb{M}_N(\mathbb{C})$. Specifically, for any $t\in[0,1]$ and $\delta>0$, 
	\[ 
	\P\left(\|B(t)-I-W(t)\|\geq \delta\right)\leq \left(C t/\delta\right)^{N^{2/3}}
	\]
	for a constant $C=C_\rho$. Leveraging independence of multiplicative increments of the Brownian motion then allows us to use powerful (anti-)concentration tools for Gaussian matrices to complete the Hermitization procedure for convergence of eigenvalues.  
\end{abstract}

\maketitle

\tableofcontents

\section{Introduction\label{Intro}}

In this paper, we answer the longstanding question of convergence of eigenvalues of Brownian motions on the Lie group $\mathrm{GL}(N,\mathbb{C})$ of complex invertible $N\times N$ matrices. 

Brownian motion and heat kernel analysis on Lie groups have been studied extensively since the 1970s; see the survey \cite{Varopoulos} and especially \cite{Taylor1988,Gross1993,Hall1994,Driver1995} for perspective relevant to the present work. 
One effective construction of Lie {\em group} Brownian motion is through a stochastic differential equation driven by ``flat'' Brownian motion on the Lie {\em algebra}; see \eqref{eq.BM.Strat.SDE}.  This provides a direct connection to the classical ensembles of random matrix theory.  The unitary group $\mathrm{U}(N)$ has, as Lie algebra, the space of skew-Hermitian matrices in $\M_N(\C)$; the ``flat'' Brownian motion on this space is $\i X(t)$ where, for each $t\ge 0$, the Hermitian matrix $X(t)$ is a {\em Gaussian Unitary Ensemble} (scaled by $\sqrt{t}$).  Meanwhile, the general linear group $\mathrm{GL}(N,\mathbb{C})$ has as Lie algebra the full space $\M_N(\C)$; the standard ``flat'' Brownian motion $W(t)$ here is, for each $t\ge0$, a matrix of all i.i.d.\ complex Gaussian random variables, i.e.\ a {\em Ginibre Ensemble} (scaled by $\sqrt{t}$).

These ensembles comprise two of the most well-known settings for large-$N$ limits of eigenvalue distributions: Wigner \cite{Wigner1955,Wigner1958} proved the density of eigenvalues of $X(1)$ converges to the semicircle law, while Ginibre \cite{Ginibre1965} proved that $W(1)$'s eigenvalues converge to the circular law, as $N\to\infty$.  Wigner's semicircle law was extended to matrices with i.i.d.\ entries above the diagonal under the optimal finite second moment hypothesis within a decade; cf. \cite{Arnold1,Arnold2,Grenander}.  On the other hand, the analogous extension  of the circular law
was not accomplished until 45 years after Ginibre's work in \cite{TaoVu2010-aop}, following incremental progress by several authors -- see the survey \cite{Bordenave-Chafai-circular}.  This is a testament to the vastly greater difficulty working with eigenvalues of non-normal matrices than with selfadjoint matrices. 

As curved-space complements to these classical theorems, it is very natural therefore to study large-$N$ limits of eigenvalues of Brownian motions on $\mathrm{U}(N)$ and $\mathrm{GL}(N,\C)$.  The unitary case was initiated by Philippe Biane in \cite{Biane1997b,Biane1997JFA}, with some related work by Eric Rains around the same time \cite{Rains1997}.  Inspired by earlier work of Applebaum, Hudson, and Parthasarathy \cite{AH1984,HP1984} and especially Bo\.{z}ejko, K\"ummerer,  and Speicher \cite{BS1991,KS1992}, Biane set the framework for free stochastic calculus, and used it to construct a noncommutatve process $(u(t))_{t\ge 0}$ he called {\em free unitary Brownian motion} as a solution to a free SDE (the free version of \eqref{eq.UBM.SDE}).  He then proved that, with the correct scaling of the metrics on the groups $\mathrm{U}(N)$, the unitary Brownian motion $U(t)\in\mathrm{U}(N)$ converges to $u(t)$ as $N\to\infty$ in {\em $\ast$-distribution}: convergence of traces of polynomial functions of the process at different times.  Given that both $U(t)$ and $u(t)$ are unitary operators, this sufficed to prove convergence of the density of eigenvalues of $U(t)$ to a probability measure determined by $u(t)$.  This convergence result, and the free unitary Brownian motion process, found many uses in the theory of free entropy (e.g.\ \cite{Voiculescu1999,CollinsKemp2014}). Many more refined convergence results were explored in the ensuing two decades, including concentration at the edge and absence of outliers \cite{Kemp2018}, strong metric convergence and mesoscopic scale analysis of eigenvalue evolution \cite{Melcher2018}, linear fluctuations \cite{LevyMaida2010,KempCebron2022}, applications to Yang--Mills theory \cite{KempMM1,KempMM2}, and connections with non-intersecting Brownian motions \cite{Liechty}.

In parallel, Biane introduced the {\em free multiplicative Brownian motion} $(b(t))_{t\ge 0}$ -- another non-commutative stochastic process that satisfies a free SDE (a special case of \eqref{eq.b.fSDE}).  In \cite{Biane1997JFA}, it played a central role in constructing a free version of the Segal--Bargmann--Hall transform -- a unitary isomorphism in geometric quantization theory that concretely implements wave-particle duality. The transform intertwines heat kernels (i.e.\ marginal distributions of Brownian motions) on $\mathrm{U}(N)$ and its complexification $\mathrm{U}(N)_{\C} = \mathrm{GL}(N,\C)$.  While connected by analogy with Brownian motion $B(t)$ on $\mathrm{GL}(N,\C)$, Biane wrote \cite[p.\ 19]{Biane1997b} ``It is very likely that the [free multiplicative Brownian motion] process is the limit in distribution\ldots of the Brownian motion with values in $\mathrm{GL}(N,\C)$\ldots but we have not proved this.''  The functional calculus and Fourier analytic methods Biane invented to prove convergence of eigenvalues in the unitary setting are unavailable for $\mathrm{GL}(N,\C)$, making the problem much harder for the same reason the circular law is much harder than its semicircular cousin: most matrices in $\mathrm{GL}(N,\C)$ are non-normal, so the spectral theorem is no use.  (Indeed, in \cite[Prop.\ 4.15]{Kemp2016}, it was shown that, with probability $1$, $B(t)$ is not normal at any time $t>0$ for $N>1$.)

\begin{figure}[h!]
\centering
  \begin{tikzpicture}

    \node[inner sep=2pt] at (0cm, 0cm)
      {\includegraphics[width=4cm]{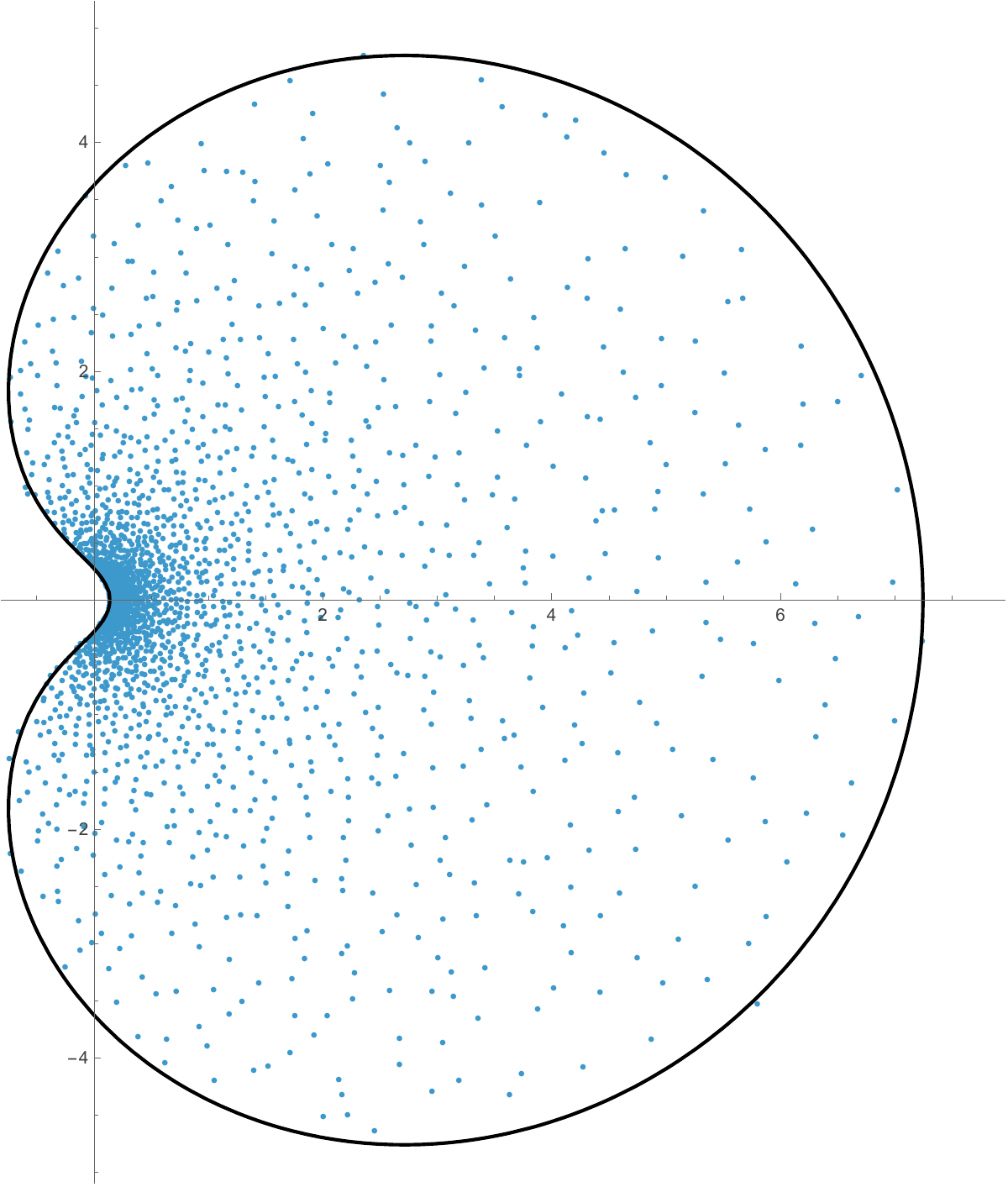}};

    \node[inner sep=2pt] at (4.2cm, -0.07cm)   
      {\includegraphics[width=4cm]{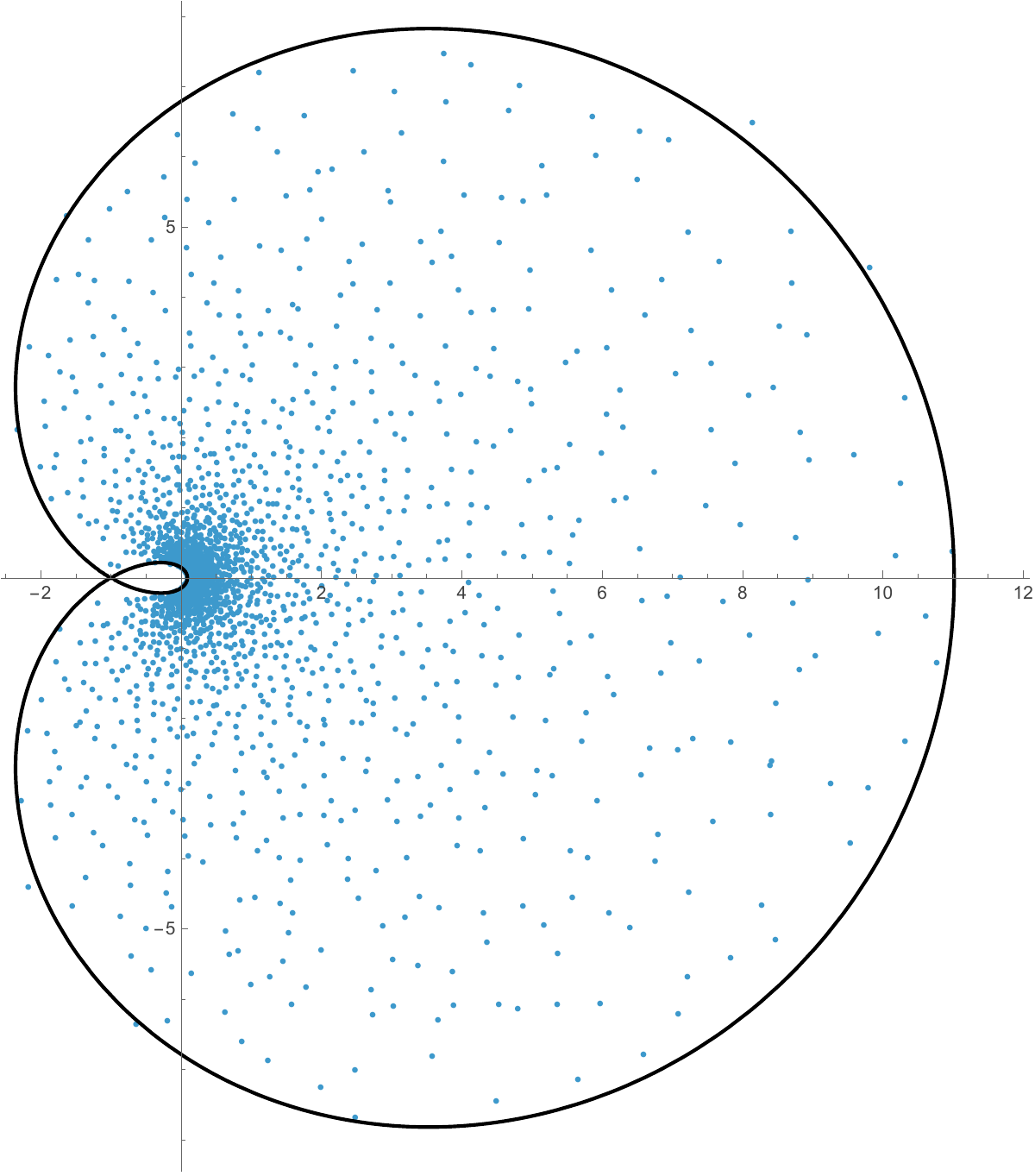}};

    \node[inner sep=2pt] at (8.4cm, -0.03cm)   
      {\includegraphics[width=4cm]{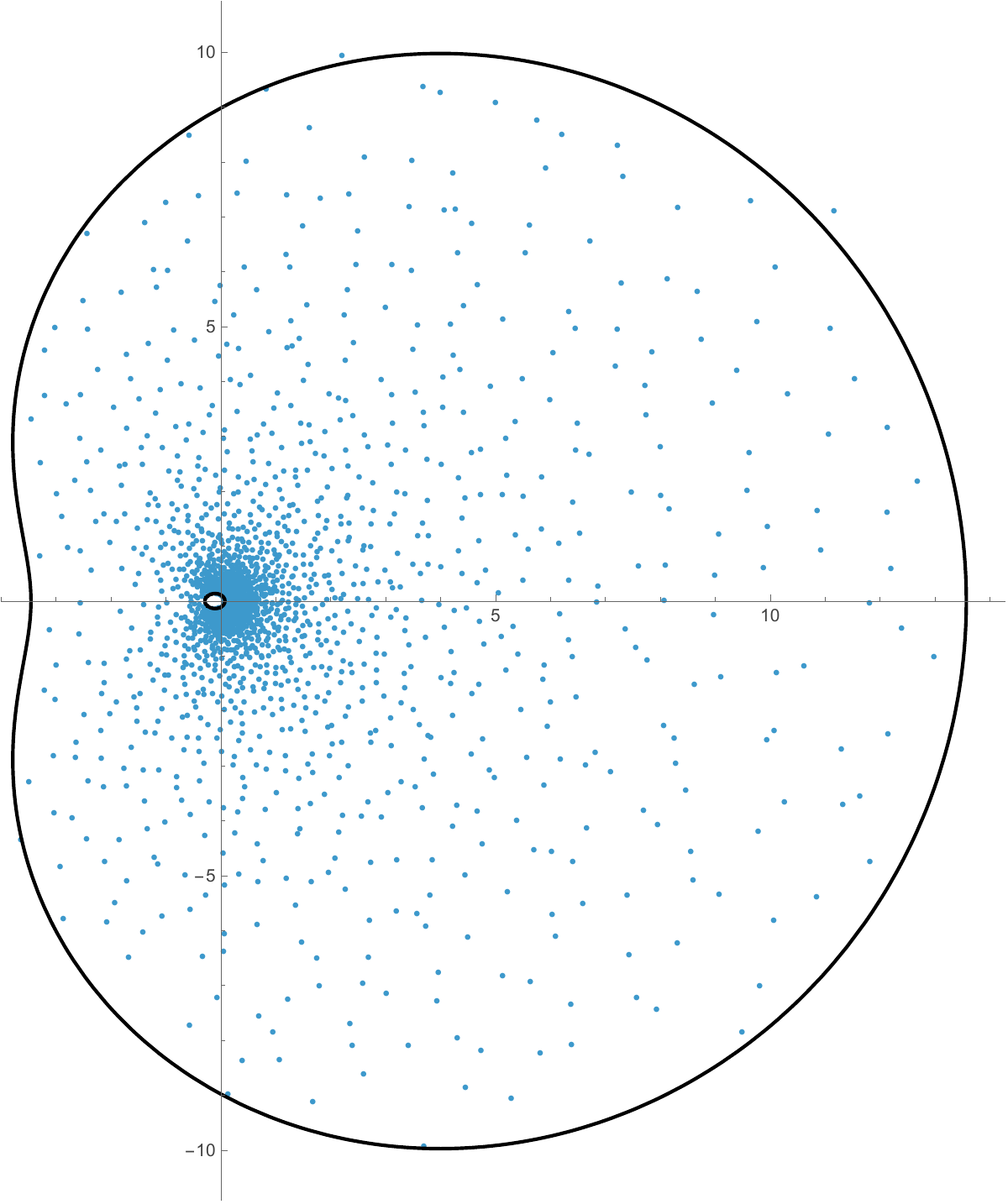}};

  \end{tikzpicture}
  \caption{Eigenvalues and limit support set of the Brownian motion $B(t)$ on $\mathrm{GL}(N,\C)$, with $N=2000$ and $t=3,4,4.5$ left to right.  These figures closely match the plots in \cite{Biane1997JFA} of the region Biane denoted $\Sigma_t$ (herein denoted $\Sigma(1,t)$, see Section \ref{sect.Brown(ian).meas}), instigating the work that has culminated in this paper.  \label{fig-evolving-lima-bean}}
\end{figure}

The framework to begin understanding the large-$N$ limit of the Brownian motion $(B(t))_{t\ge0}$ on $\mathrm{GL}(N,\C)$ was created independently in \cite{Cebron2013} and \cite{DHK2013}. The processes $(W(t))_{t\ge 0}$ and $(B(t))_{t\ge0}$, along with their Hermitian and unitary cousins, have Markov generators that are Laplacians on different spaces (see Section \ref{sect:BMLie} for details).  Careful analysis of the associated Markov semigroups  yields explicit covariance estimates, e.g.\ $\mathrm{Cov}[P(B(t),B(t)^\ast),Q(B(t),B(t)^\ast)] = \mathcal{O}_{t,P,Q}(\frac{1}{N^2})$ for any $t\ge 0$ and non-commutative polynomials $P,Q$. In \cite{Kemp2016,Kemp2017}, these results were generalized to multiple times and to a two-parameter family of diffusion processes on $\mathrm{GL}(N,\C)$, all Brownian motions with respect to invariant metrics, which includes the unitary Brownian motion as a boundary case.  Meanwhile, the original Segal--Bargmann--Hall transform was extended to allow for complexified time \cite{DHK2020}, and this gave rise to a third (and final, in a sense explained at the end of Section \ref{sect.intro.BMs} below) parameter for invariant metrics. Very recently \cite{BCC2025} upgraded these convergence results to this $3$-parameter ``elliptical'' setting, and moreover upgraded the convergence from {\em traces} to {\em norms} of polynomial functions (i.e.\ strong convergence).

All that being said, the spectrum is discontinuous in the topology of convergence in $\ast$-distribution.  For example \cite[Example 1.2]{Bordenave-Chafai-circular} the permutation matrix $Q$ of a full $N$-cycle and the unilateral shift $S$ on $\mathbb{C}^N$ have the same limit $\ast$-distribution (being bounded rank-$1$ perturbations of each other); but the spectrum of $Q$ consists of the $N$th roots of unity, while the spectrum of $S$ is $\{0\}$ for all $N$. In general, convergence of eigenvalues is much harder to establish than convergence in $\ast$-distribution.  For example, in the case of polynomials in independent Ginibre random matrices, the convergence in $\ast$-distribution has been known since the early nineties, see \cite{Voiculescu1991}, but the convergence of the empirical eigenvalue distribution remains a currently active problem -- see the 2022 paper \cite{Cook} for partial results, and the very recent preprint \cite{Yi-Han} for a proposed full solution.

The challenges presented by the gap between $\ast$-distribution and eigenvalue convergence left Biane's original conjecture about convergence of eigenvalues of Brownian motion on $\mathrm{GL}(N,\C)$ open.  Until now.  Our main Theorem \ref{thm.main} proves convergence of the density of eigenvalues for all invariant $\mathrm{GL}(N,\C)$ processes that deserve the title ``Brownian motion''.

To set the stage for this main result, we step back briefly to concretely describe Brownian motions on fairly general Lie groups.

\subsection{Brownian Motion on Lie Groups\label{sect:BMLie}}

Generally, a {\em Brownian motion} is any Feller--Markov process whose generator is ($\frac12$ of) a Laplace operator.  What is meant by a Laplace operator may depend on the precise geometric context; it is an elliptic second-order linear differential operator that resembles the classical Laplacian $\sum_j \partial_j^2$ on $\mathbb{R}^d$.  On manifolds, this form can be precisely mirrored by sums of squares of vector fields; on Lie groups, the group structure begs for the vector fields to be left- (or right-)invariant.

Let $\mathfrak{G}$ be a Lie group with Lie algebra $\mathfrak{g} = T_IG$, the tangent space at the identity $I\in \mathfrak{G}$.  For any vector $\xi\in\mathfrak{g}$, the vector field $\partial_\xi$ (often denoted $\tilde{\xi}$) is defined by $(\partial_\xi f)(A) = \left.\frac{d}{dt}f(A\exp(t\xi))\right|_{t=0}$ for $f\in C^\infty(\mathfrak{G})$.  Here, $\exp\colon\mathfrak{g}\to\mathfrak{G}$ is the exponential map which is a local diffeomorphism from a neighborhood of $0\in\mathfrak{g}$ onto a neighborhood of $I\in\mathfrak{G}$; if $\mathfrak{G}$ is a matrix Lie group, $\exp$ is simply the matrix exponential.  The vector field $\partial_{\xi}$ is left-invariant, meaning that $\partial_\xi (f\circ L_A) = (\partial_\xi f)\circ L_A$ where $L_A(B) = A^{-1}B$, and every left-invariant vector field on $\mathfrak{G}$ has this form (if $\tilde{X}$ is left-invariant then $\tilde{X} = \partial_{\xi}$ where $\xi = \tilde{X}|_I$).

Given a finite collection of vectors $\beta\subset\mathfrak{g}$, we may consider the associated sum-of-squares operator
\begin{equation} \label{eq.Laplacian.0} 
\Delta_\beta = \sum_{\xi\in\beta} \partial_{\xi}^2. \end{equation}
If $\beta$ is a vector space basis of $\mathfrak{g}$, then $\Delta_\beta$ is an elliptic operator. If $\mathfrak{G}$ is unimodular, $\Delta_\beta$ coincides with the the Laplace--Beltrami operator for the left-invariant Riemannian metric on $\mathfrak{G}$ induced by the inner product on $\mathfrak{g}$ for which $\beta$ is an orthonormal basis.  (It is then a standard calculus exercise that the operator $\Delta_\beta$ does not depend on which orthonormal basis $\beta$ is used.)

If $\beta$ is not a vector space basis for $\mathfrak{g}$, the operator $\Delta_\beta$ is degenerate, but it may still have nice properties.  H\"ormander's theorem \cite{Hormander1967} asserts that if $\beta$ generates $\mathfrak{g}$ as a Lie algebra (under vector space operations and Lie brackets) then $\Delta_\beta$ is {\em hypoelliptic}, meaning in particular that the associated heat kernel (the fundamental solution $u=u(t,\cdot)$ to the heat equation $\partial_t u =\frac12\Delta_\beta u$, which is the marginal distribution of the Brownian motion at time $t$) is a smooth function on $\mathfrak{G}$, strictly positive if $\mathfrak{G}$ is connected.  Feller--Markov processes with hypoelliptic generators retain many of the robust properties imbued by ellipticity (such as neighborhood recurrence and unique stationary distributions when $\mathfrak{G}$ is connected), and hence also deserve to be called Brownian motions.

The sum-of-squares form \eqref{eq.Laplacian.0} for the Laplacian allows for a simple description of Brownian motion in $\mathfrak{G}$ in terms of a stochastic differential equation driven by standard Brownian motion $W$ in (a subspace of) the Lie algebra $\mathfrak{g}$.  Specifically: for a given linearly independent set $\beta\subset\mathfrak{g}$, let $\{W_{\xi}\colon\xi\in\beta\}$ be independent standard Brownian motions on $\mathbb{R}$, and define
\begin{equation} \label{eq.W} W(t) = W_{\beta}(t):= \sum_{\xi\in\beta} W_{\xi}(t)\,\xi. \end{equation}
The process $W$ is determined by the inner product on $\mathrm{span}_{\R}\,\beta$ for which $\beta$ is orthonormal; if $\beta'$ is any other orthonormal basis with respect to this inner product, $W_\beta$ and $W_{\beta'}$ have the same law.  We therefore think of $W$ associated to an inner product on (a subspace of) $\mathfrak{g}$.

The diffusion process $B$ on $\mathfrak{G}$ whose generator is $\frac12\Delta_\beta$ (and is rightfully called a Brownian motion provided $\beta$ satisfies H\"ormander's condition) can equivalently be described as the strong solution to the Stratonovich SDE
\begin{equation} \label{eq.BM.Strat.SDE}
B(t) = B(0) + \int_0^t B(s)\circ dW(s).
\end{equation}
In the case that $\mathfrak{G}$ is a matrix Lie group, there is a simple conversion 
of \eqref{eq.BM.Strat.SDE} to an It\^o integral:
\begin{equation} \label{eq.BM.Ito.SDE}
B(t) = B(0) + \int_0^t B(s)\,dW(s) + \frac12\int_0^t B(s)\Xi\,ds
\end{equation}
where $\Xi = \sum_{\xi\in\beta}\xi^2$ \cite[p.\ 116]{McKean}.
(The It\^o correction term is $\frac12$ the quadtratic covariation of $B$ and $W$.)  If $\beta$ satisfies H\"ormander's condition (i.e.\ generates $\mathfrak{g}$ as a Lie algebra), the sum-of-squares matrix $\Xi$ is in the center of $\mathfrak{g}$.  Note that $\Xi$ can be computed directly from the given $\mathfrak{g}$-valued diffusion $W$ \eqref{eq.W} by
\begin{equation} \label{eq.Xi.EW2}
\mathbb{E}[W(t)^2] = \sum_{\xi_1,\xi_2\in\beta} \mathbb{E}[W_{\xi_1}(t)W_{\xi_2}(t)]\xi_1\xi_2 = \sum_{\xi\in\beta} t\xi^2 = t\Xi.
\end{equation}

If $B_I$ denotes the solution to \eqref{eq.BM.Ito.SDE} with $B(0)=I$ and if $B_0$ is any random matrix independent from the process $B_I$, the process $B=(B_0B_I(t))_{t\ge 0}$ satisfies \eqref{eq.BM.Ito.SDE} with $B(0)=B_0$; hence, Brownian motion with any given initial distribution $B_0$ can equivalently be constructed this way, and we generally restrict the SDE to $B(0)=I$. Similar considerations show that $B$ has {\em independent multiplicative increments}: for $0\le s<t$, $B(s)^{-1}B(t)$ is independent from $\{B(r)\colon r\le s\}$.  Moreover, the inverse is itself an It\^o process: a stochastic calculus exercise shows that the process $\Binv=(\Binv(t))_{t\ge0}$ defined by
\begin{equation} \label{eq.SDE.Binv} \Binv(t) = I - \int_0^t dW(s)\circ\Binv(s) = I - \int_0^t dW(s)\Binv(s) + \frac12\int_0^t \Xi\,\Binv(s)\,ds \end{equation}
is the inverse of $B(t)$ for each $t$; i.e. $B(t)\Binv(t) = I$.  That is: since $-W$ has the same law as $W$, the inverse $B(t)^{-1} = \Binv(t)$ is a {\em right}-invariant Brownian motion on $\mathfrak{G}$, driven by $W$.  See \cite[Section 4.2]{Kemp2016} for further discussion.

\subsection{Brownian Motions on $\mathrm{GL}(N,\C)$\label{sect.intro.BMs}}

Now, let us specialize to the case $\mathfrak{G} = \mathrm{GL}(N,\mathbb{C})$, in which case the Lie algebra is $\mathfrak{g} = \mathfrak{gl}(N,\mathbb{C}) = \M_N(\C)$, the algebra of all $N\times N$ complex matrices. The above constructions are determined by a choice of a linearly independent set $\beta\subset\M_N(\C)$,
or equivalently a choice of a possibly degenerate inner product on $\M_N(\C)$.  Consider the (rescaled) Pauli spin matrices
\begin{equation} \label{eq.Pauli} \begin{aligned} \mathscr{S}_N = \Big\{\textstyle{\frac{1}{\sqrt{N}}}E_{i,i}\colon 1\le i\le N\Big\} \; &\sqcup\; \Big\{\textstyle{\frac{1}{\sqrt{2N}}}(E_{i,j}+E_{j,i})\colon 1\le i<j\le N\Big\} \\
&\sqcup\; \Big\{\textstyle{\frac{\i}{\sqrt{2N}}}(E_{i,j}-E_{j,i})\colon 1\le i<j\le N\Big\} \end{aligned} \end{equation}
where $E_{i,j}\in\M_N(\C)$ has a $1$ in the $(i,j)$-entry and $0$ elsewhere.
The set $\mathscr{S}_N$ is orthonormal with respect to the scaled Hilbert--Schmidt inner product
\begin{equation} \label{eq.HS.innprod} 
\langle A,B\rangle_N := N\,\mathrm{Re}\,\mathrm{Tr}(B^\ast A) 
\end{equation}
where $\mathrm{Tr}(A) = \sum_{j=1}^N A_{jj}$ is the trace.  We denote the associated process \eqref{eq.W} as $W=X = (X(t))_{t\ge0}$, a {\bf Hermitian Brownian motion}.  The joint law of entries of $t^{-1/2}X$ is a {\bf Gaussian Unitary Ensemble} (GUE): $\{X_{ij}\colon i\le j\}$ are independent, with diagonal entries real Gaussians of variance $t/N$ and upper-triangular entries complex Gaussians with independent real and imaginary parts of variance $t/2N$.  This is the variance scaling that Biane considered in \cite{Biane1997b,Biane1997JFA}, ergo the $N^{-1/2}$ scaling of $\mathscr{S}_N$ and associated scaling up of the inner product in \eqref{eq.HS.innprod}.  See further discussion below in Definition \ref{2HBdef}.

\begin{rem}[Unitary case] \label{rem.UBM} The span $\i\mathscr{S}_N$ of skew-Hermitian matrices is the Lie algebra $\mathfrak{u}(N)$ of the compact Lie group $\mathrm{U}(N)$ of unitary matrices.  In particular, $\i\mathscr{S}_N$ does not satisfy H\"ormander's condition for all of $\M_N(\C)$.  The sum of squares $\Xi$ of matrices in $\i\mathscr{S}_N$ is equal to $-I$, and so in this context \eqref{eq.BM.Ito.SDE} becomes
\begin{equation} \label{eq.UBM.SDE} U(t) = I + \i\int_0^t U(s)\,dX(s) - \frac12\int_0^t U(s)\,ds \end{equation}
where we have renamed $B=U$ here as it is {\em Brownian motion on $\mathrm{U}(N)$}.  Note that $U(t)^\ast$ then satisfies \eqref{eq.SDE.Binv}, showing directly that $U(t)^\ast=U(t)^{-1}$, i.e.\ $U$ takes values in $\mathrm{U}(N)$.  Indeed, with the present scaling, $U$ is the unitary Brownian motion whose large-$N$ limit was studied by Biane \cite{Biane1997b}.
\end{rem}

In this paper, we consider the general class of Brownian motions on $\mathrm{GL}(N,\mathbb{C})$ driven by {\bf elliptical Brownian motions} on $\M_N(\C)$: i.e.
\begin{equation} \label{eq.W.elliptical} 
W = e^{\i\theta}(aX+\i bY) 
\end{equation}
where $X$ and $Y$ are independent Hermitian Brownian motions, $a,b\in\mathbb{R}$ (not both $0$), and $\theta\in[-\pi,\pi)$.  If $ab\ne 0$, $e^{-\i\theta}W$ corresponds to \eqref{eq.W} with a spanning set consisting of positive scalar multiples of all the matrices in $\mathscr{S}_N\sqcup\i\mathscr{S}_N$, and so the associated Laplacian is elliptic; it then follows that $W$ has an elliptic generator as well.  If either $a=0$ or $b=0$, the Laplacian is degenerate.  However, straightforward computations show that $\i\mathscr{S}_N\subset [\mathscr{S}_N,\mathscr{S}_N]$, and this shows that the H\"ormander condition holds {\em except} in the cases that $W$ reduces to a real multiple of $\i X$ or $\i Y$ --- which yield the unitary Brownian motion (Remark \ref{rem.UBM}).

It is useful to make the following change of variables:
\begin{equation} \label{eq.rho.zeta}
\begin{aligned} \rho: &= \mathbb{E}\mathrm{tr}[|W(1)|^2] = a^2+b^2 \\
\zeta: &= \mathbb{E}\mathrm{tr}[W(1)^2] = e^{2\i\theta}(a^2-b^2)
\end{aligned}
\end{equation}
where $\mathrm{tr} = \frac{1}{N}\mathrm{Tr}$ is the normalized trace on $\M_N(\C)$ and, for any matrix $A$, $|A|^2:=A^\ast A$.  The parameters $(\rho,\zeta)$ satisfy $|\zeta|\le\rho$.  Conversely, for any pair $(\rho,\zeta)\in\R_+\times\C$ with $|\zeta|\le\rho$, all the elliptical Brownian motions \eqref{eq.W} with $(a,b,\theta)$ satisfying \eqref{eq.rho.zeta} have the same law.  Ergo the pairs $(\rho,\zeta)$ parametrize all elliptical Brownian motions (in law); we henceforth denote $W=W_{\rho,\zeta}$.  Note that:
\begin{itemize}
    \item $W_{1,0}=\frac{1}{\sqrt{2}}(X+\i Y)$ is the full {\bf Ginibre Brownian motion}.  All its entries are i.i.d.\ complex Brownian motions.  When $\zeta=0$, the process is invariant under left (or right) multiplication by any deterministic unitary matrix; in particular, multiplying by $e^{\i\theta}$ does not change the distribution in this case.
    \item The ellipticity condition $ab\ne 0$ exactly corresponds to $|\zeta|<\rho$.
    \item When $|\zeta|=\rho$, $W_{\rho,\zeta}$ is hypoelliptic except in the particular case $\zeta = -\rho$ which yields time rescaled Unitary Brownian motion -- see Remark \ref{rem.UBM}.
\end{itemize}
See Section \ref{sect.background.BMs} for more detail on the relationship between $(a,b,\theta)$ and $(\rho,\zeta)$.

We can now fully describe the group Brownian motions of \eqref{eq.BM.Ito.SDE} corresponding to the elliptical driving processes $W_{\rho,\zeta}$.  Since $\Xi$ is in the center of $\mathfrak{gl}(N,\mathbb{C}) = \M_N(\C)$ which consists of scalar multiples of $I$, it follows from \eqref{eq.Xi.EW2} and \eqref{eq.rho.zeta} that $\Xi = \mathbb{E}[W(1)^2] = \zeta I$.  Hence, the main objects of study in the present work are Brownian motions $B_0 B_{\rho,\zeta}$ where
\begin{equation} \label{eq.B.SDE.rho.zeta}
B_{\rho,\zeta}(t) = I + \int_0^t B_{\rho,\zeta}(s)\,dW_{\rho,\zeta}(s) + \frac{\zeta}{2}\int_0^t B_{\rho,\zeta}(s)\,ds. \end{equation}
The standard Brownian motion on $\mathrm{GL}(N,\mathbb{C})$ is the specal case $B_{1,0}$, driven by the Ginibre Brownian motion $W_{1,0}$ on $\mathfrak{gl}(N,\mathbb{C})$.

\begin{figure}[h!]
\centering
  \begin{tikzpicture}

    \node[inner sep=2pt] at (0cm, 0cm)
      {\includegraphics[height=5cm]{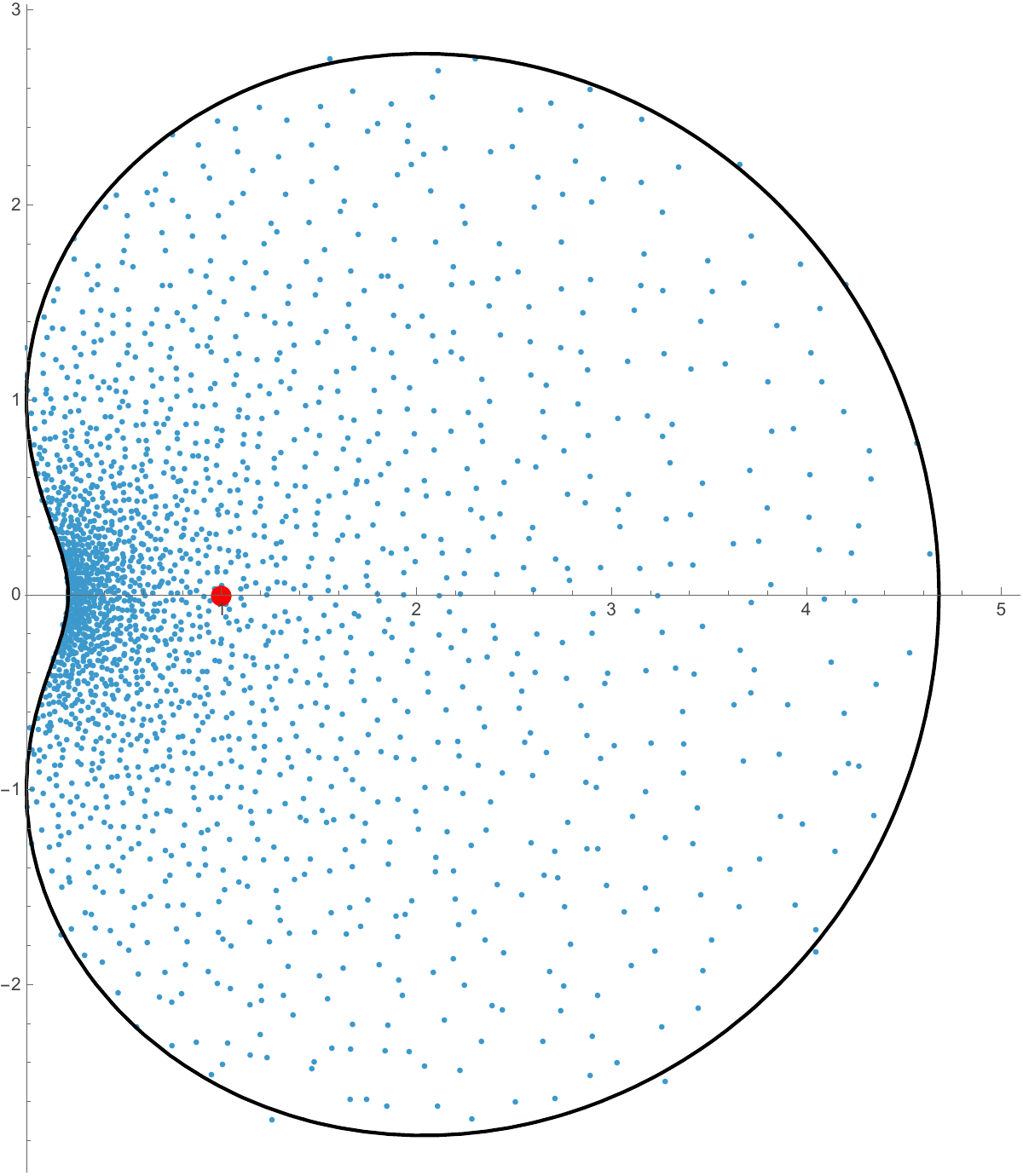}};

    \node[inner sep=2pt] at (5.5cm, -0.3cm)   
      {\includegraphics[width=5cm]{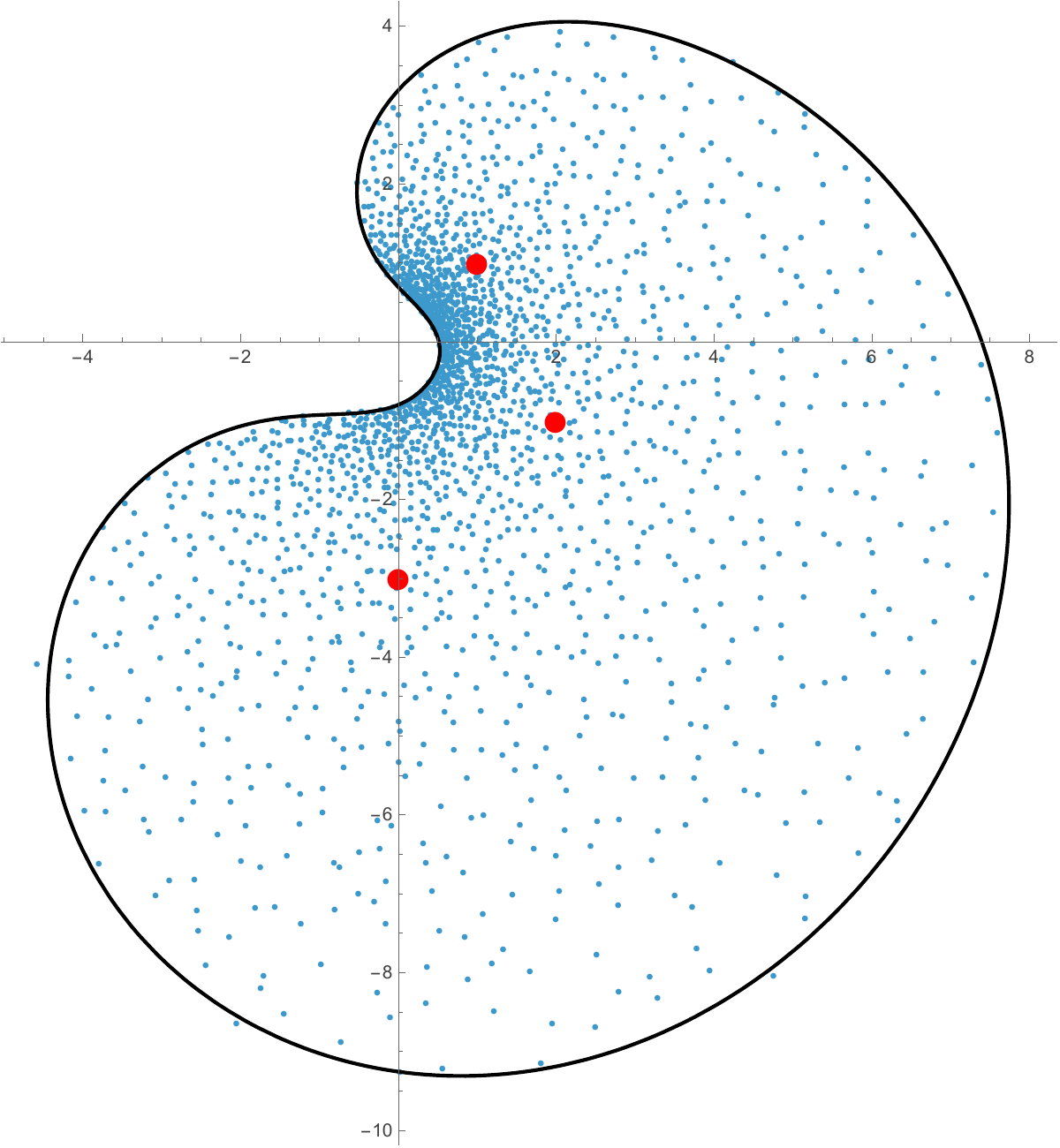}};

    \node[inner sep=2pt] at (0cm, -5.5cm)
      {\includegraphics[height=5cm]{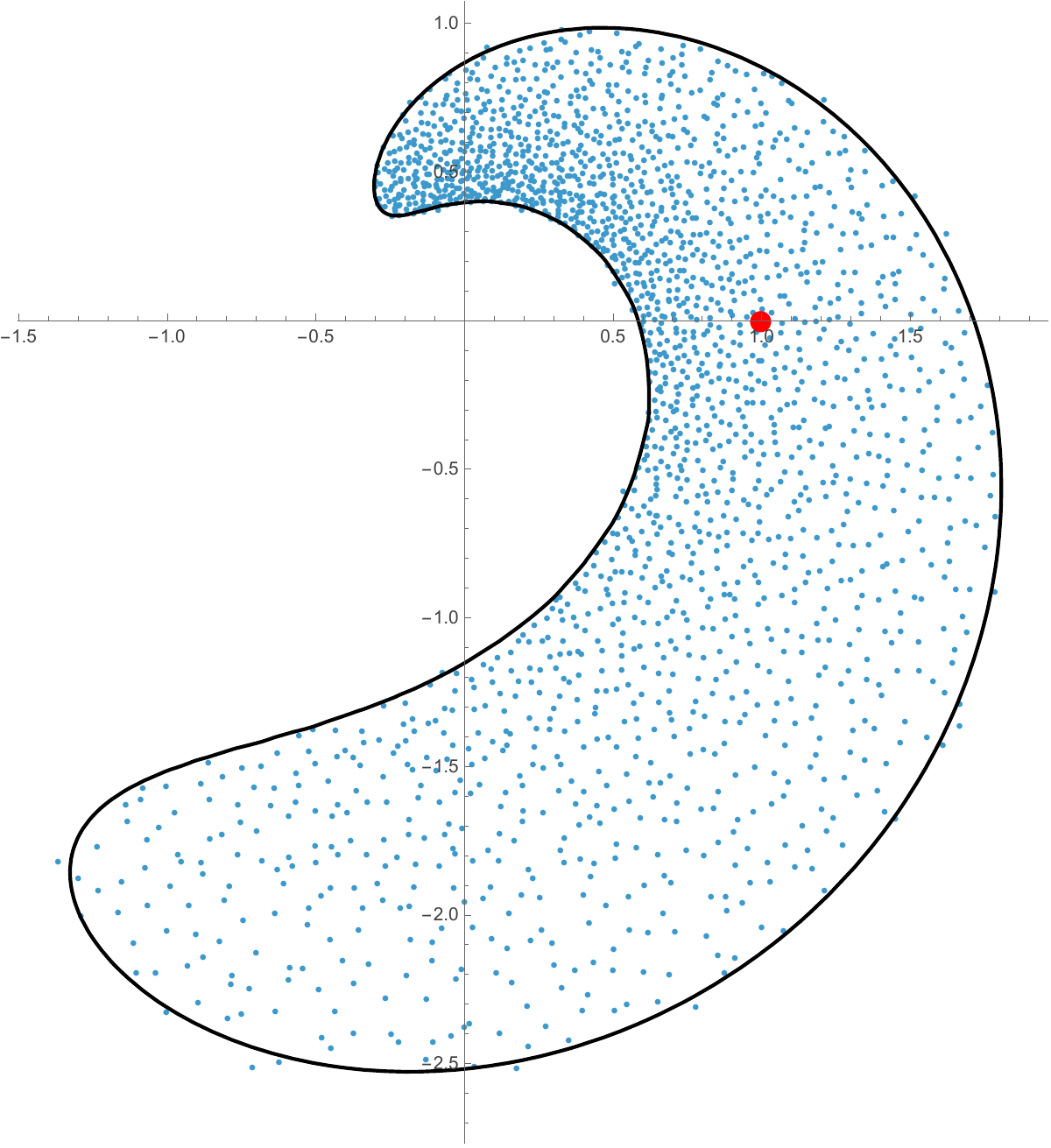}};

    \node[inner sep=2pt] at (5.5cm, -5.22cm)  
      {\includegraphics[width=5cm]{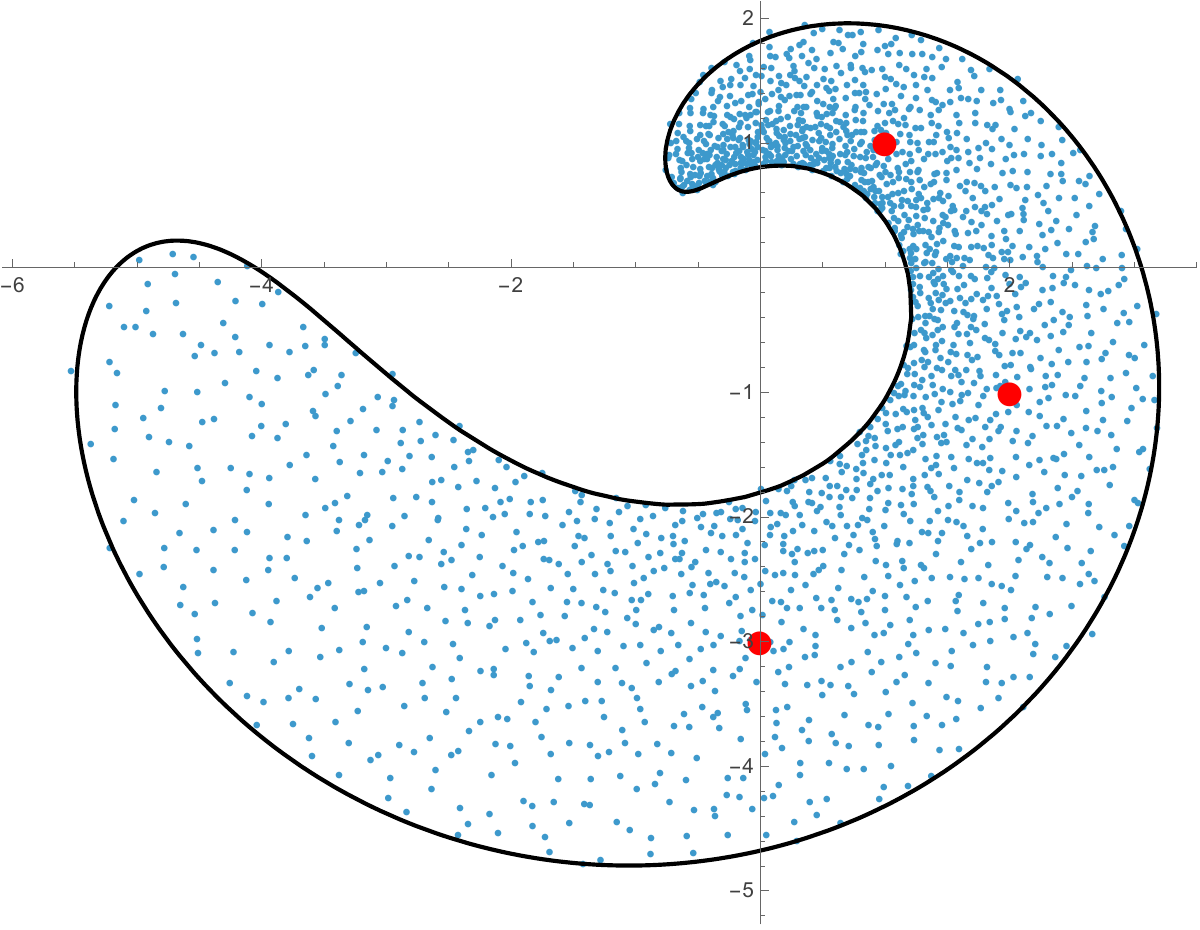}};
  \end{tikzpicture}
  \caption{Eigenvalues of $B_0 B_{\rho,\zeta}(1)$ with $N=2000$.  On the left $B_0 = I$, while on the right $B_0$ is a normal matrix with eigenvalues $\{2-\i,1+\i,-3\i\}$ weighted equally. In the top row $(\rho,\zeta)=(2,0)$; in the bottom row $(\rho,\zeta)=(2,-1-\i)$.
  \label{fig-grid-1vs3pt-slanted-vs-not}}
\end{figure}

\begin{figure}[h!]
\centering
  \begin{tikzpicture}

    \node[inner sep=2pt] at (0cm, -6cm)
      {\includegraphics[width=5cm]{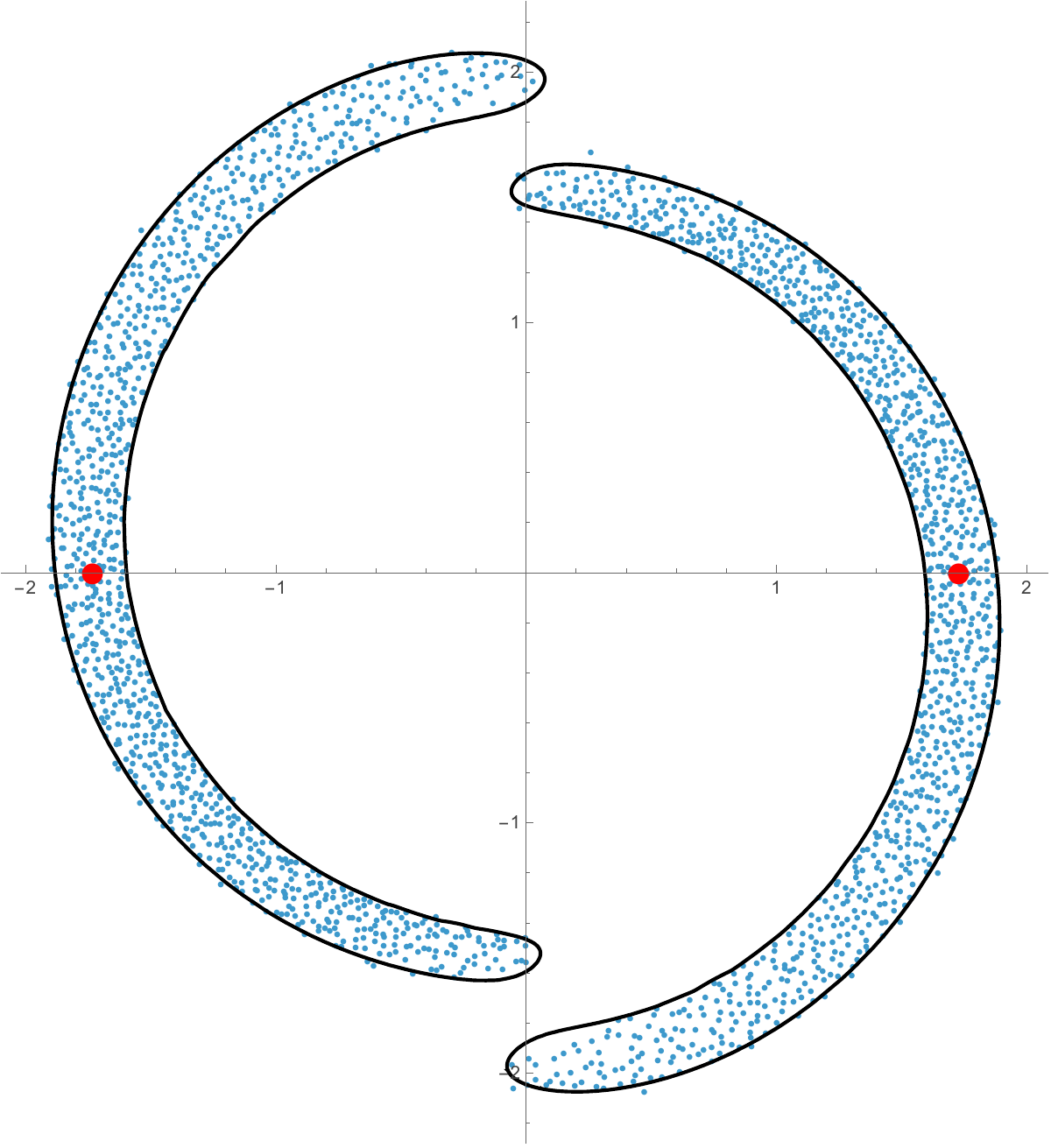}};

    \node[inner sep=2pt] at (5.5cm, -6cm)  
      {\includegraphics[width=5cm]{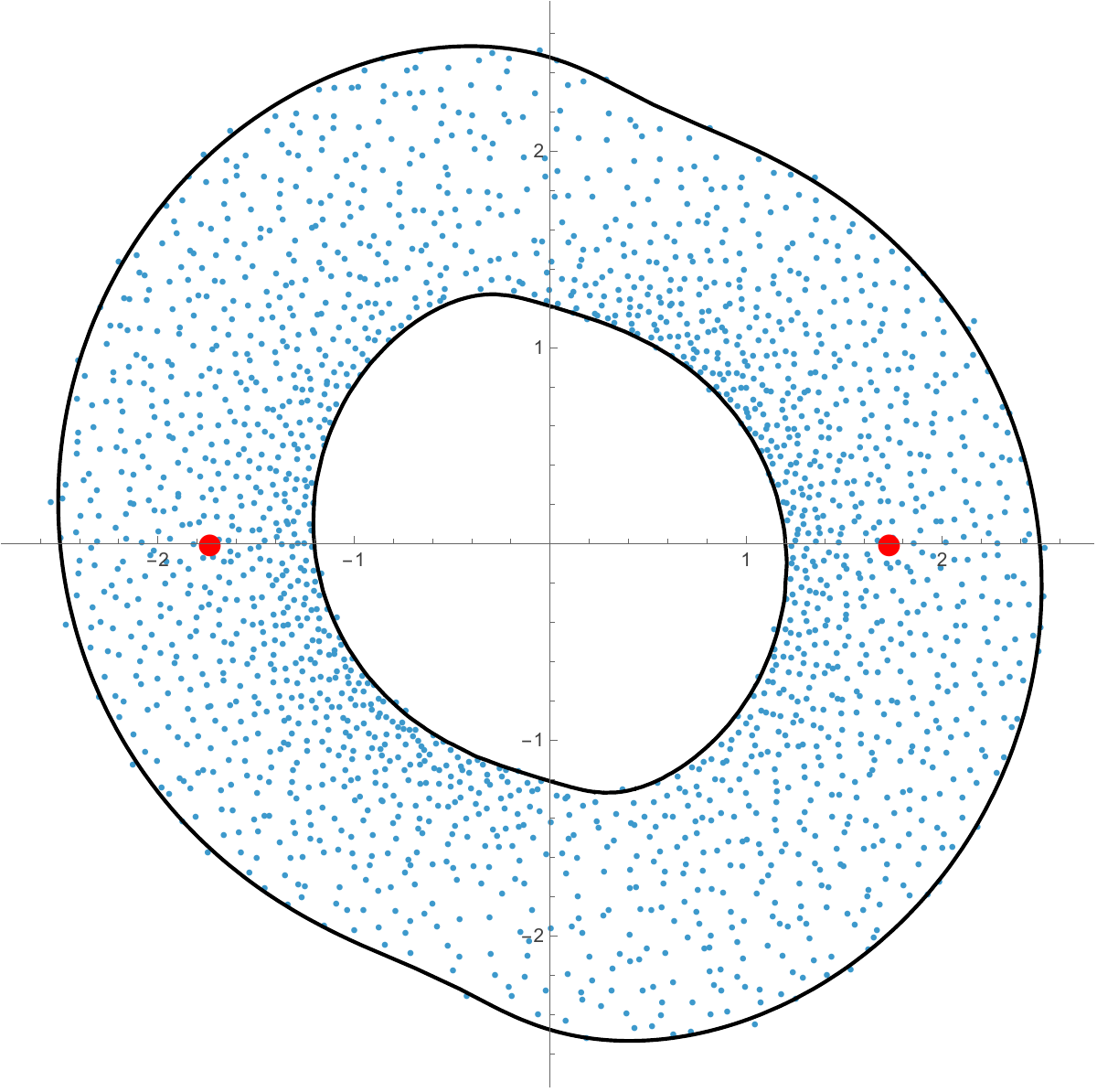}};
  \end{tikzpicture}
  \caption{Eigenvalues of $B_0B_{\rho,\zeta}(1)$ with $N=2000$ and $(\rho,\zeta)=(1.9,-1.7-0.6\i)$.  The initial condition has eigenvalues $\pm\sqrt{3}$ in both cases; on the left $B_0$ is normal, while on the right \break
  $ B_0 = \begin{bmatrix} 0 & 1 \\ 3 & 0 \end{bmatrix} \otimes I_{N/2}. $
}
\end{figure}

Note that all the elliptical Brownian motions \eqref{eq.W.elliptical} are invariant under unitary conjugation: if $U\in\mathrm{U}(N)$ is a deterministic unitary matrix then $UW_{\rho,\zeta}U^\ast$ is equal in distribution to $W_{\rho,\zeta}$.  It follows from \eqref{eq.B.SDE.rho.zeta} that the Brownian motions $B_{\rho,\zeta}$ on $\mathrm{GL}(N,\C)$ are similarly invariant; see \eqref{eq.B.conj.inv}.  Unitary conjugation invariance is a historically useful property for connecting random matrices to large-$N$ limits with free probability, see Proposition \ref{prop.Haar.conj}.  In \cite[Theorem 3.3, Proposition 3.7]{DHK2020}, the authors characterized all inner products, and hence all Brownian motion processes $W$ \eqref{eq.W}, on $\M_N(\C)$ that are $\mathrm{U}(N)$-conjugation invariant; they are all of the form $W_{\rho,\zeta}+cZ I$ where $c\in\R$ and $Z$ is a (scalar) complex Brownian motion independent from $W_{\rho,\zeta}$.  (The proof assumes strict positivity of the inner product, but a simple modification shows that the characterization also applies to the degenerate case $|\zeta|=\rho$.) Since $Z I$ commutes with $\M_N(\C)$, it follows that the most general $\mathrm{U}(N)$-conjugation invariant Brownian motions on $\mathrm{GL}(N,\C)$ have the form $e^{cZ(t)}B_{\rho,\zeta}(t)$. As the spectral distribution of $e^{cZ(t)}B_{\rho,\zeta}(t)$ is obtained from that of $B_{\rho,\zeta}(t)$ by multiplying each eigenvalue by the independent scalar random variable $e^{cZ(t)}$, we henceforth treat only the case $c=0$.

\subsection{Main Result}

Given a matrix $A\in\M_N(\C)$, denote its $N$ complex eigenvalues by $\{\lambda_j(A)\}_{j=1}^N$ (counted with multiplicity and labeled arbitrarily).  
The {\bf empirical law of eigenvalues} of $A$, denoted $\mu_A$, is the random probability distribution with equal masses at all eigenvalues of $A$:
\begin{equation} \label{eq.ESD.def} \mu_A = \frac1N\sum_{j=1}^N\delta_{\lambda_j(A)}. \end{equation}
Our main theorem asserts that, almost surely, the empirical law of eigenvalues $\mu_{B(t)}$ converges weakly to a deterministic measure as $N\to\infty$.  In fact, we prove this more generally for the Brownian motion $B_0 B(t)$ for any initial distribution $B_0$ that has a large-$N$ limit of its own.

\begin{theorem} \label{thm.main}  Let $\rho>0$, $\zeta\in\mathbb{C}$ with $|\zeta|\leq \rho$, and let $B = B_{\rho,\zeta}$ denote a Brownian motion on $\mathrm{GL}(N,\mathbb{C})$, see \eqref{eq.B.SDE.rho.zeta}.  Let $B_0$ be a random matrix in $\mathrm{GL}(N,\mathbb{C})$ independent from the process $B$ and suppose that $B_0$ converges a.s.\ in $\ast$-distribution (see Definition \ref{def.*-dist}) to an operator $b_0$ as $N\to\infty$.

Then for fixed $t>0$, a.s.\ as $N\to\infty$, the empirical law $\mu_{B_0 B(t)}$ of eigenvalues of the Brownian motion $B_0 B(t)$ converges weakly to a deterministic measure $\mu_{b_0,t}$ (see Definition \ref{def:mub0t} and Section \ref{sect.Brown(ian).meas} below).
\end{theorem}

If $B_0$ is a normal matrix, the $\ast$-distribution convergence criterion simply means that the empirical eigenvalue distribution of $B_0$ converges weakly a.s.\ to some deterministic probability measure on $\C$. 

When $B_0$ is not normal, as discussed in the last two paragraphs before Subsection \ref{sect:BMLie}, convergence in $\ast$-distribution does not generally imply convergence of eigenvalues.  Thus, our theorem implies that multiplying a matrix by a multiplicative Brownian  (at any time $t>0$) has a regularizing effect on its empirical measure which simplifies considerably the proof of its convergence. Theorem \ref{thm.main} thus provides a multiplicative analogue of a well-known result of Sniady \cite{Sniady}; see Section \ref{sec:future} for further discussion.

\subsection{Ideas of the Proof}
\label{sec:ideas}

In this work we develop a robust approach to convergence of the spectrum of non-normal matrix diffusions based on strong quantitative approximation by multiplicative random walks. Potential future applications of this strategy are discussed in Section \ref{sec:future}.

As in most works on the spectrum of non-normal random matrices, we apply
Girko's Hermitization framework, reviewed in Section \ref{sect.Brown.meas}, which reduces our task to establishing sufficient control on small-to-mesoscopic scale singular values of $B_0 B(t)-zI$, for a.e.\ $z\in\mathbb{C}$, to guarantee convergence of the log potential. This is often the most challenging step for proofs of convergence of the spectrum of non-normal matrices, and amounts to proving stability of the spectrum under small perturbations. 
Recall that the singular values of an $N\times N$ matrix $A$ are the square roots of the eigenvalues of $A^\ast A$, or equivalently the eigenvalues of $|A|=\sqrt{A^\ast A}$; they are denoted
\[ \|A\|=\sigma_1(A)\ge\cdots\ge\sigma_N(A) \ge0. \]
Note that 
$\sigma_N(A)= \|A^{-1}\|^{-1}$ if $A$ is invertible and $\sigma_N(A)=0$ otherwise, so bounding small singular values of $B_0B(t)-zI$ from below is equivalent to bounding large singular values of the resolvent $(B_0B(t)-zI)^{-1}$ from above.

Perhaps the most well-known quantification of spectral (in)stability of a matrix is by its \emph{pseudospectrum}. Recall that the $\eps$-pseudospectrum of $A\in\M_N(\C)$ is the set of points $z\in\C$ at which the resolvent has norm at least $1/\eps$, or equivalently $\sigma_N(A-zI)\le \eps$. We refer to the beautiful text \cite{Trefethen} for a comprehensive review of pseudospectrum and its relevance for numerical analysis. Implementing Girko's framework requires showing that (almost) any $z\in \C$ is unlikely to be in the $\eps$-pseudospectrum of $B_0B(t)$ for $\eps=\eps_N$ not shrinking too fast (a polynomial decay will suffice); but further requires strong lower bounds on the $k$th smallest singular value of $B_0B(t)$ for all $k=o(N)$.

The key tool we develop
for these tasks is a small-time affine approximation of the process $B$, given in Theorem \ref{thm.2}, which in turn is established using an approximation of the process $B(t)$ by a discrete-time multiplicative random walk.

\subsubsection{Random walk approximation}

As shown by Berger in \cite{Berger}, Brownian motion $B$ \eqref{eq.BM.Ito.SDE} on a matrix Lie group is well approximated by a Donsker-like multiplicative random walk:
\[ B(t) \approx B_n(t):= \prod_{i=1}^{\lfloor t2^n\rfloor} \left(I+W(\textstyle{\frac{i}{2^n}})-W(\textstyle{\frac{i-1}{2^n}}) + \frac{1}{2^{n+1}}\Xi\right). \]  
We prove (in the case of interest, $B=B_{\rho,\zeta}$ where $\Xi = \zeta I$) a quantitative $L^p$ version of Berger's theorem: for any fixed $t\ge 0$, 
\begin{equation}
    \label{intro.BBn-Lp}
    \|B(t)-B_n(t)\|_{p}:= ( \E \ts (|B(t)-B_n(t)|^p))^{1/p} \le C_{n,t,p}2^{-n/2p}
\end{equation}
where the constant $C_{n,t,p}$ is independent of $N$ and locally uniform in $t$ (see Proposition \ref{prop:Lp.conv}). Here, the trace is normalized to take value 1 at the identity, and for a matrix $A$ we write $|A|:= (A^*A)^{1/2}$. 
The main step for \eqref{intro.BBn-Lp} is to show $B_n(t)$ is Cauchy in $L^2$, which we do using Gaussian integration by parts (also known in this context as the Schwinger--Dyson equation -- see Proposition \ref{prop:SDeq}).

\subsubsection{Affine approximation}

The following is the core technical result of the paper, and our main tool for controlling small singular values of scalar shifts $B(t)-zI$ of the process $B$. It shows that over a small time increment, the multiplicative process $B$ is well approximated by an additive Gaussian elliptical perturbation of the identity.

\begin{theorem}[Small-time affine approximation] \label{thm.2}
For any finite $T$ there exists $C=C_{\rho,T}<\infty$ so that, for all $t\in[0,T]$ and $\delta>0$, 
\[ 
\P\left(\left\|B(t)-I-W(t)\right\|\ge \delta \right) \le \left(\frac{Ct}{\delta}\right)^{N^{2/3}}. 
\]
\end{theorem}

In other words, we have $B(t) = I + W(t) + E(t)$, for an error $E(t)$ bounded in operator norm by $\mathcal{O}(t)$ with overwhelming probability for any fixed $t\le 1$, say. 
For comparison, note that $W(t)$ typically has norm of order $\sqrt{t}$. For the purposes of proving Theorem \ref{thm.main} a bound $\|E(t)\|=\mathcal{O}(t^{1/2+c})$ would suffice for any positive constant $c$, but Theorem \ref{thm.2} captures the correct order in $t$ of the norm (as can be seen by computing the expected squared Frobenius norm of $B(t)-I-W(t)$). Note that from \eqref{eq.BM.Strat.SDE} it is intuitive that $B(t)\approx I+W(t)$ for small $t$; the point of Theorem \ref{thm.2} is get an effective bound that is independent of $N$.

By the moment method, to establish Theorem \ref{thm.2} it suffices to show 
\begin{equation}
    \label{intro-goal1}
    \mathbb{E}\left[\mathrm{tr}\left(|B(t)-I-W(t)|^{2k}\right)\right] \le \mathcal{O}(t)^{2k}
\end{equation}
for $k$ of order $N^{2/3}$, where the implicit constant depends only on $\rho,T$. From the $L^p$ approximation \eqref{intro.BBn-Lp} with $p=2k$ we can replace $B(t)$ with $B_n(t)$ in \eqref{intro-goal1}. As $B_n(t)$ is a polynomial in the increments $W(\frac i{2^n})-W(\frac{i-1}{2^n})$, which in turn are combinations of GUE matrices (see \eqref{eq.W.elliptical}), the problem is reduced to bounding
\begin{equation}
    \label{intro-goal2}
    \E \ts Q(\bs{X}^N, \bs{Y}^N) \le \mathcal{O}(t)^{2k}
\end{equation}
where $\bs{X}^N=(X^N_i)_{i=1}^{\lfloor t2^n\rfloor}, \bs{Y}^N=(Y^N_i)_{i=1}^{\lfloor t 2^n\rfloor}$ are independent sequences of independent GUE matrices, and $Q$ is a non-commutative polynomial of degree $2k\times \lfloor t2^n\rfloor$. 

To bound the left hand side of \eqref{intro-goal2}, we use a quantitative expansion of the expected trace in powers of $\frac{1}{N^2}$:
\begin{equation}
    \label{intro-expand}
    \E \ts Q(\bs{X}^N,\bs{Y}^N) = \sum_{q=0}^{r-1} \frac{\alpha_q}{N^{2q}} + \frac{\tilde \alpha_r^N }{N^{2r}}.
\end{equation}
In the case that $Q$ is a monomial, this expansion coincides with the {\em genus expansion} which appeared first in \cite{GenusExpansion1}.  In that case, $\alpha_q$ is a positive integer, and represents a count of certain colored maps on a surface of genus $2q$; see \cite{Guionnet-maps} for more general development.  This combinatorial description is not very useful for quantitative computations: beyond $r=1$ there are are no known formulas or even reasonable bounds.  Recently, a more analytic approach to the expansion was introduced in \cite{CGP2022} (in the case $r=1$) and fully developed in \cite{Parraud2023}.  The core idea is to work in a space where the GUE matrices $\bs{X}^N,\bs{Y}^N$ and their free probability limit operators $\bs{x},\bs{y}$ coexist, and then interpolate between them using Ornstein--Uhlenbeck type processes. By differentiating along two subsequent interpolations, using the Schwinger--Dyson equations and Gaussian integration by parts, the $r=1$ term appears; the procedure can be iterated $r$ times to achieve the full expansion.  See Section \ref{sect.Schwinger-Dyson} for a more detailed explanation of the method.

Following this approach, the coefficients $\alpha_q$ have explicit (though complicated) expressions in terms of iterated applications of a certain differential operator $L$ (the composition of the generators of the two interpolations) to the polynomial $Q$, evaluated at the limit semicircular variables; the remainder coefficient $\tilde\alpha_r^N$ is similar, involving an extended system of GUE matrices.  See Lemma \ref{3apparition} for the full technical result, which allows for effective estimation.  In our setting, we will use Lemma \ref{3apparition} to show $\E \tr Q(X^N,Y^N) = \mathcal{O}(t)^{2k} \exp(\mathcal{O}(k^4/N^2))$ with implicit constants depending only on $\rho,T$, giving the desired estimate \eqref{intro-goal2} for $k\sim N^{2/3}$. 

We note that a different matrix-to-free interpolation was also used in the recent paper \cite{BCC2025} to study $(\rho,\zeta)$-Brownian motions; this approach does not allow simultaneous consideration of $B(t)$ and $W(t)$ and therefore does not suit our needs. Moreover, the overarching strategy used in \cite{BCC2025} (i.e.\ studying the expectation of traces of a smooth function evaluated in random matrices to deduce $\mathcal{O}(\frac{1}{N^2})$ fluctuations), while effective to prove strong convergence, is not sufficiently sharp to obtain  the fine-scale concentration estimates on the norm that
are needed for our main results. We believe our strategy of expressing $B(t)$ and $W(t)$ jointly in terms of small increments of the driving process $W$ provides a general approach for analyzing the spectrum of matrix SDE -- see Section \ref{sec:future} for further discussion.

\subsubsection{Application to small singular values}

To control the small singular values of scalar shifts $B_0B(t)-zI$ we
use Theorem \ref{thm.2} in conjunction with the following observation.  If $B$ and $\widetilde{B}$ are independent $(\rho,\zeta)$-Brownian motions then, for any $t>0$ and  $\varepsilon\in(0,t)$, $B(t-\varepsilon)\widetilde{B}(\varepsilon)$ has the same law as $B(t)$. 
From
Theorem \ref{thm.2},
\begin{equation}    \label{Ceps}
    \widetilde{B}(\varepsilon)\approx I+W(\varepsilon)
\end{equation}
for small $\varepsilon$, and hence $B(t)\approx B(t-\varepsilon)(I+W(\varepsilon))$. The robust behavior of singular values of  the elliptical Ginibre matrix $W(\varepsilon)$ can thence be parlayed into good behavior of singular values for $B_0B(t)-zI$. We note that the basic strategy of exploiting more tractable randomness in the Lie algebra was brilliantly used in \cite{RudelsonVershynin2014} to bound the smallest singular value of Haar unitary/orthogonal perturbations of arbitrary matrices. In that argument the analogue of the step \eqref{Ceps} is elementary, whereas here we rely on Theorem \ref{thm.2}, the key technical innovation of this work.

Apart from a lower bound on the smallest singular value $\sigma_N(B_0B(t)-zI)$, we need lower bounds on ``mesoscopic'' singular values $\sigma_{N-k}(B_0B(t)-zI)$ for $k\in[N^{1-c}, \kappa N]$ for some small fixed $c,\kappa>0$. 
The most common approach in the literature is to deduce these from upper bounds on the Stieltjes transform of the empirical singular value distribution $\mu_{|B_0B(t)-zI|}$; see for instance \cite{Bai1997,HaagerupT2005Annals,TaVu-circ1,GuionnetKZ-single-ring,Cook}. Indeed, this was the route taken in \cite{GuionnetKZ-single-ring} for the proof of the Single Ring Theorem, where bounds on the Stieltjes transform were obtained via a quantitative analysis of the Schwinger--Dyson loop equations. Here we take a more flexible approach pioneered by Tao and Vu in their proof of the Circular Law under the optimal finite second moment hypothesis \cite{TaoVu2010-aop}. The argument makes use of the \emph{inverse second moment identity} (see \eqref{inv2mom}) to pass to a problem on the concentration of distances of columns to the span of remaining columns for arbitrary shifts $M+W(\eps)$ of the elliptical Gaussian matrix $W(\eps)$. While such estimates are standard for the case that $W$ is real or complex Ginibre, or $W$ is GO/UE and $M$ is Hermitian (see Appendix \ref{app:small} for review of the literature), the case that $W$ is GUE and $M$ is arbitrary (and possibly non-normal) appears to be new -- see Lemmas \ref{lem:ssv.GUE}, \ref{lem:meso.app}.

We use this framework to prove Propositions \ref{prop:smallest} and \ref{prop:meso}, which give quantitative 
lower bounds on the singular values of arbitrary shifts of $B(t)$, holding with high probability.

In Section \ref{sect.analytic.approach} we provide an alternative argument to bound mesoscopic singular values by the Stieltjes transform approach, which may be of independent interest. Using the interpolation method of Section \ref{section.concentration} we can pass to the Stieltjes transform of the $*$-distribution limit $|b(t)-z|$ (see Section \ref{sect:fmbm}) to which we can apply the delicate SDE and PDE methods developed in \cite{DHKBrown}, whereby regularity properties for the Stieltjes transform are obtained from analysis of characteristic curves for a Hamilton--Jacobi-type PDE for a regularization of the log-potential. 
The resulting Corollary \ref{cor:wegner} is incomparable to Proposition \ref{prop:meso}: the former controls $\sigma_{N-k}(B(t)-zI)$ for $k$ of order at least $N^{9/11}$, whereas the latter permits $k$ of order $\log N$. On the other hand, Corollary \ref{cor:wegner} gives lower bounds of the correct shape $\sigma_{N-k}(B(t)-zI) \gs k/N$, 
whereas the bound $\gs (k/N)^2$ of Proposition \ref{prop:meso}, while sufficient to establish Theorem \ref{thm.main}, is not sharp. 

Finally, in Section \ref{sect.final} we use  our regularity estimates on small and mesoscopic shifted singular values to prove the requisite convergence in Theorem \ref{thm.main}; see also Section \ref{sect.Brown.meas} for a bigger picture discussion of the Hermitization method.

\subsection{Free Multiplicative Brownian Motions\label{sect:fmbm}}

To fully describe the deterministic limit density of eigenvalues in Theorem \ref{thm.main}, we need some constructs and tools from free probability theory; see Section \ref{sect.freeprob} for a brief primer and references.

The Brownian motion $B = B_{\rho,\zeta}$ \eqref{eq.B.SDE.rho.zeta} on $\mathrm{GL}(N,\mathbb{C})$ is driven by the elliptical Brownian motion $W = W_{\rho,\zeta}$ \eqref{eq.W.elliptical}, which is constructed from a pair $X,Y$ of independent Hermitian Brownian motions in $\M_N(\C)$ (Definition \ref{2HBdef}).  These processes have a large-$N$ limit.  Let $(\A,\tau)$ be a $W^\ast$-probability space that supports  two freely independent free semicircular Brownian motions $x,y$ (Definitions \ref{2freeprob} and \ref{smdvskjvsdhg}).  Voiculescu showed \cite{Voiculescu1994} that $\left(X,Y\right)$ converges a.s.\ in finite dimensional $\ast$-distributions to $(x,y)$ as $N\to\infty$.  In particular,
\begin{equation} \label{eq.free.elliptic}
w_{\rho,\zeta}= e^{\i\theta}(ax + \i by)
\end{equation}
is a free stochastic process in $\A$ that is the a.s.\ large-$N$ limit in finite dimensional $\ast$-distributions of $W_{\rho,\zeta} = e^{\i\theta}(aX + \i bY)$; here $(\rho,\zeta) \leftrightsquigarrow (\theta,a,b)$ as in \eqref{eq.rho.zeta}.  The process $w_{\rho,\zeta}$ is a non-commutative martingale \cite{JKN2025},so stochastic integrals and quadratic covariation may be treated analogously to the classical setting.  In particular, the free It\^o equation
\begin{equation} \label{eq.b.fSDE} b_{\rho,\zeta}(t) = 1 + \int_0^t b_{\rho,\zeta}(s)\,dw_{\rho,\zeta}(s) + \frac{\zeta}{2}\int_0^t b_{\rho,\zeta}(s)\,ds
\end{equation}
has a unique solution in $\A$ for all $t\ge 0$, see \cite{BianeSpeicher1998,Kargin2011}.  The original free multiplicative Brownian motion introduced by Biane in \cite{Biane1997b,Biane1997JFA} is the special case $(\rho,\zeta)=(1,0)$, while the free unitary Brownian motion is the edge case $(\rho,\zeta)=(1,-1)$.  The two-parameter family introduced in \cite{Kemp2016} corresponds to $b_{\rho,\zeta}$ with $\zeta\in(-\rho,\rho)$ (i.e.\ real $\zeta$).  The fully complex parameter $\zeta$ with $|\zeta|<\rho$ was later introduced in \cite{DHK2020} in the context of heat kernel analysis on compact type Lie groups (notably the question of when the heat kernel can be complexified in both space {\em and time}).  The three-parameter families of free stochastic processes $w_{\rho,\zeta}$ and $b_{\rho,\zeta}$, including the degenerate cases $|\zeta|=\rho$, were studied extensively in \cite{HallHo2023,HallHo2025Spectrum}, as strong limits in \cite{BCC2025}, and most recently in terms of general additive perturbations of $w_{\rho,\zeta}$ in \cite{Zhong2025}.

\begin{notation} As with the matrix-valued processes  $W = W_{\rho,\zeta}$ and $B = B_{\rho,\zeta}$, we generally suppress the explicit dependence of the notation for the Brownian motion processes on the parameters $(\rho,\zeta)$, and define
\begin{align*}
w_{\rho,\zeta}(t) =: w(t) \qquad\;\; b_{\rho,\zeta}(t) =: b(t).
\end{align*}
\end{notation}

The following lemma guides our aim to describe the large-$N$ limit distribution $\mu_{b_0,t}$ of eigenvalues of $B(t)$ in Theorem \ref{thm.main}.

\begin{lemma} \label{lem.init.cond.*-conv} Let $\rho>0$ and $\zeta\in\mathbb{C}$ with $|\zeta|\leq \rho$, and let $B$ be a $(\rho,\zeta)$-Brownian motion on $\mathrm{GL}(N,\mathbb{C})$.  Let $(\A,\tau)$ be a $W^\ast$-probability space that contains a free multiplicative $(\rho,\zeta)$-Brownian motion $b$, and let $b_0$ be an operator in $\A$ that is freely independent from $b$.

Let $B_0=B_0^N$ be a sequence of random matrices in $\M_N(\C)$ independent from the process $B$ such that $B_0$ converges a.s.\ in $\ast$-distribution to $b_0$ as $N\to\infty$.

Then for any $t\ge 0$, $B_0 B(t)$ converges in $\ast$-distribution to $b_0b(t)$ a.s.
\end{lemma}

Under the additional assumption that $B_0$ converges {\em strongly} to $b_0$ (see Definition \ref{def.*-dist}), the recent paper \cite{BCC2025} proved Lemma \ref{lem.init.cond.*-conv} with the stronger conclusion of strong convergence.  Lemma \ref{lem.init.cond.*-conv} is proved in Section \ref{sect.background.BMs}.

\subsection{Singular Values, Eigenvalues, and Hermitization\label{sect.Brown.meas}}

Lemma \ref{lem.init.cond.*-conv} suggests that the large-$N$ limit $\mu_{b_0,t}$ of the empirical law of eigenvalues of $B_0B(t)$ is connected to the free multiplicative Brownian motion $b_0b(t)$.  That connection is through {\em Brown measure}.

\begin{defi} \label{def.BrownMeasure} Let $(\A,\tau)$ be a $W^\ast$-probability space, and let $a\in\A$ (just $L^2(\A,\tau)$ suffices).  Then $a^\ast a = |a|^2$ is selfadjoint, and so has a spectral measure $\mu_{|a|^2}$: the unique compactly-supported probability distribution on $\R_+$ whose moments are given by $\int x^k \mu_{|a|^2}(dx) = \tau[(a^\ast a)^k]$ (see \eqref{eq.spectral.meas}).  Denote by $\nu_a$ the push-forward of $\mu_{|a|^2}$ under $x\mapsto\sqrt{x}$; in other words, $\nu_a$ is the spectral measure of $|a|:=\sqrt{a^\ast a}$.

For $z\in\C$, define the {\bf log potential} $U_a\colon\C\to\R\cup\{-\infty\}$ by
\[ U_a(z):= \frac{1}{2\pi}\int_0^\infty \log(x)\,\nu_{a-z}(dx). \]
It is a subharmonic function, and so its (weak) Laplacian is a positive Radon measure.  In this case it is a probability measure: the {\bf Brown measure of $a$}
\[ \mu_a:= \Delta U_a \quad \text{i.e.} \quad
\int_{\C} \psi\,d\mu_a = \int_{\C} \Delta\psi(z)\, U_a(z)\,dz \qquad \forall\,\psi\in C^\infty(\C). \]
The Brown measure is then also uniquely defined by 
\[ \int_{\C} \log |\lambda-z|\,\mu_a(d\lambda) = U_a(z). \]
\end{defi}
The measure $\mu_a$ was introduced by Brown in \cite{Brown1986}.  We use the same notation $\mu_a$ for Brown measure as for spectral measure because if $a$ is normal then its Brown measure {\em is} its spectral measure.

Brown measure is closely related to the Hermitization procedure invented by Girko \cite{Girko} for studying eigenvalues of non-normal random matrices.  Indeed, in the $W^\ast$-probability space $(\M_N(\C),\ts)$, the Brown measure is exactly the empirical law of eigenvalues.  To see why, note that for $A\in\M_N(\C)$, the eigenvalues of $|A| = \sqrt{A^\ast A}$ are the singular values $\sigma_j(A)$ of $A$, and hence $\nu_A$ is the empirical law of singular values.  But then
\[ \int_0^\infty \log(x)\,\nu_{A-zI}(dx) = \frac1N\sum_{j=1}^N \log \sigma_j(A-zI) = \frac1N\log\det[|A-zI|]. \]
But $\log \det[|A-zI|] = \frac12\log \det[|A-zI|^2] = \log|\det[A-zI]|$, and hence
\[ \frac1N\log\det[|A-zI|] = \frac1N\sum_{j=1}^N \log|z-\lambda_j(A)| \]
where $\lambda_j(A)$ are the eigenvalues of $A$.  The (weak) Laplacian of this is indeed the empirical law of eigenvalues \eqref{eq.ESD.def}.

The log potential $U_A$ is locally uniformly integrable in $z$, with universal bounds independent of $A$; hence, proving a sequence of Brown measures $\mu_{A^N}$ converges weakly to a Brown measure $\mu_a$ is tantamount to proving pointwise convergence of the log potentials:  $U_{A^N}(z)\to U_a(z)$ for a.e.\ $z\in\C$.  From Definition \ref{def.BrownMeasure}, this means the goal is to prove
\begin{equation} \label{eq.log.weak.conv} \lim_{N\to\infty} \int_0^\infty \log\,d\nu_{A^N-zI} = \int_0^\infty \log\,d\nu_{a-z} \quad \text{for a.e. }z\in\C. \end{equation}
Equation \eqref{eq.log.weak.conv} had the flavor of weak convergence, but the test function $\log$ is neither bounded nor continuous near $0$ and $\infty$.  Hence, to prove convergence of the ESD (i.e.\ Brown measure) of a random matrix ensemble $A^N\in\M_N(\C)$ to the Brown measure of an operator $a\in\A$, it is necessary to show not only that the the measure $\nu_{A^N-zI}$ (empirical law of singular values of $A-zI$) converges weakly to $\nu_{a-z}$; one must also show that mass does not escape at $0$ or $\infty$ -- i.e.\ uniform integrability of $\log$ with respect to the measures $\{\nu_{A^N-zI}\}_{N\in\N}$.  Proving small singular values of $A-zI$ don't concentrate and decay too fast can be quite subtle, and can fail even in some simple examples, cf.\ \cite[Example 1.2]{Bordenave-Chafai-circular}.

After Girko \cite{Girko} presented the above method for handling eigenvalues of non-normal random matrix ensembles (in particular i.i.d.\ matrices and the circular law), it took more than 25 years for his program to reach full fruition, with the first major breakthrough by Bai \cite{Bai1997} requiring the law of each entry of $A^N$ to have a bounded density on $\C$ and more than $6$ finite moments. Numerous improvements over the following decade culminated with Tao and Vu's full proof \cite{TaoVu2010-aop}.

The same Hermitization roadmap has been used for various generalizations of the Circular Law (see the excellent survey \cite{Bordenave-Chafai-circular} for an extensive summary of work up to $\sim$2011), the much larger class of ``Single Ring'' ensembles (that are unitarily bi-invariant) \cite{GuionnetKZ-single-ring,BasakDembo2013} and generalizations, e.g.\ \cite{BenaychGeorges2017,BaoES2019singlering,HoZhong2025}, the elliptic law \cite{NgO'R:elliptic}, inhomogeneous circular laws \cite{CHNR1,CHNR2,AEK:circ,AlKr:circ}, local circular laws \cite{BourgadeYauYin2014}, polynomials of independent matrices \cite{O'RSo,ORSV,GNT:product,Cook,Yi-Han}, sparse directed graphs \cite{Cook2019,BCZ,LLTTY:circ,SSS:sparse} and (related to the present work) very general matrix random walks with unitarily bi-invariant increments \cite{multi}. 

Our present interest is in Brownian motion $B_0B(t)$ on $\mathrm{GL}(N,\C)$.  For these ensembles, weak convergence of empirical singular values $\nu_{B_0B(t)-zI}$ to $\nu_{b_0b(t)-z}$ follows directly from Lemma \ref{lem.init.cond.*-conv}.  Theorem \ref{thm.main}, via Theorem \ref{thm.2}, completes the rest of the Hermitization procedure.  In particular, we can now clarify:

\begin{defi}    \label{def:mub0t}
    The deterministic measure $\mu_{b_0,t}$ in Theorem \ref{thm.main} is the Brown measure of the operator $b_0b(t)$.  In particular: Theorem \ref{thm.main} asserts that {\bf the empirical spectral distribution of $B_0B(t)$ converges weakly a.s.\ to the Brown measure of $b_0b(t)$.}
\end{defi} 

\subsection{Directions for Future Work\label{sec:future}}

The tools developed in this paper open the door to a more general study of the spectrum of continuous-time non-normal matrix processes and refined properties of the spectrum.

\subsubsection{Other matrix diffusion processes}

The methods we develop for Theorem \ref{thm.main} are not intrinsically tied to diffusion processes on Lie groups. Consider a general matrix diffusion processes of the form $dM(t) = \sigma(M(t))dW(t) + b(M(t))dt$ for sufficiently nice matrix-valued drift and diffusion functions $\sigma,b$, or more generally
\[
dM(t) = \sigma(M(t))\# dW(t) + b(M(t))\,dt
\]
where $\sigma\colon\M_N(\C)\to\M_N(\C)\otimes\M_N(\C)$ and $(A\otimes B)\# X = AXB$. Under appropriate regularity assumptions, we may seek a version of Theorem \ref{thm.2} to locally approximate the process in terms of an increment of the standard Brownian motion $W(t)\in \M_N(\C)$, which is enough to regularize the spectrum and establish convergence to the Brown measure.

\subsubsection{Regularization of Brown's spectral measure}

A well-known result of Sniady (see \cite[Theorem 6]{Sniady}) states that if $B_0=B_0^N$ is a sequence of random matrices converging almost surely in $*$-distribution to an operator $b_0$, and $W(t)=W^N(t)$ is (additive) Ginibre Brownian motion independent from $B_0$, then for every fixed $t>0$, $\mu_{B_0+W(t)}$ converges weakly to the Brown measure $\mu_{b_0+w(t)}$, and moreover, there exists a sequence $t_N\to0$ such that $\mu_{B_0+W(t_N)} \to \mu_{b_0}$ weakly almost surely.
The speed at which $t_N$ may vanish and its dependence on the initial condition $B_0$ was investigated in \cite{GWZ2014,FPZ2015}.
Our main result Theorem \ref{thm.main} establishes a multiplicative version of the first half of Sniady's result concerning fixed $t>0$. It is natural to pursue an analogue of the second half: Does there exist a sequence $t_N\to0$ such that $\mu_{B_0B(t_N)}\to \mu_{b_0}$ weakly almost surely? How does the allowable speed of vanishing of $t_N$ depend on the initial condition?

\subsubsection{Path properties}

Consider the empirical laws $\mu_t^N=\mu_{B_0B(t)}$ as a stochastic process $(\mu_t^N)_{t\ge0}$ taking values in the Polish space $\mathcal M_1(\C)$ of probability measures on $\C$ with the topology of weak convergence. 
Theorem \ref{thm.main} immediately implies a.s.\ convergence of finite dimensional marginals of this process, i.e.\ for fixed times $0<t_1<\cdots<t_n$, $(\mu^N_{t_1},\dots, \mu^N_{t_n})$ converges almost surely in the product weak topology on $\mathcal M_1(\C)^n$ to the deterministic vector $(\mu_{b_0b(t_1)},\dots, \mu_{b_0b(t_n)})$. (In fact we can take any countable sequence of times.)
Does the random curve $(\mu_t^N)_{t>0}$ in $\mathcal M_1(\C)$ almost surely converge locally uniformly to the deterministic curve $(\mu_{b_0b(t)})_{t>0}$?

\subsubsection{Local universality}

Beyond the limiting law for the global spectral distribution, the Hermitization approach and control of the pseudospectrum provides a route to prove universality at mesoscopic and local spectral scales. 
In particular, a \emph{local law} should hold, showing the empirical law $\mu_{B_0B(t)}$ is well approximated by the limit $\mu_{b_0,t}$ at scales down to $N^{o(1)-\frac12}$; such a result was established for matrices with i.i.d.\ entries in \cite{TaVu:iidloc,BourgadeYauYin2014}. 
At the scale $N^{-1/2}$ of the typical separation of eigenvalues we expect rescaled correlation functions to converge to universal determinantal point processes, with Ginibre kernel in the bulk and Faddeeva kernel at smooth points of the boundary of the support; for related results see \cite{TaVu:iidloc, CCEJ, Osman,DuYa} for i.i.d.\ matrices and \cite{Berman, HeWe-acta} for normal matrix ensembles. 
The current understanding of point processes of eigenvalues in the vicinity of \emph{singular} boundary points 
-- such as at $z=-1$ for the Lima Bean Law at time $t=4$ (see middle of Figure \ref{fig-evolving-lima-bean}) 
-- 
is incomplete even for \emph{additive} Ginibre Brownian motion $B_0+W(t)$. The possible behaviors of the limiting spectral density of $B_0+W(t)$ near singular boundary points was classified in \cite{ErJi,AlKr}, and universality (allowing $W$ to have non-Gaussian entries) for local eigenvalue statistics near such points was established in \cite{CEJ:critical}, assuming $B_0$ is normal; the correlation functions themselves have only been computed for certain initial conditions $B_0$ \cite{BGNTW2015, LiLu}. 
See also recent works \cite{AKMW, BLY} on lemniscate-type normal matrix ensembles.  
Whether the classification of singular boundary points and universal correlation functions carry over to multiplicative Brownian motion $B_0B(t)$, with $B_0$ possibly non-normal, are open questions.

\section{The Brown(ian) Measure\label{sect.Brown(ian).meas}}
In studying the bulk limits of eigenvalues of non-normal random matrix ensembles, there are two fundamental (challenging) problems: identifying the (putative) limit measure, and proving convergence to it.  In this paper, we solve the convergence problem: Theorem \ref{thm.main} proves that the empirical law of eigenvalues of Brownian motion $B_{\rho,\zeta}(t)$ on $\mathrm{GL}(N,\C)$ with any initial state $B_0$ possessing a large-$N$ limit $b_0$ converges to an identifiable deterministic measure: the Brown measure of the operator $b_0b_{\rho,\zeta}(t)$.  The former problem of identifying those Brown measures has been well studied in the past five years, see \cite{DHKBrown,HoZhong2020Brown,HallHo2023,DemniHamdi2022,HallEak2025,HallHo2025Spectrum}.  These papers succeeded in computing the Brown measure precisely for general $|\zeta|\le\rho$ and any {\em unitary} $b_0$, along with some partial results for {\em positive definite} $b_0$, using stochastic, complex analytic, and PDE methods.

In this section, we briefly summarize known results about these Brown measures (all from 2022 or later), with some pictures of their support sets with simulations of eigenvalues of the associated Brownian motions for fairly large $N$.  Given the now known (thanks to Theorem \ref{thm.main}) robust connection to the eigenvalues, the simulations are suggestive of the more general picture of the support of the Brown measure for general initial conditions.

We begin with the case of the ``standard'' free multiplicative Brownian motion $b_{1,0}$, with unitary initial condition $b_0=u_0$.  The computation of the Brown measure with $b_0=1$ was initially accomplished by Driver--Hall--Kemp in \cite{DHKBrown}, and then extended to include unitary initial conditions by Ho--Zhong in \cite{HoZhong2020Brown}.

\begin{theorem}[\cite{DHKBrown},\cite{HoZhong2020Brown}] \label{thm.Brown.meas.10} Let $b_{0,1}(t)$ be a free multiplicative Brownian motion, and let $u_0$ be a unitary operator freely independent from $b_{0,1}$.  Denote by $\mu_{u_0}$ the spectral measure of $u_0$, a probability measure on the unit circle $\mathbb{T}$.

Define $T(u_0,\cdot)\colon\C\to\R_+$ by
\begin{equation} \label{eq.T.u} T(u_0,z) = \frac{\log |z|^2}{|z|^2-1}\left(\int_{\mathbb{T}} \frac{\mu_{u_0}(d\xi)}{|\xi-z|^2}\right)^{-1}.
\end{equation}
Also, for each $t>0$, define the open precompact set $\Sigma(u_0,t)$ by
\begin{equation} \label{eq.Sigma.u} \Sigma(u_0,t) = \{z\in\C\colon T(u_0,z)<t\}. \end{equation}
The Brown measure $\mu_{u_0b_{1,0}(t)}$ has support equal to $\overline{\Sigma(u_0,t)}$.  The Brown measure is absolutely continuous with respect to Lebesgue measure, and has density on $\overline{\Sigma(u_0,t)}$ that is real analytic in $\Sigma(u_0,t)$.  The density has the form $z=re^{i\theta}\mapsto w_t(\theta)/r^2$, with the radial function $w_t$ bounded by $w_t(\theta)<\frac{1}{\pi t}$.
\end{theorem}

\begin{figure}[h!]
\centering
  \begin{tikzpicture}

    \node[inner sep=2pt] at (0cm, 0cm)
      {\includegraphics[width=5.5cm]{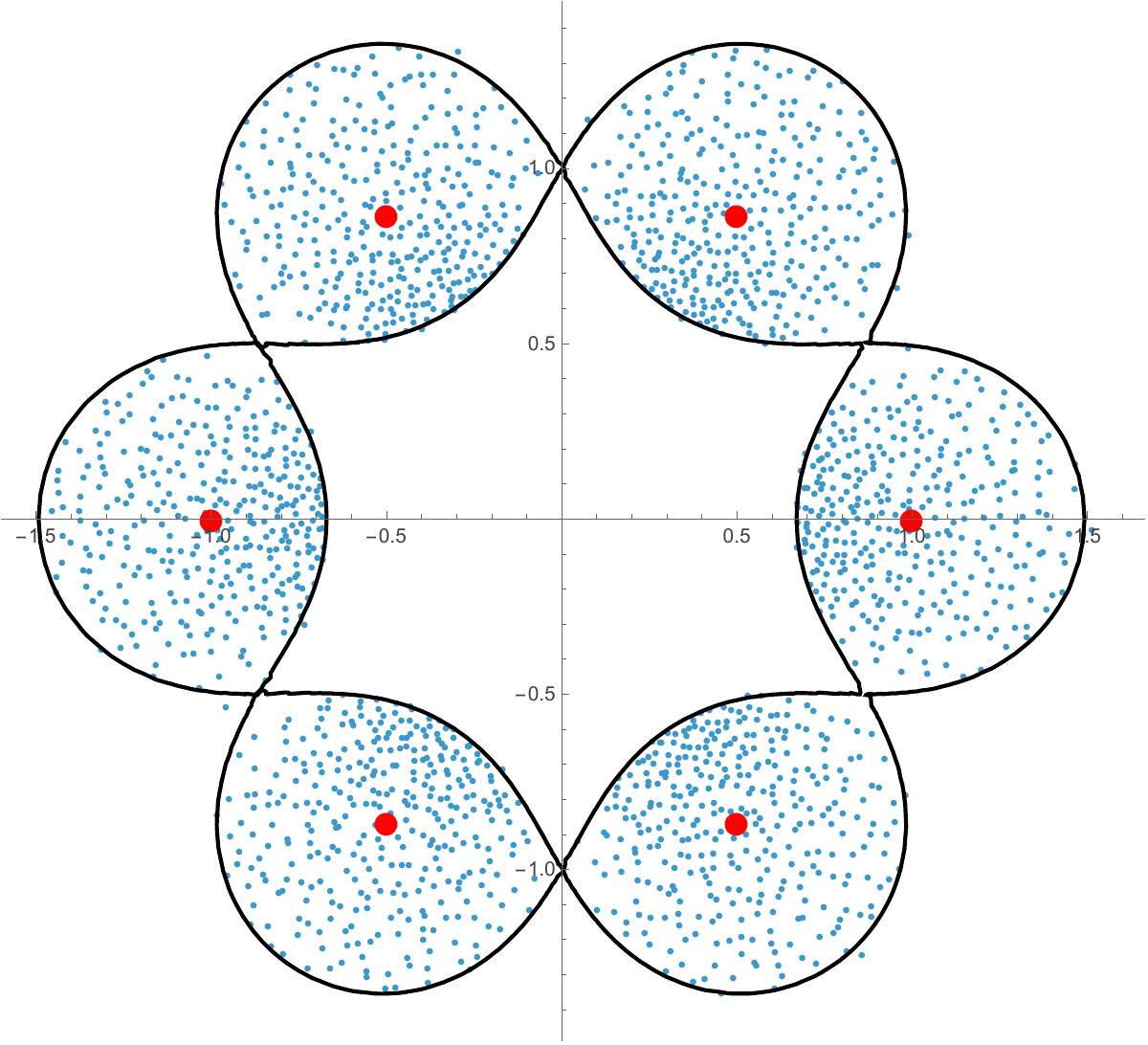}};

    \node[inner sep=2pt] at (6cm, -0.0cm)   
      {\includegraphics[width=5.5cm]{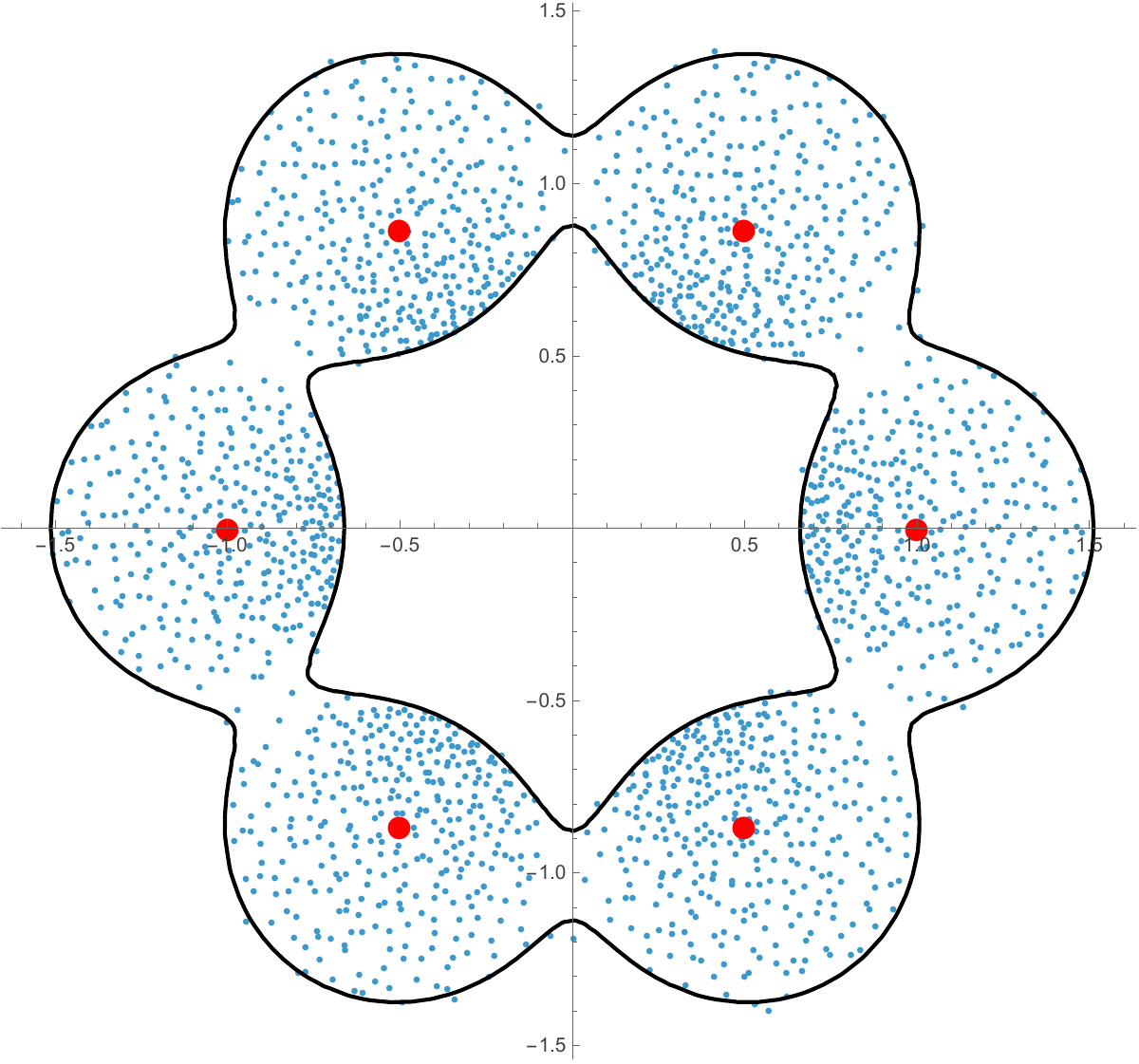}};

  \end{tikzpicture}
  \caption{Eigenvalues of $B_0B_{1,0}(t)$ with $N=2000$ and $t=\frac23$ (left) and $t=0.7$ (right), with initial condition $B_0$ a unitary $u_6$ with equal mass eigenvalues at the $6$th roots of unity.  Also plotted is the boundary $\partial\Sigma(u_6,t)$ (see \eqref{eq.Sigma.u}).
  \label{fig.six-pointed}}
\end{figure}

Following this, in \cite{HallHo2023} Hall--Ho refined the PDE analysis from the above-mentioned papers to allow the Brownian motions $b_{\rho,\zeta}(t)$ from the full parameter range $|\zeta|\le\rho$.

\begin{notation} \label{notat.ovweparam}
    The laws of the random matrices $W_{\rho,\zeta}(t)$ and $B_{\rho,\zeta}(t)$ are overparametrized. Since $(W(t))_{t\ge 0}$ is the same as the law of the process $(W(rt)/\sqrt{r})_{t\ge 0}$ for any $r>0$, it follows that $W_{\rho,\zeta}(t)$ (and consequently $B_{\rho,\zeta}(t)$) satisfy
\begin{equation} \label{eq.overparam}
\begin{aligned}
W_{\rho,\zeta}(t) &\equaldist W_{t\rho,t\zeta}(1) \equaldist W_{1,\zeta/\rho}(\rho t) \\
B_{\rho,\zeta}(t) &\equaldist B_{t\rho,t\zeta}(1) \equaldist B_{1,\zeta/\rho}(\rho t).
\end{aligned} \end{equation} 
The same applies to the free limits $w_{\rho,\zeta}(t)$ and $b_{\rho,\zeta}(t)$.

\cite{HallHo2023,HallHo2025Spectrum} elected to simplify parameters and only work with $b_{\rho,\zeta}(1)$; we do the same for the remainder of this section.  We also follow \cite{HallHo2023}'s convention and, in this parameter reduction, let $t$ play the role of $\rho$.  Hence we treat the operators
\begin{align*}
W(t,\zeta)&:= W_{t,\zeta}(1)  \qquad & \qquad B(t,\zeta)&:= B_{t,\zeta}(1) \\
w(t,\zeta)&:= w_{t,\zeta}(1) \qquad & \qquad b(t,\zeta)&:= b_{t,\zeta}(1).
\end{align*} 
Then the Brownian motion $u_0b_{1,0}(t)$ in Theorem \ref{thm.Brown.meas.10} has the same law as
\[ u_0b_{1,0}(t) \equaldist u_0b(t,0). \]
\end{notation}

For each $t>0$ and any $\zeta\in\C$ with $|\zeta|\le t$, the Brown measure $\mu_{u_0b(t,\zeta)}$ is a push-forward of the Brown measure $\mu_{u_0b(t,0)}$ from Theorem \ref{thm.Brown.meas.10} under a continuous map that is a diffeomorphism in the interior $\Sigma(u_0,t)$ and a homeomorphism on the closure.  To properly define this map, we introduce the following analytic transform: for any operator $a$ in a $W^\ast$-probability space, define
\begin{equation} \label{eq.Ja} J_a(z) = \frac12\int_{\C}\frac{\xi+z}{\xi-z}\,\mu_a(\xi) \end{equation}
where $\mu_a$ is the Brown measure of $a$.  This integral defines a conformal map on the complement of $\supp\mu_a$; on this region, it can be written in the form $J_a(z) = \frac12\tau[(a+z)(a-z)^{-1}]$.  The integral \eqref{eq.Ja} may be well defined for (some) $z$ in the support of $\mu_a$, as will be the case relevant to the following theorem.

\begin{theorem}[\cite{HallHo2023,HallHo2025Spectrum}] \label{thm.Brown.meas.rhozeta} Let $t>0$ and $\zeta\in\C$ with $|\zeta|\le t$ (but $\zeta\ne -t$, which is the unitary Brownian motion).  Let $b(t,\zeta)$ be a free multiplicative Brownian motion (see Notation \ref{notat.ovweparam}) and let $u_0$ be a unitary operator freely independent from $b(t,\zeta)$.  Define the map $\Psi^{u_0}_{t,\zeta}$ by
\begin{equation} \label{eq.Psi} \Psi^{u_0}_{t,\zeta}(z) = z\exp\{\zeta J_{u_0b(t,0)}(z)\}. \end{equation}
This map is continuous on $\C$.  Restricted to $\Sigma(u_0,t)$, it is a diffeomorphism onto its image, an open set we denote $\Sigma(u_0,t,\zeta)$.  The map $\Psi^{u_0}_{t,\zeta}$ is a homeomorphism from the closure of $\Sigma(u_0,t)$ onto the closure of $\Sigma(u_0,t,\zeta)$, and is a conformal map between the complements of these closures.

The Brown measure of $u_0b(t,\zeta)$ is the push-forward under $\Psi^{u_0}_{t,\zeta}$ of the Brown measure of $u_0b(t,0)$ (from Theorem \ref{thm.Brown.meas.10}).  Its support is $\overline{\Sigma(u_0,t,\zeta)}$, and it has a smooth density on its support.
\end{theorem}

One key feature of Theorem \ref{thm.Brown.meas.rhozeta} which explains the parameter choices is that, for fixed $t>0$, the support regions $\{\Sigma(u_0,t,\zeta)\colon |\zeta|\le t\}$ are all diffeomorphic to the region $\Sigma(u_0,t)$ from the standard Brownian motion $u_0b(t,0)$.  I.e.\ the topology of the domain where the eigenvalues cluster can change with $t$, while $\zeta$ is a deformation parameter that does not change the topology of the support set.

\begin{figure}[h!]
\centering
  \begin{tikzpicture}

    \node[inner sep=2pt] at (0cm, 0cm)
     {\includegraphics[width=5cm]{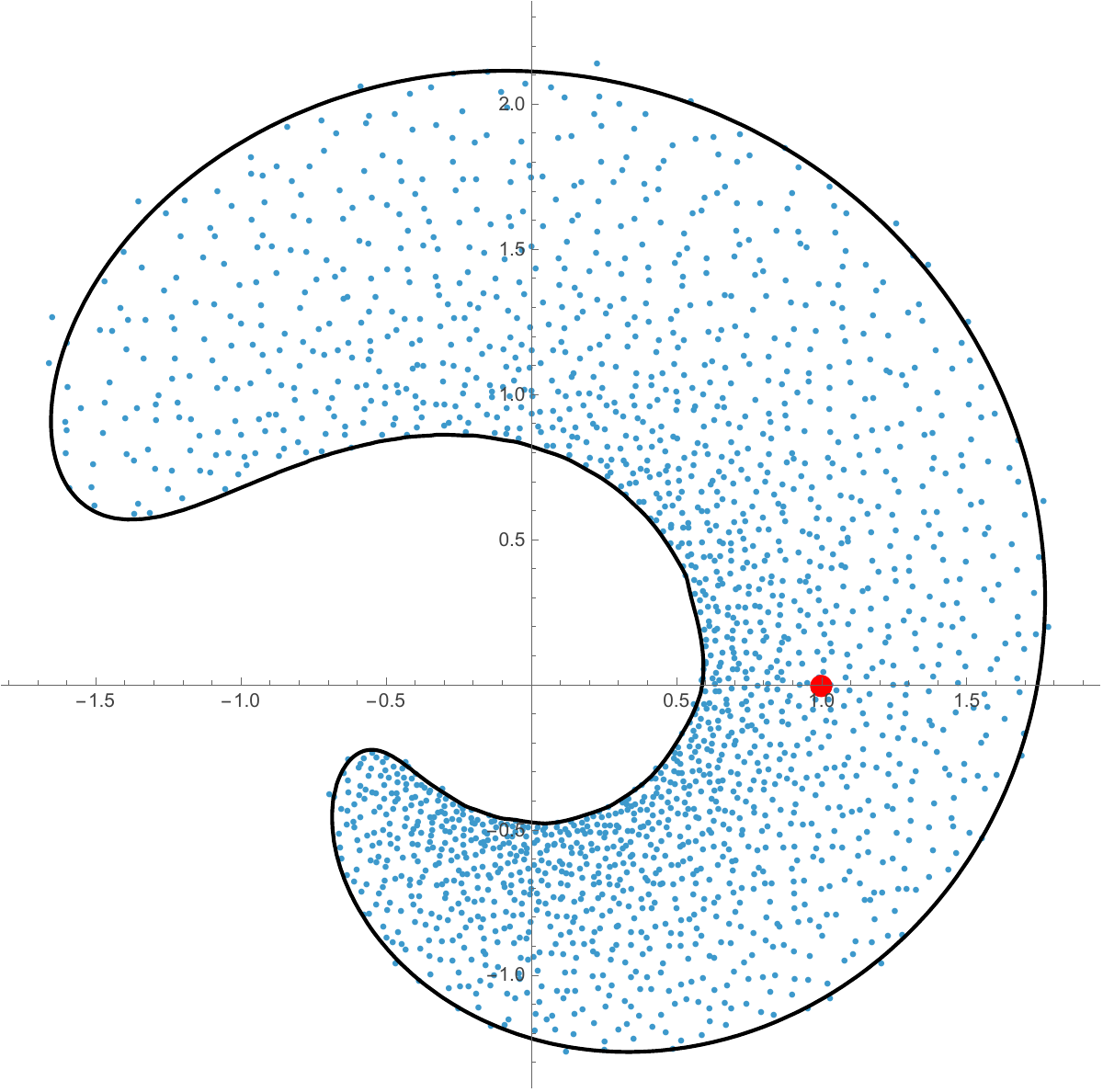}};

    \node[inner sep=2pt] at (5.5cm, 0cm)
        {\includegraphics[width=5cm]{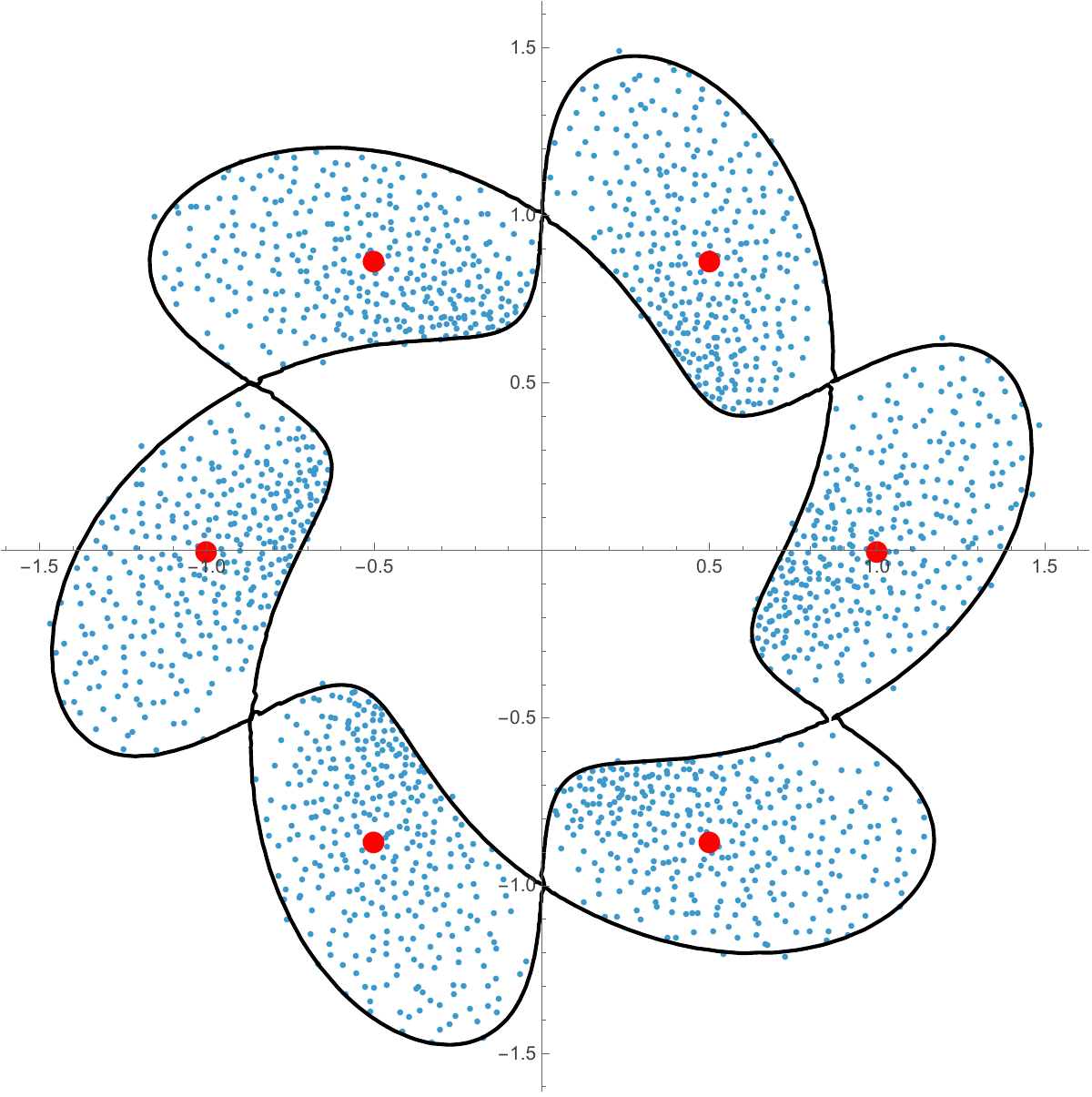}};

  \end{tikzpicture}
  \caption{Eigenvalues of $B_0 B(t,\zeta)$ with $N=2000$.  On the left, $B_0=I$ and $(t,\zeta)=(3,-2+\i)$; on the right, $B_0$ is unitary with $6$th roots of unity as eigenvalues, and $(t,\zeta)=(\frac23,\frac{\i}{3})$.  Also plotted are the boundaries $\partial\Sigma(b_0,t,\zeta)$, see Theorem \ref{thm.Brown.meas.rhozeta}.
  \label{fig-worm-star}}
\end{figure}

\begin{figure}[h!]
\centering
  \begin{tikzpicture}

    \node[inner sep=2pt] at (-0.5cm, 0cm)
      {\includegraphics[height=4cm]{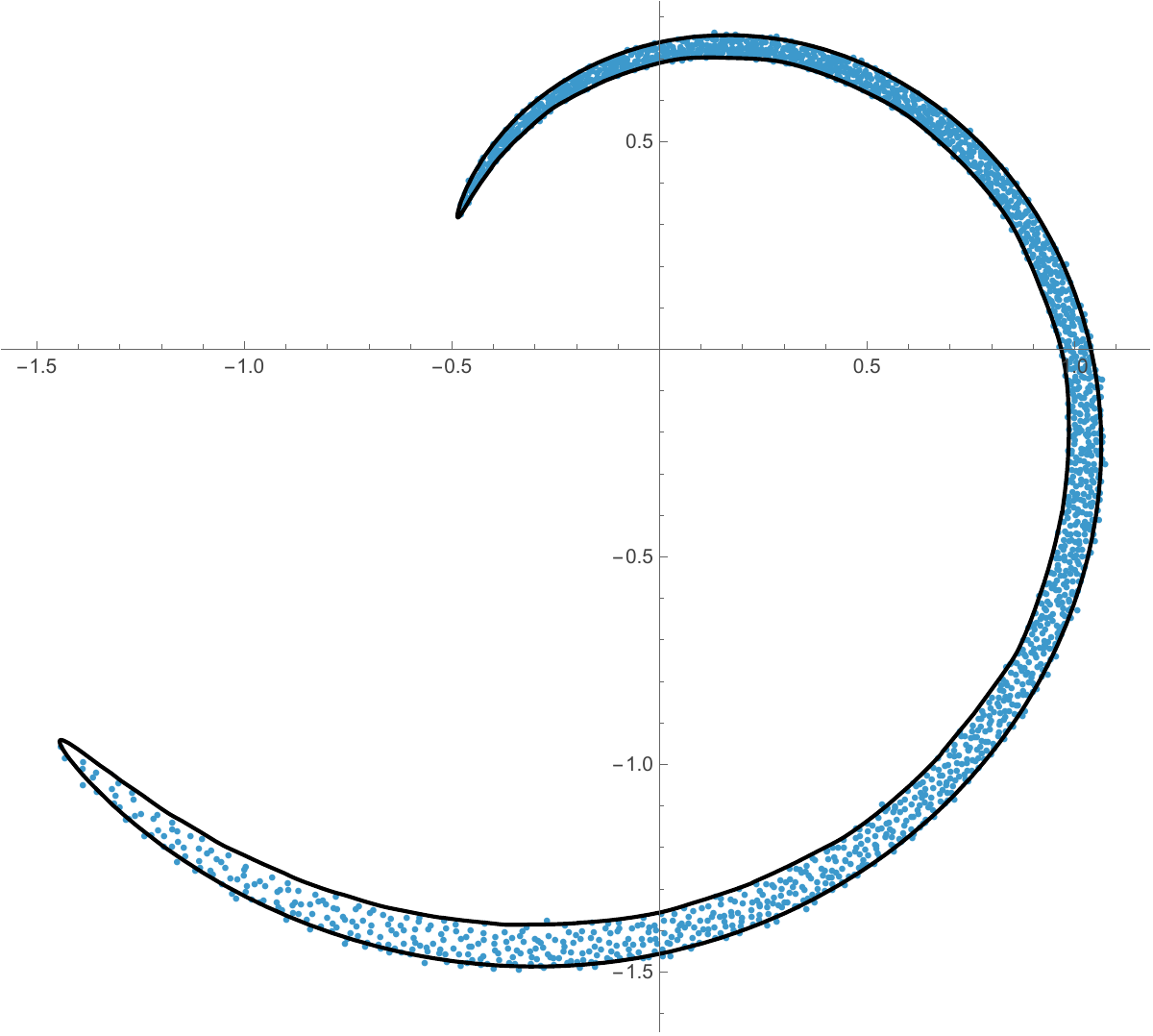}};

    \node[inner sep=2pt] at (3.75cm, 0cm)   
      {\includegraphics[height=4cm]{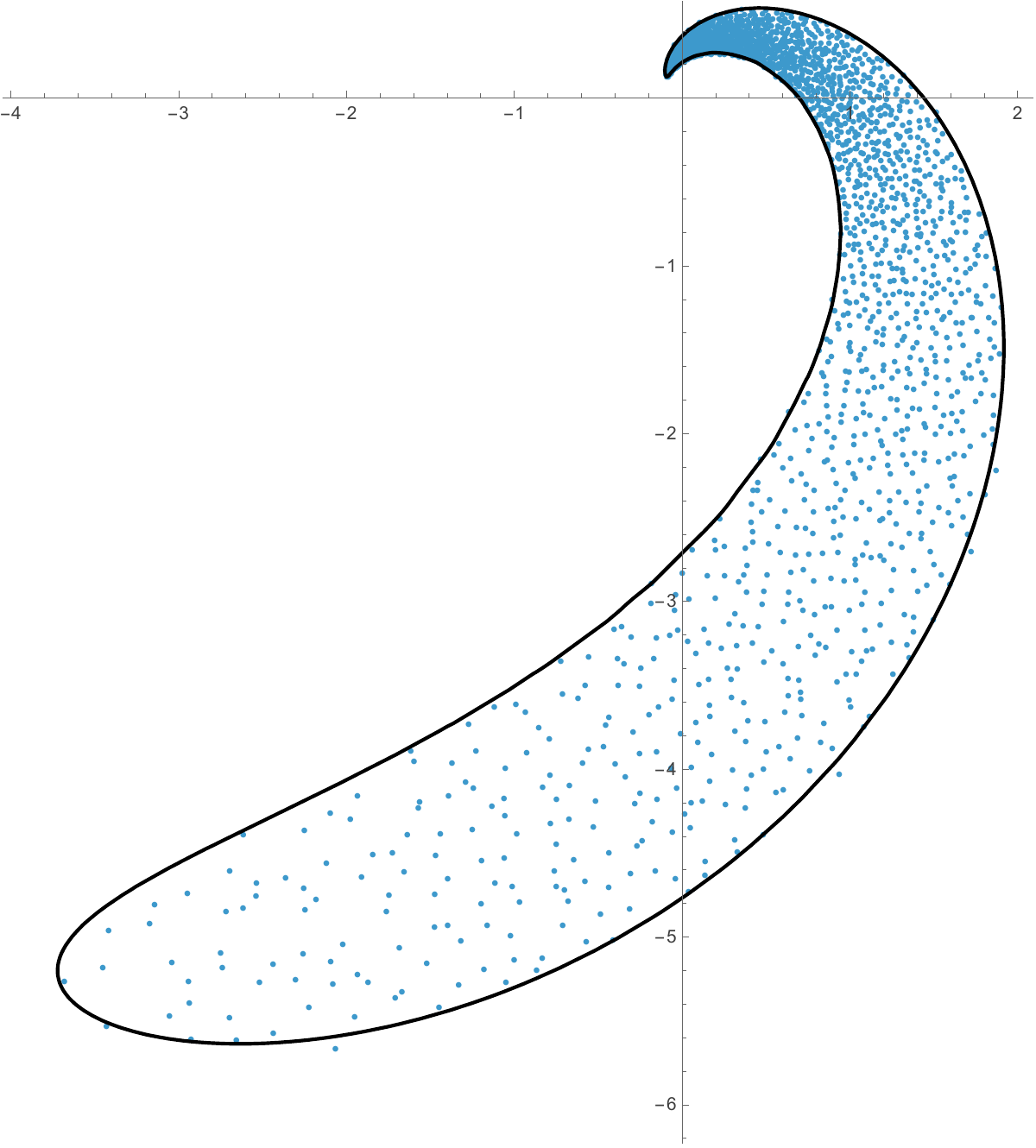}};
      
    \node[inner sep=2pt] at (-0.5cm, -3.75cm)
      {\includegraphics[width=4.5cm]{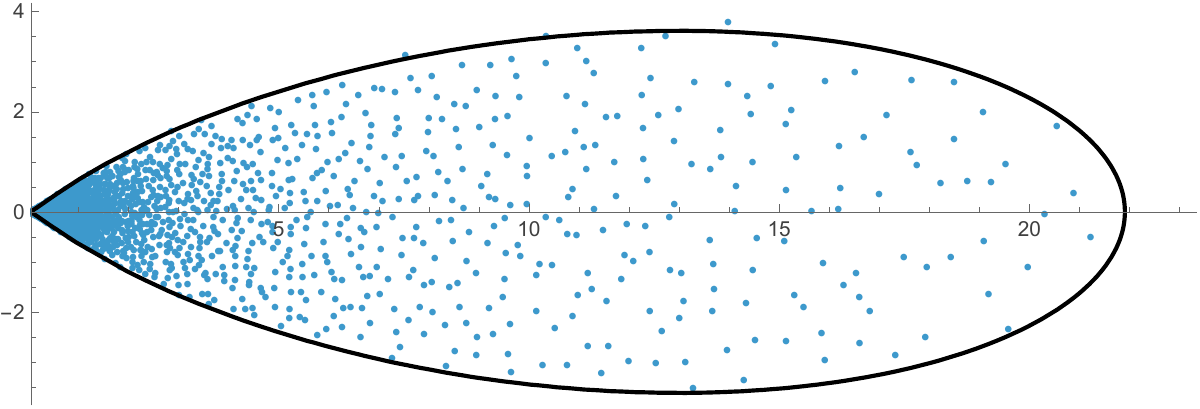}};

    \node[inner sep=2pt] at (4cm, -4cm)  
      {\includegraphics[height=3.6cm]{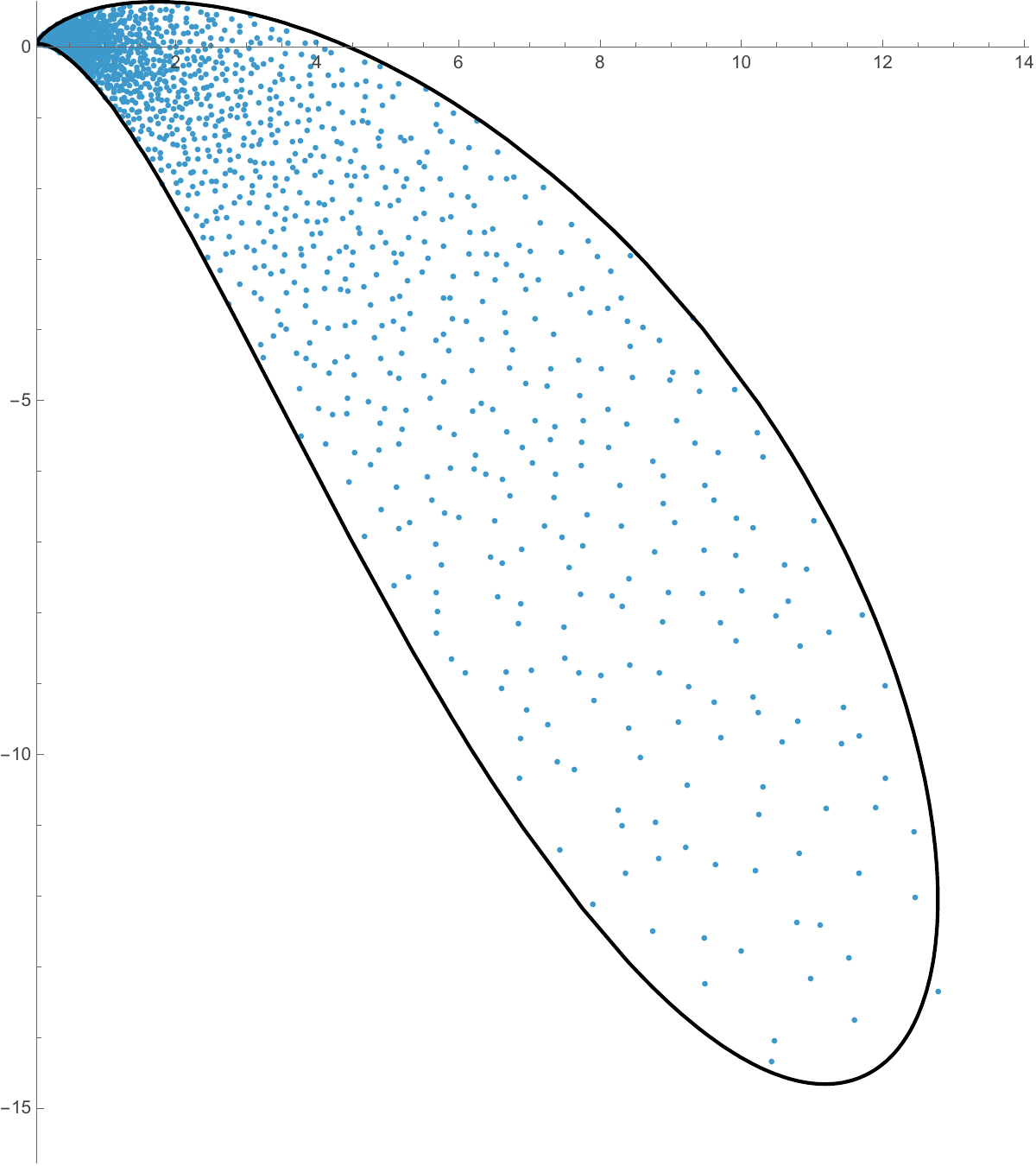}};

      
  \end{tikzpicture}
  \caption{Eigenvalues of $B_{\rho,\zeta}(2)$ with $N=2000$, $\rho=2$, and $\zeta = -2e^{i\alpha\pi}$ where $\alpha$ ranges through $\frac16$, $\frac12$, $\frac56$, and $1$ clockwise from top left.  Also plotted are the boundaries $\partial\Sigma(1,2,\zeta)$, see Theorem \ref{thm.Brown.meas.rhozeta}.  These simulations represent the hypoelliptic case, where $|\zeta|=\rho$ but $\zeta\ne -\rho$.  The eigenvalues ``unravel'' as $\zeta$ rotates around the circle $|\zeta|=\rho$ away from $\zeta=-\rho$ (i.e.\ $\alpha=0$), where the Brownian motion is unitary-valued.
  \label{fig-hypoelliptic}}
\end{figure}

Finally, moving beyond the case of unitary initial conditions, the papers \cite{DemniHamdi2022,HallEak2025} use a related but qualitatively different PDE approach to analyze the support of the Brown measure of $b_0b(t,\zeta)$ in the case that $b_0$ is a positive definite operator. In \cite{DemniHamdi2022}, Demni--Hamdi considered the case $\zeta=0$ and restricted $b_0$ to be a projection; refining those ideas, Eaknipitsari--Hall \cite{HallEak2025} considered general $b(t,\zeta)$ and general positive definite $b_0$.  The formulas and techniques in the latter paper are not particularly specific to the positive case, and so we state the results as a special case of a conjecture.

\begin{conj}[\cite{HallEak2025}] \label{conj.b0} Let $b_0$ be an invertible operator in a $W^\ast$-probability space $(\A,\tau)$.  Define
\[ \tilde{p}_0(z) = \tau[|b_0-z|^{-2}] \qquad \& \qquad \tilde{p}_2(z) = \tau[b_0^\ast b_0|b_0-z|^{-2}] \]
where, as usual, for any operator $a\in\A$, $|a|^2:=a^\ast a$.  Define $T(b_0,\cdot)\colon\C\to\R_+$ by
\begin{equation} \label{eq.T.b0} T(b_0,z) = \frac{\log(\tilde{p}_2(z))-\log(|z|^2\tilde{p}_0(z))}{\tilde{p}_2(z)-|z|^2\tilde{p}_0(z)}.
\end{equation}
For each $t>0$, define the open precompact set $\Sigma(b_0,t)$ by
\begin{equation} \label{eq.Sigma.b0} \Sigma(b_0,t) = \{z\in\C\colon T(b_0,z)<t\}. \end{equation}
Let $b(t,\zeta)$ be a free multiplicative Brownian motion (see Notation \ref{notat.ovweparam}) freely independent from $b_0$.  Then the support of the Brown measure of $b_0b(t,\zeta)$ is equal to the image $\overline{\Sigma(b_0,t,\zeta)}$ of $\overline{\Sigma(b_0,t)}$ under the map $\Psi^{b_0}_{t,\zeta}$ of \eqref{eq.Psi}.
\end{conj}
The expression \ref{eq.Sigma.b0} is defined by limits if any of the terms diverge.  The quantities $\tilde{p}_0(z)$ and $\tilde{p}_2(z)$ are positive but possibly infinite; if one is infinite, a calculus exercise shows that $T(b_0,z)=0$, meaning that $z\in\Sigma(b_0,t)$ for all $t>0$.

If $b_0$ is unitary then $\tilde{p}_2(z)=\tilde{p}_0(z)$, and \eqref{eq.T.b0} reduces \eqref{eq.T.u}; so Conjecture \ref{conj.b0} follows from Theorem \ref{thm.Brown.meas.rhozeta} in the case of a unitary initial condition.  In \cite{HallEak2025}, the authors considered the case that $b_0>0$ (i.e.\ positive definite), and the conjecture was half proved in that case: $\supp \mu_{b_0b(\rho,\zeta)}$ was shown to be {\em contained} in $\overline{\Sigma(b_0,\rho,\zeta)}$ (with an exception at $0$).  The obstructions to extending this containment to more general initial conditions were technical, and it is natural to expect it to hold true in full generality.  The following simulations of the region $\Sigma(b_0,\rho,\zeta)$ with eigenvalues of $B^N_{\rho,\zeta}(1)$ for fairly large $N$ are convincing evidence for the full conjecture.

\begin{figure}[h!]
\centering
  \begin{tikzpicture}

    \node[inner sep=2pt] at (0cm, 0cm)
      {\includegraphics[width=5cm]{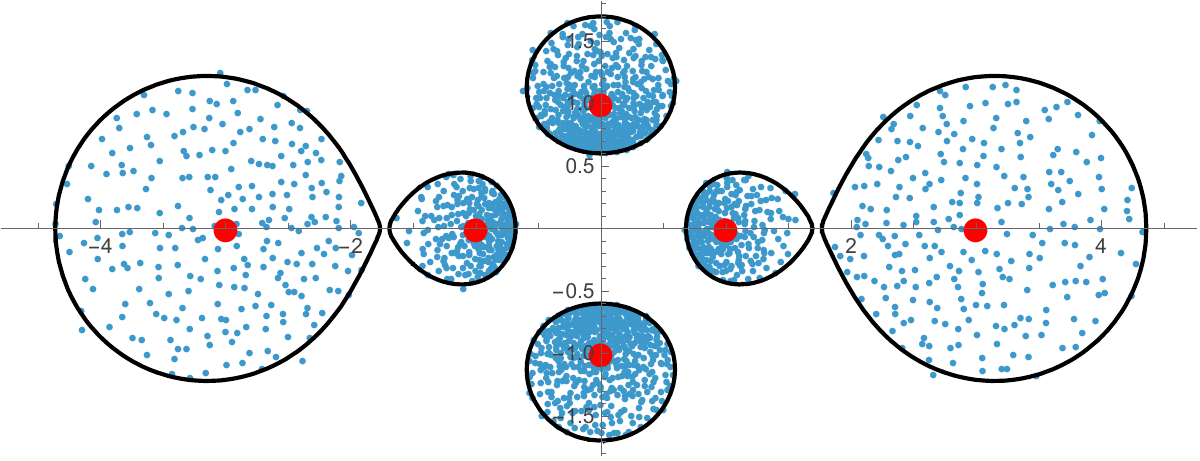}};

    \node[inner sep=2pt] at (5.5cm, -0cm)   
      {\includegraphics[width=5cm]{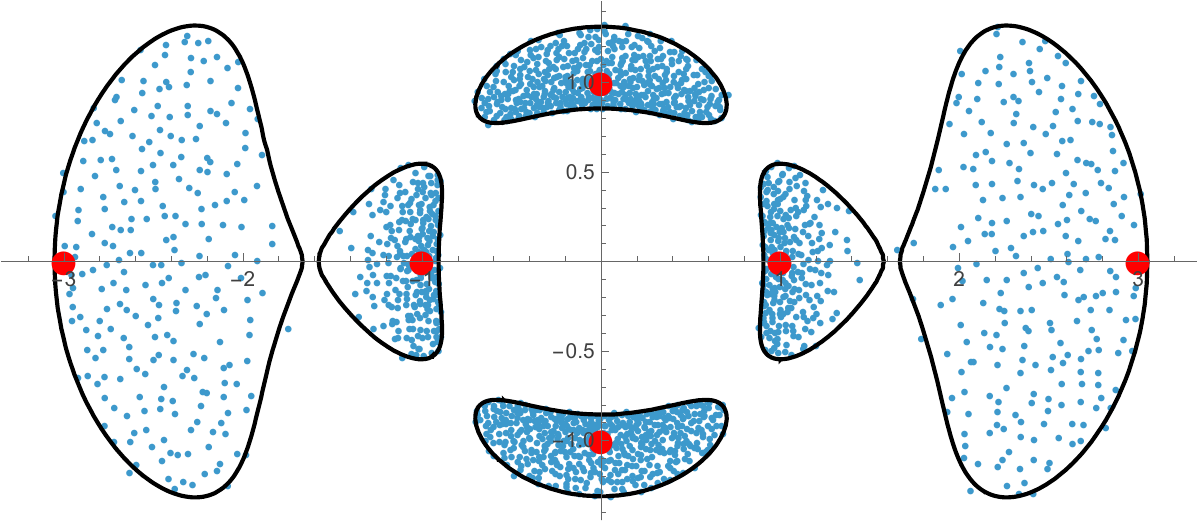}};

    \node[inner sep=2pt] at (0cm, -2.5cm)
      {\includegraphics[width=5cm]{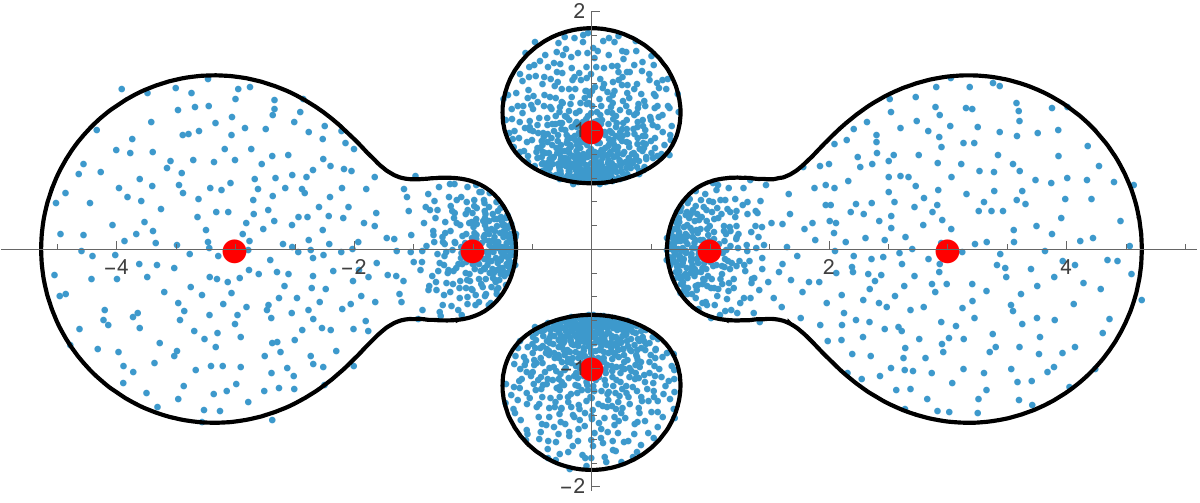}};

    \node[inner sep=2pt] at (5.5cm, -2.5cm)  
      {\includegraphics[width=5cm]{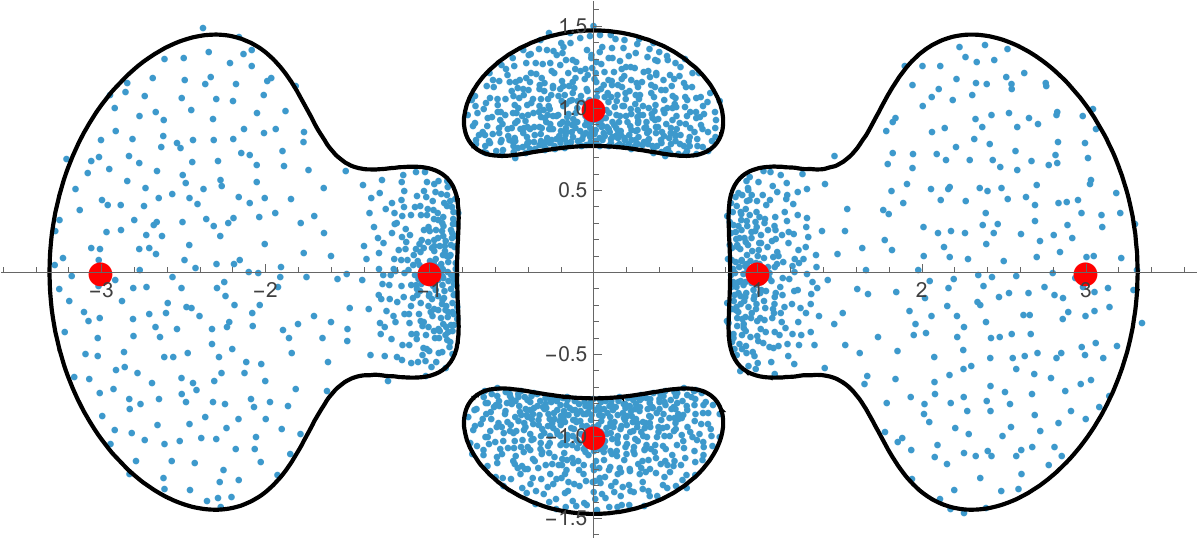}};
      
    \node[inner sep=2pt] at (0cm, -5cm)
      {\includegraphics[width=5cm]{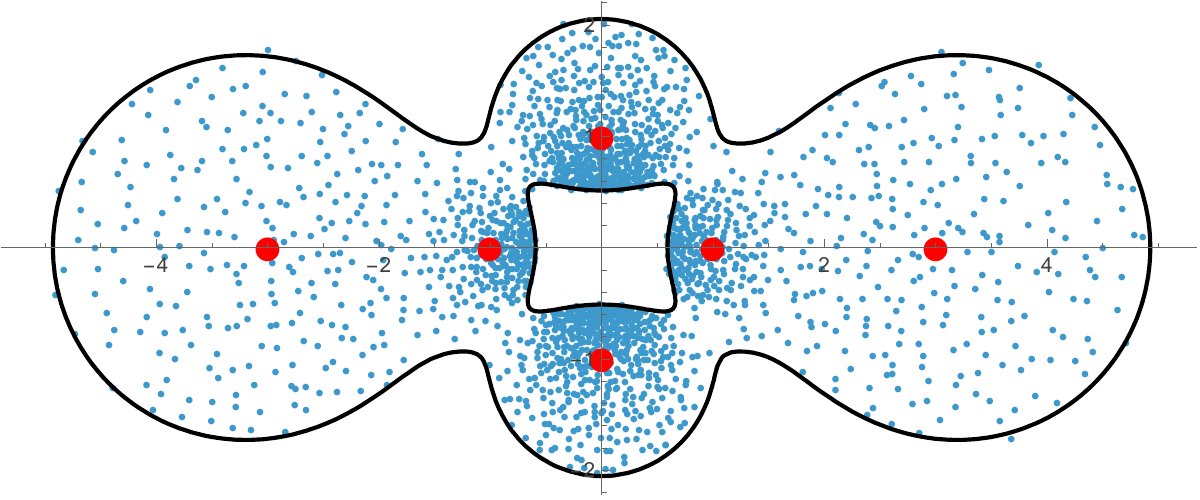}};

    \node[inner sep=2pt] at (5.5cm, -5cm)  
      {\includegraphics[width=5cm]{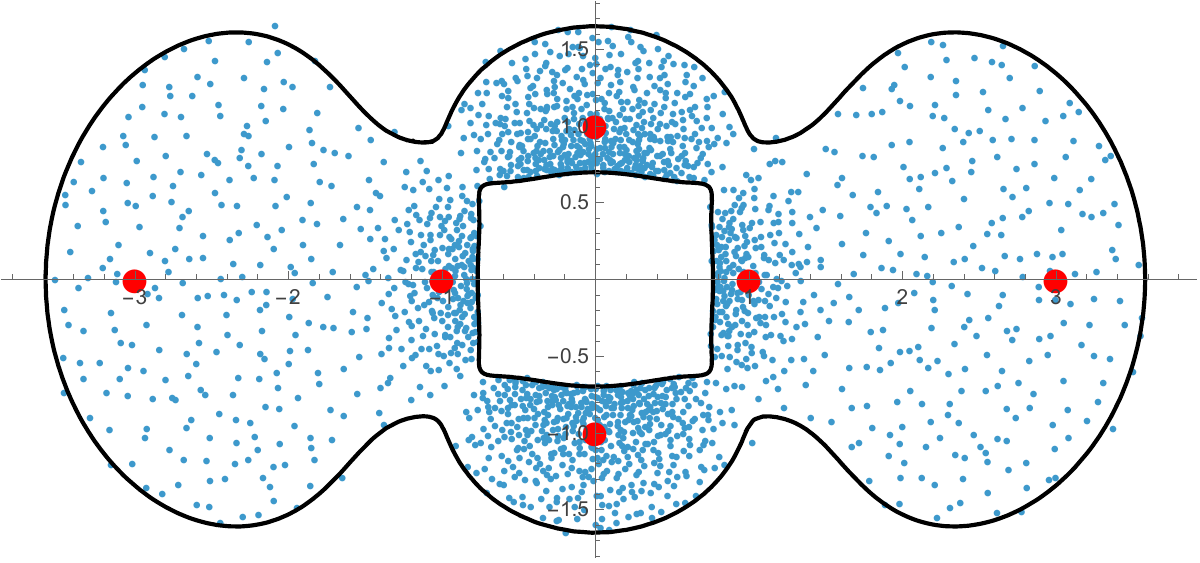}};
  \end{tikzpicture}
  \caption{Eigenvalues of $B_0B(t,\zeta)$ with $N=2000$. From top to bottom, $t$ ranges through $0.8,1,1.2$; $\zeta=0$ on the left and $\zeta=-0.5$ on the right.  In all cases, $B_0$ is a normal matrix with empirical eigenvalue distribution $\frac18(\delta_1+\delta_{-1}+\delta_3+\delta_{-3})+\frac14(\delta_{\i}+\delta_{-\i})$.
  Also plotted are the boundaries $\partial\Sigma(b_0,t,\zeta)$ from \eqref{eq.Sigma.b0}.
  \label{fig-6-blobs}}
\end{figure}

\section{Preliminaries and Background}

\subsection{Basic Constructions from Free Probability\label{sect.freeprob}} 

We begin by reviewing the following definitions and basic results from non-commutative (and especially free) probability.

The basic arena is a {\bf $W^\ast$-probability space}: a pair $(\A,\tau)$ consisting of a von Neumann algebra $\A$ (a $\ast$-algebra of bounded operators on a Hilbert space, that is closed under the weak operator topology), and $\tau\colon\A\to\C$ is a tracial ($\tau(ab) = \tau(ba)$) state ($\tau(1_\A)=1$) that is faithful ($\tau(a^\ast a)=0$ iff $a=0$) and $\sigma$-weakly continuous.  The core motivating example is $\A = L^\infty(\Omega,\mathscr{F},\P)$, the algebra of bounded random variables on a probability space, with $\tau = \E$ the expectation with respect to $\P$. Indeed, all commutative $W^\ast$-probability spaces have this form; the general $\sigma$-weak continuity assumption is in place precisely to verify the Dominated Convergence Theorem in general.

More relevant to us presently is the example $(\M_N(\C),\ts_N)$: the algebra of $N\times N$ complex matrices equipped with the {\bf normalized trace} $\ts_N =\frac1N\tr_N$, with the usual matrix trace $\tr_N[A] = \sum_{j=1}^N A_{jj}$.  We will also frequently consider the ``product'' of these two examples: the algebra $L^\infty(\Omega,\mathscr{F},\P;\M_N(\C))$ of random matrices (with bounded entries), equipped with the expected normalized trace $\E\ts_N$ -- or, when relevant, we use only the trace $\ts_N$ (in which case $\ts_N[A]$ is $\C$-valued random variable rather than just a scalar).

Not content to always insist on bounded random variables, we can generally extend a $W^\ast$-probability space to include random variables with $p$ finite moments.  For any $a\in\A$, $a^\ast a$ is positive semidefinite and so the Spectral Theorem can be used to define $|a|^p:= (a^\ast a)^{p/2}$.  The $L^p$-norm is then defined
\[ \norm{a}_p:= \left(\tau[|a|^p]\right)^{1/p}, \qquad 1\le p<\infty. \]
The usual convexity proof demonstrates the triangle inequality, and the faithfulness of $\tau$ means it is a genuine norm.  The trace property implies that $\norm{a}_p = \norm{a^\ast}_p$.

The {\bf non-commutative $L^p$-space} $L^p(\A,\tau)$ is defined to be the closure of $\A$ in the norm $\norm{\,\cdot\,}_p$. The state $\tau$ satisfies $|\tau[a]|\le \tau[|a|]$, which means that it extends to a linear functional on $L^1(\A,\tau)$ (which we also denote $\tau$).  The $L^p$-norms satisfy H\"older's inequality: if $1<p<\infty$ and $\frac{1}{p}+\frac{1}{p'}=1$, then for $a\in L^p(\A,\tau)$ and $b\in L^{p'}(\A,\tau)$,
\[ |\tau(ab)| \le \|ab\|_1 \le \|a\|_p\|b\|_{p'}. \]
In particular, since $\|1_\A\|_p=1$ for all $p$, it follows that $\|a\|_p^p = \tau[|a|^p\cdot 1_\A] \le \||a|^p\|_r\|1_\A\|_{r'} = \|a\|_{pr}^p$ for $r\ge 1$; i.e.\ $p\mapsto \|a\|_p$ is increasing, so that $L^p(\A,\tau)\subseteq L^q(\A,\tau)$ for $q\le p$.  Because $\tau$ is faithful, it follows that the operator norm on $\A$ satisfies $\|a\| = \lim_{p\to\infty} \|a\|_p$; so we may set $\|a\|_\infty :=\|a\|$ and extend H\"older's inequality fully to $1\le p\le\infty$, identifying $L^\infty(\A,\tau)=\A$.

We regard the elements of a $W^\ast$-probability space as (generalized) random variables; as such, we must define their distributions.  Classically, a (joint) probability distribution of random variables is determined by testing expectations of functions on them.  Owing to the non-commutativity, the natural class of test functions readily available is {\bf non-commutative polynomials}.

\begin{notation} \label{def.ncpoly}
	We denote $\PP_d = \C\langle X_1,\dots,X_d\rangle $ the set of non-commutative polynomials in $d$ variables. In this paper, $X_i$ denote generic non-commuting indeterminates for polynomials. A given $P\in \PP_d$ may be evaluated at a $d$-duple $\bs{a}=(a_1,\ldots,a_d)\in\A^d$ in $W^*$-probability space $\A$ to produce a new element $P(\bs{a})\in\A$.
\end{notation}

\begin{defi} \label{def.*-dist}
Let $\bs{a} = (a_1,\ldots ,a_d)$ be a $d$-tuple of non-commutative random variables in some $W^\ast$-probability space $(\A,\tau)$. The {\bf $\ast$-distribution} of $\bs{a}$ is the linear functional $\mu_{\bs{a}}\colon\PP_{2d}\to\C$ given by $P \mapsto \tau\left[ P(\bs{a}, \bs{a}^\ast) \right]$.

If $(\A_N,\tau_N)$ is a sequence of $W^\ast$-probability spaces for $N\ge 1$, and $\bs{a}_N \in\A_N$ are $d$-tuples, we say $\bs{a}_N$ {\bf converges in $\ast$-distribution to $\bs{a}$} if the $\ast$-distributions $\mu_{\bs{a}_N}$ converges pointwise to $\mu_{\bs{a}}$; i.e.\ if, for all $P\in\PP_{2d}$
\[ \lim_{N\to\infty} \tau_N[P(\bs{a}_N,\bs{a}_N^\ast)] = \tau[P(\bs{a},\bs{a}^\ast)]. \]
We say that $\bs{a}_N$ {\bf converges strongly to $\bs{a}$} if, in addition to converging in $\ast$-distribution, also
\[ \lim_{N\to\infty} \|P(\bs{a}_N,\bs{a}_N^\ast)\| = \|P(\bs{a},\bs{a}^\ast)\|. \]
\end{defi}
In this paper, Definition \ref{def.*-dist} will be used for the sequence of matrix $W^\ast$-probability spaces, i.e.\ $(\A_N,\tau_N)=(\M_N,\ts_N)$, typically with random elements.  As such, the above notions of convergence must be applied randomly as well -- in expectation, in probability, or almost surely (with the latter requiring that we sample all entries of all matrices in $\M_N(\C)$ for all $N$ from a single probability space).

If $a\in\A$ is a single {\em normal} element, $aa^\ast = a^\ast a$, and if $P\in\PP_2$, then $P(a,a^\ast)$ reduces to a linear combination of monomials of the form $a^n a^{\ast m}$.  Since $a$ is bounded, there is a unique compactly-supported Borel probability measure $\mu_a$ on $\C$ determined by
\begin{equation} \label{eq.spectral.meas} \tau[a^n a^{\ast m}] = \int \lambda^n\bar\lambda^m\,\mu_a(d\lambda). \end{equation}
The support of $\mu_a$ is precisely the spectrum $\sigma(a)$, and thence the $\ast$-distribution of $a$ is determined entirely by $\mu_a$ in the classical fashion:
\[ \tau[P(a,a^\ast)] = \int_\C P(\lambda,\bar\lambda)\,\mu_a(d\lambda). \]
This measure is called the {\bf spectral distribution} of $a$.  In the particular case of a normal matrix $A\in\M_N(\C)$, the spectral distribution is none other than the empirical law of eigenvalues
\begin{equation}\label{eq.spectral.meas.ESD} \mu_{A} = \sum_{j=1}^N \delta_{\lambda_j(A)}. \end{equation}
Note, however, that for non-normal $a$, the $\ast$-distribution is no longer encoded by a probability measure; similarly, for a tuple $\bs{a}$ of $\ge 2$ random variables that do not commute with each other, there is no reduction of the $\ast$-distribution to a measure on $\C^d$ as occurs in the classical (commutative) case.  The ESD \eqref{eq.spectral.meas.ESD} of course always exists, and more generally so does the Brown measure (also denoted $\mu_a$) introduced in Section \ref{sect.Brown.meas}.  But the connections between $\mu_a$ and the $\ast$-distribution of $a$ is tenuous at best when $a$ is not normal.

We now come to the framework that gives free probability its name.

	\begin{defi} 	\label{2freeprob}	
		Let $(\A,\tau)$ be a $W^\ast$-probability space, and let $\A_1,\dots,\A_n$ be $*$-subalgebras of $\A$, sharing the same unit $1_\A$. They are said to be {\bf freely independent} or just {\bf free} if, for all $k$, for all $a_i\in\A_{j_i}$ such that $j_1\neq j_2$, $j_2\neq j_3$, \dots , $j_{k-1}\neq j_k$:
        \begin{equation}
        \label{skdjvnsksovdn}
            \tau\Big( (a_1-\tau(a_1))(a_2-\tau(a_2))\dots (a_k-\tau(a_k)) \Big) = 0.
        \end{equation}
        I.e.\ products of centered elements from adjacent-alternating algebras have trace $0$.
		Families of non-commutative random variables are said to be free if the $*$-subalgebras they generate are free.
    \end{defi}

Freeness does not occur in finite-dimensional matrix algebras, or classical $L^\infty$ probability settings, except in fairly degenerate settings (e.g.\ constants are always freely independent from all other random variables).  However, a plethora of examples from random matrix theory exhibit \textit{asymptotic} freeness, i.e. as dimensions of the matrices tend to $\infty$, they become free. We state a classical example of this phenomenon below.

\begin{defi} \label{def.Haar.unitary} A {\bf Haar Unitary Ensemble} $U^N$ is a random $N\times N$ matrix whose distribution is the Haar measure on $\mathrm{U}(N)$.  In a $W^\ast$-probability space $(\A,\tau)$, a {\bf Haar unitary operator} $u\in\A$ is a unitary operator whose $\ast$-distribution is determined by $\tau[u^n] = \tau[u^{\ast\,n}]=\delta_{n0}$ for all $n\in\N$. \end{defi}

\begin{prop}[\cite{Voiculescu1991}; \cite{Speicher1993}] \label{prop.Haar.conj} 
Let $A_0^N,A_1^N$ be random matrix ensembles, i.e.\ random elements in the $W^\ast$-probability space $(\M_N(\C),\ts_N)$.  Let $(\A,\tau)$ be a $W^\ast$-probability space, with elements $a_0,a_1\in\A$ such that $A_0^N\to a_0$ a.s.\ in $\ast$-distribution and $A_1^N\to a_1$ a.s.\ in $\ast$-distribution.  For each $N\in\N$, let $U^N$ be a Haar Unitary Ensemble independent from $\{A_0^N,A_1^N\}$, and let $u$ be a Haar unitary operator in $\A$ freely independent from $\{a_0,a_1\}$.  Then $(A_0^N,U^NA_1^NU^{N\,\ast})$ converges a.s.\ in joint $\ast$-distribution to $(a_0,ua_1u^\ast)$; $ua_1u^\ast$ has the same distribution as $a_1$; and $a_0,ua_1u^\ast$ are freeely independent.  I.e.\ $A_0^N$ and $U^N A_1^N U^{N\,\ast}$ are a.s.\ asymptotically free.
\end{prop}

Now we introduce important families of free random variables that are large-$N$ limits of Gaussian random matrix ensembles and stochastic processes.

\begin{defi} \label{smdvskjvsdhg} A family of non-commutative random variables $ x=(x_1,\dots ,x_m)$ in a $W^*$-probability space $(\A,\tau)$ is called 
		a {\bf free semicircular system} if $x_1,\ldots,x_m$ are freely independent, 
		selfadjoint ($x_i=x_i^*$, $i=1 \dots m$), and for all $k\in\N$ and $i=1,\dots,m$ one has
		\begin{equation}
        \label{slkdnvsldnmv}
			\tau( x_i^k) =  \int t^k \,\varsigma(dt),
		\end{equation}
		with $\varsigma(dt) = \frac 1 {2\pi} \sqrt{4-t^2} \ \mathbbm{1}_{|t|\leq2}\, dt$ the semicircle distribution. For an explicit construction of such operators, we refer to \cite[Lecture 7]{SpeicherNicaBook}.
        
        A {\bf free semicircular Brownian motion} $(x(t))_{t\geq 0}$ is a family of semicircular variables such that $x(t)/\sqrt{t}$ has the semicircle distribution for each $t>0$, and the additive increments of $x$ are stationary and freely independent: for $0 \leq t_1 < t_2 < \infty$, $x(t_2)-x(t_1)$ has the same law as $ x(t_2-t_1)$, and $x(t_2) - x(t_1)$ is freely independent from $(x(t))_{0\leq t\leq t_1}$.   See \cite{BianeSpeicher1998} for a detailed construction and properties.
\end{defi}

We also review relevant classical random matrix models.

\begin{defi}
	\label{2HBdef}
	A {\bf Hermitian Brownian motion} $(X^N(t))_{t\ge 0}$ of size $N$ is the Brownian motion on the vector space of Hermitian matrices in $\M_N(\C)$, with respect to the scaled Hilbert--Schmidt inner product \eqref{eq.HS.innprod}.  Entry-wise, it can be described as follows:
	\begin{itemize}
		\item For $1\leq i\leq N$, the random variables $\{\sqrt{N} [X^N(t)]_{i,i}\}_{t\ge 0}$ are independent standard Brownian motions on $\R$.
		\item For $1\leq i<j\leq N$, the random variables $\{\sqrt{2N}\,\mathrm{Re}[X^N(t)]_{i,j}\}_{t\ge 0}$ and $\{\sqrt{2N}\,\mathrm{Im}[X^N_t]_{i,j}\}_{t\ge 0}$ are independent standard Brownian motions on $\R$, independent from $\{[X^N(t)]_{i,i}]\}_{t\ge 0}$.
	\end{itemize}
    A random matrix $Y^N$ whose law is equal to that of $X^N(1)$ is called a {\bf Gaussian Unitary Ensemble}, or $\mathrm{GUE}$ for short; its law can be described by its density with respect to Lebesgue measure on Hermitian matrices, which is proportional to $X\mapsto \exp(-N\mathrm{Tr}_N[X^2])$.  This density is invariant under unitary conjugation, hence the name.
    The process $(X^N(t))_{t\ge 0}$ is sometimes called a $\mathrm{GUE}$ Brownian motion.
\end{defi}

\begin{defi}    \label{def:GinibreBM}
    A {\bf Ginibre Brownian motion} of size $N$ is an $\M_N(\C)$-valued diffusion process $Z^N(\cdot)$ of the form
    \[ Z^N(t) \equaldist \frac{1}{\sqrt{2}}(X^N(t) + \i Y^N(t) ) \]
    where $X^N,Y^N$ are independent Hermitian Brownian motions of size $N$.  Entrywise, this can be described simply: for $1\le i,j\le N$, 
    $\{\sqrt{2N}\,\mathrm{Re}Z_{i,j}^N, \sqrt{2N}\,\mathrm{Im}Z_{i,j}^N\}_{t\ge 0}$ are all independent standard Brownian motions on $\mathbb{R}$.

    The random matrix $Z^N(1)$, whose entries are all i.i.d.\ complex normal random variables, is called a {\bf Ginibre Ensemble} first introduced in \cite{Ginibre1965}. Equivalently, the law of $Z$ has a density with respect to Lebesgue measure on $\M_N(\C)$ which is proportional to $Z\mapsto \exp(-N\mathrm{Tr}_N[Z^\ast Z])$.  Note that this density is invariant under $Z\mapsto UZV^\ast$ for any deterministic unitary matrices $U,V\in\mathrm{U}(N)$; thus, the law of the Ginibre ensemble is unitarily bi-invariant.
\end{defi}

        We will often need to consider free semicircular variables and matrices in the same space.
        To do so, let $(x_{u,v})_{1\leq u\leq v\leq N}$ and $(y_{u,v})_{1\leq u< v\leq N}$ be free semicircular variables in a $W^*$-probability space $(\A,\tau)$, and set 
        $$ x^N_{u,v} = \left\{ \begin{array}{cc}
           \displaystyle{\frac{x_{u,v}+\i y_{u,v}}{\sqrt{2N}}}& \text{ if } u< v  \\
           \displaystyle{\frac{x_{u,v}}{\sqrt{N}}} & \text{ if } u= v  \\
           \displaystyle{\frac{x_{v,u}-\i y_{v,u}}{\sqrt{2N}}} & \text{ else.}
        \end{array} \right. $$
        Then if we set $\tau_N \deq \tau\circ\ts_N$, one has that the trace of the $k$-th moment of $x^N$ is equal to the right hand side of Equation \eqref{slkdnvsldnmv}. Thus $x^N$ is a free semicircular variable. Let us now define $x_i^N$ just like $x^N$ but where each $x_i^N$ is generated with freely independent semicircular systems. Next, we set $\A_N$ to be the $W^*$-subalgebra of $\M_N(\A)$ generated by $\M_N(\C)$ and $(x_i^N)_{1\leq i\leq d}$ and we endow $\A_N$ with $\tau_N$. We then have the following lemma.

        \begin{lemma}
        \label{lem:bigspace}
            With $(\A_N,\tau_N)$ built as above, we have that $(x_i^N)_{1\leq i\leq d}$ is a free semicircular system, free from $\M_N(\C)$.
        \end{lemma}
        
        Thus one can view $(\A_N,\tau_N)$ as the free product of $\M_N(\C)$ with the $W^\ast$-algebra generated by a free semicircular system. For the proof see \cite[Proposition 2.3]{huit}. 
        
        In particular, since the law of $(x_i^N)_{1\leq i\leq d}$ does not depend on $N$, we will often simply denote it $(x_i)_{1\leq i\leq d}$.  Note that when restricted to $\M_N(\C)$, $\tau_N$ is just the regular normalized trace on matrices. Similarly, the restriction of $\tau_{N}$ to the $W^\ast$-algebra generated by $(x_i^N)_{1\leq i\leq d}$ does not depends on $N$ and hence will simply be denoted by $\tau$.

\subsection{Elliptical Brownian motions, $\mathrm{GL}(N,\C)$, and the Large-$N$ Limit\label{sect.background.BMs}}

We summarize the constructions in Section \ref{sect.intro.BMs} here, with further discussion of their large-$N$ limits.

\begin{defi}[Elliptical Brownian motions] \label{def:elliptical}
Let $\rho>0$ and let $\zeta\in\C$ with $0<|\zeta|\le\rho$.  Define
\begin{equation}
    \label{parametrization}
    \theta = \frac12\arg \zeta \qquad a = \sqrt{\frac12(\rho+|\zeta|)} \qquad b = \sqrt{\frac12(\rho-|\zeta|)}. 
\end{equation}
Then $a^2+b^2=\rho$ and $e^{2\i\theta}(a^2-b^2)=\zeta$, so that $(a,b,\theta)$ satisfy \eqref{eq.rho.zeta}.  This yields a bijective mapping
\begin{equation}\label{eq.bijection} \{(\rho,\zeta)\colon 0<|\zeta|\le\rho\}\mapsto\{(a,b,\theta)\colon a>b\ge 0, \theta\in[0,\pi)\}. \end{equation}
We may extend the definitions of $a,b$ to allow $\zeta=0$, in which case $a=b$; here $\theta$ is undefined.

Let $(X^N(t))_{t\geq 0}$, $(Y^N(t))_{t\geq 0}$ be independent Hermitian Brownian motions of size $N$. 
We define the general {\bf elliptical Brownian motion} $(W^N_{\rho,\zeta}(t))_{t\ge0}$ by
	\[ W^N_{\rho,\zeta}(t) = e^{\i \theta} \left(aX^N(t) + \i b Y^N(t)\right). \]
    Similarly, if $(x(t))_{t\geq 0}$ and $(y(t))_{t\geq 0}$ are freely independent free semicircular Brownian motion, the {\bf free elliptical Brownian motion $(w_{\rho,\zeta}(t))_{t\ge0}$} is defined as
    \[ w_{\rho,\zeta}(t) = e^{\i \theta} \left(a x(t) + \i b y(t)\right). \]
In both cases, when $a=b$ (i.e.\ $\zeta=0$), all values of $\theta$ yield the same distribution for the right-hand-side, so we may take canonically $\theta=0$ in this case.

Note that the laws of $W^N_{\rho,\zeta}(t)$ and $w_{\rho,\zeta}(t)$ are invariant under $a\mapsto-a$, $b\mapsto-b$, and $(a,b,\theta)\mapsto(b,a,\theta+\frac{\pi}{2})$.  It then follows from \eqref{eq.bijection} that $\rho>0$, $|\zeta|\le\rho$ paramaterize the laws of all possible elliptical Brownian motions.

To lighten notation we will often drop the subscripts $\rho,\zeta$ if there can be no confusion. 
\end{defi}
Of particular note is the case $(\rho,\zeta)=(1,0)$ (corresponding to $a=b=1$ with $\theta$ undetermined); here $W^N$ is a Ginibre Brownian motion.

As we proceed to consider putative large-$N$ limits, it is useful to disambiguate the dimension of the Brownian motion processes.

\begin{notation} The processes above will be indexed by $N$ as needed.  Moreover, the parameters $(\rho,\zeta)$ will remain fixed throughout most of this paper.  As such, for readability, we often suppress the explicit dependence of the notation for the Brownian motion processes on the parameters, and define
\begin{align*}
W_{\rho,\zeta}(t) &= W_{\rho,\zeta}^N(t) = W^N(t) = W(t) \in\M_N(\C) \\
B_{\rho,\zeta}(t) &= B_{\rho,\zeta}^N(t) = B(t) = B(t) \in \mathrm{GL}(N,\C).
\end{align*}
\end{notation}

\begin{defi}\label{def.SDE}
	We define the $(\rho,\zeta)$-Brownian motion $(B_{\rho,\zeta}^N(t))_{t\geq 0}$ as the solution of the SDE \eqref{eq.B.SDE.rho.zeta}, restated here:
	\[ B_{\rho,\zeta}^N(t) = I_N + \int_{0}^{t} B_{\rho,\zeta}^N(s)\, dW^N_{\rho,\zeta}(s) + \frac{\zeta}{2} \int_{0}^{t} B_{\rho,\zeta}^N(s)\,ds.\]
	We also define its free version $b_{\rho,\zeta}$ by replacing the independent Hermitian Brownian motions in the SDE above by freely independent self-adjoint free Brownian motions. To lighten notation we will often drop the subscripts $\rho,\zeta$. 
    \end{defi}

\begin{rem}[Other Parametrizations in the Literature] Our SDE in \eqref{eq.B.SDE.rho.zeta} and Definition \ref{def.SDE} has a slight parametrization change from some papers in the literature dealing with the same Brownian motion processes.
\begin{itemize}
    \item In Biane's original papers \cite{Biane1997b,Biane1997JFA} where the free multiplicative Brownian motion $\Lambda$ was introduced, the given SDE was $d\Lambda(t) = \i\, dZ(t)\Lambda(t)$, where $Z$ was a circular Brownian motion (i.e.\ our $w_{1,0}$).  Since $Z$ and $\i Z$ have the same law, and right vs left Brownian motions have the same law, Biane's $\Lambda$ is the same as our $b_{0,1}$.
    \item In \cite{Kemp2016,Kemp2017,KempCebron2022,HallHo2023,HallHo2025Spectrum,BCC2025}, the processes $W_{\rho,\zeta}$ and $B_{\rho,\zeta}$ differ in parametrization from ours in two ways: $W\leftrightsquigarrow \i W$, and the parameters $(\rho,\zeta)$ are $(s,\tau)$ where $s=\rho$ and $\tau = \rho-\zeta$; in particular, the Brownian motions are indexed by $(s,\tau)\in\R_+\times\C$ with $|s-\tau|\le s$.  This relates to the origin of the complex parameter in \cite{DHK2020}, where it arises in the complexification of time in the Segal--Bargmann transform.
    \item The free elliptical Brownian motions $w_{\rho,\zeta}(t)$ are discussed in \cite{Zhong2025}, using the same parametrization as ours, with $\zeta$ renamed $\gamma$.
\end{itemize}
\end{rem}

The ``flat'' Brownian motions $W^N = W^N_{\rho,\zeta}$ are (by design) invariant under unitary conjugation, i.e.\ 
for a fixed $U\in\mathrm{U}(N)$ $(UW^N(t)U^*)_{t\ge0}$ has the same law (as a process) as $(W^N(t))_{t\ge0}$.  It follows that the same holds true for the process $(B^N(t))_{t\ge0}$.  A straightforward calculation with the SDE (Definition \ref{def.SDE}) shows that the Brownian motions $(B^N_{\rho,\zeta}(t))_{t\ge0}$ are also $\mathrm{U}(N)$-conjugation invariant:
\begin{equation} \label{eq.B.conj.inv} UB^N_{\rho,\zeta}(t)U^\ast \equaldist B^N_{\rho,\zeta}(t) \qquad \forall\; t\ge 0,\, \rho>0,\, |\zeta|\le\rho,\, U\in\mathrm{U}(N).
\end{equation}

For any $t\ge 0$, the free multiplicative Brownian motion $b_{\rho,\zeta}(t)$ is the large-$N$ limit in $\ast$-distribution of the $\mathrm{GL}(N,\C)$ Brownian motion $B^N_{\rho,\zeta}(t)$; this was proved for $\zeta\in(-\rho,\rho)$ in \cite{Kemp2016} and generalized to complex $\zeta$ with $|\zeta|\le\rho$ (and strong convergence) in \cite{BCC2025}. We need this convergence paired with an arbitrary initial condition, stated in Lemma \ref{lem.init.cond.*-conv}; we now provide the proof.

\begin{proof}[Proof of Lemma \ref{lem.init.cond.*-conv}]
Fix $t\ge 0$.  \cite{Cebron2013,Kemp2016,BCC2025} proved that $B^N(t)\to b(t)$ in $\ast$-distribution; by assumption of the Lemma, $B_0^N\to b_0$ in $\ast$-distribution. Hence, by Proposition \ref{prop.Haar.conj}, $(B_0^N,U^N B^N(t)U^{N\,\ast})$ converge a.s.\ in $\ast$-distribution to $(b_0,ub(t)u^\ast)$.

Now, let $f\in C_b(\M_N(\C))$, and note that
\begin{align*} &\E\left[f(B_0^N,U^NB^N(t)U^{N\,\ast})\right] \\
=& \int_{\mathrm{U}(N)} \mathrm{Haar}(dU)\int_{\M_N(\C)} \mathrm{Law}_{B_0^N}(dB_0) \int_{\M_N(\C)} \mathrm{Law}_{B^N(t)}(dB_1)\;f(B_0,UB_1U^\ast) \end{align*}
where we have used the change-of-variables formula and the independence of $B_0^N$ and $B^N(t)$.
In the inside integral, $U$ is a fixed unitary matrix; hence, by \eqref{eq.B.conj.inv}, the law of $B^N(t)$ is invariant under conjugation by $U$.  Thus
\[ \int_{\M_N(\C)} \mathrm{Law}_{B^N(t)}(dB_1)\;f(B_0,UB_1U^\ast) = \int_{\M_N(\C)} \mathrm{Law}_{B^N(t)}(dB_1)\;f(B_0,B_1) \]
and reassembling this establishes that
\[ \E\left[f(B_0^N,U^NB^N(t)U^{N\,\ast})\right] = \E\left[f(B_0^N,B^N(t))\right]. \]
Since this holds for all $f\in C_b(\M_N(\C))$, it follows that $(B_0^N,U^NB^N(t)U^{N\,\ast}) \equaldist (B_0^N,B^N(t))$.  Thus
\[ (B_0^N,B^N(t)) \equaldist (B_0^N,U^N B^N(t) U^{N\,\ast}) \to (b_0,ub(t)u^\ast)\equaldist(b_0,b(t))\quad \text{a.s.} \]
where the last equality in $\ast$-distribution comes from the fact that $ub(t)u^\ast$ has the same $\ast$-distribution as $b(t)$, and the pair $b_0,b(t)$ is already freely independent.  Thence, it follows that $B_0^NB^N(t)$ converges a.s.\ in $\ast$-distribution to $b_0b(t)$, as claimed.
\end{proof}

\subsection{The Schwinger--Dyson Equations\label{sect.Schwinger-Dyson}}

One of the core computational toolsets we use are the Schwinger--Dyson equations, which (in our context) are free probability analogues of Gaussian integration by parts for matrix-valued functions. They are stated, for us, in terms of non-commutative polynomials (Notation \ref{def.ncpoly}) and non-commutative derivatives, discussed next.

\begin{defi} \label{def.ncpoly.*} On the algebra $\PP_d = \C\langle X_1,\ldots,X_d\rangle$,  we define an involution $*$ by $X_i^*= X_i$ extended to $\PP_d$ by real linearity and the product rule $(\alpha P Q)^* = \overline{\alpha}\,Q^*P^*$ for $\alpha\in\C$ and $P,Q\in\PP_d$.  A polynomial $P\in\PP_d$ is called {\bf self-adjoint} if $P=P^\ast$.
\end{defi}

The above $\ast$-algebra structure on $\PP_d$ is compatible with evaluation at self-adjoint random variables in a $W^\ast$-probability space, which will be our focus for the constructions in the rest of this Section and the following one.  Note: if $P\in\PP_d$ is self-adjoint and $\bs{y} = (y_1,\ldots,y_d)$ are all self-adjoint then the random variable $P(\bs{y})$ is also self-adjoint.

\begin{defi}
\label{slkdvnsldkjg}
	
	If $1\leq j\leq d$, the {\bf non-commutative derivatives} $\partial_j: \PP_d \to \PP_d^{\otimes 2}$ are defined first on monomials $M$ by
	$$ \partial_j M = \sum_{M=AX_jB} A\otimes B ,$$
	and then extended by linearity to all polynomials.  We then define the {\bf cyclic derivatives} $D_j = m\circ \partial_j$ with $m : A \otimes B \mapsto BA$.
\end{defi}
For proof of the next proposition see for instance \cite[Proposition 2.10]{CGP2022}.

\begin{prop}[Schwinger--Dyson Equations]
    \label{prop:SDeq}
    Let $Q\in \PP_d$ and $1\le j\le d$.
    \begin{enumerate}
        \item[(a)]
        If $\bs{X}^N = (X^N_1,\ldots,X^N_d)$ is a $d$-tuple of independent GUE random matrices in $\M_N(\C)$, then 
        \begin{equation}    \label{SDeq-mat}
            \E\left[\ts_N(X_j^N Q(\bs{X}^N))\right] = \E\left[\ts_N\otimes\ts_N\left((\partial_jQ)(\bs{X}^N)\right)\right],
        \end{equation}
        where we recall $\ts_N$ is the normalized trace. 
        \item[(b)]
        If $\bs{x}=(x_1,\dots, x_d)$ is a $d$-tuple of freely independent semicircular random variables, 
        \begin{equation}    \label{SDeq-op}
            \tau(x_j Q(\bs{x})) = \tau\otimes\tau\left((\partial_jQ)(\bs{x})\right).
        \end{equation}
    \end{enumerate}
\end{prop}

One of the core technical ingredients of this paper is a method for interpolating directly between random matrices and free variables that yields a powerful quantitative interpretation and extension of genus expansions for functions of GUEs.  The method was introduced in \cite{Parraud2022,Parraud2023}, and used in \cite{BCC2025,neuf,sept,huit} for various applications.  A combinatorial approach was recently explored in \cite{Jekel-ALEA}, explaining more directly the connection to the classical genus expansion.  We now briefly outline how the first $\mathcal{O}(\frac{1}{N^2})$ term in the expansion arises, yielding a fourth-order differential operator $L$ whose iteration underlies the full expansion.

Let $\bs{X}^N$ be a $d$-tuple of independent GUE matrices in $\M_N(\C)$, and let $\bs{x}$ be a $d$-tuple of freely independent semicircular random variables; as in Lemma \ref{lem:bigspace}, we can realize $\bs{X}^N$ and $\bs{x}$ together in the same $W^\ast$-probability space $(\A_N,\tau_N)$, freely independent from each other.  For any non-commutative polynomial $Q\in\PP_d$, $\tau_N(Q(\bs{X}^N))-\tau_N(Q(\bs{x}))\to 0$ as $N\to\infty$.  To get a handle on the rate of convergence, we introduce a smooth interpolation between $\bs{X}^N$ and $\bs{x}$ akin to the Mehler kernel for the Ornstein--Uhlenbeck process:
\[ \bs{z}_t := \sqrt{1-e^{-t}}\bs{x} + e^{-t/2}\bs{X}^N. \]
Note, then, that
\begin{equation} \label{eq.Mehler.1} \E\left[\tau_N(Q(\bs{X}^N))\right] - \tau_N(Q(\bs{x})) = -\int_0^\infty \E\left[\frac{d}{dt}\tau_N(Q(\bs{z}_t))\right]dt. \end{equation}

The expected derivative on the right-hand-side can then be expressed in terms of the generator $L_1$ of the process.  With appropriate use of the Schwinger--Dyson equations \eqref{SDeq-mat} and \eqref{SDeq-op}, we can express
\begin{equation} \label{eq.Mehler.diff} \E\left[\frac{d}{dt}\tau_N(Q(\bs{z}_t))\right] 
= \E\left[e^{-t}\Big(\tau_N\otimes\tau_N - \tau_N\circ(\ts_N\otimes \ts_N)\Big)(L_1Q(\bs{z}_t))\right].
\end{equation}
where $L_1 = \frac12\sum_{i=1}^d \partial_iD_i$.  A key insight in \cite{Parraud2023} was to recognize the difference above as coming from a {\em second} Mehler-like interpolation.  To see how, first note from Definition \ref{slkdvnsldkjg} that $L_1Q\in\PP_d\otimes\PP_d$.  Hence it is a finite linear combination of tensors $P_1\otimes P_2$.  When we evaluate such a tensor at a $d$-tuple $\bs{z}$, we mean $P_1(\bs{z})\otimes P_2(\bs{z})$; we could just as well evaluate the tensor factors at different $d$-tuples $P_1(\bs{z}^1)\otimes P_2(\bs{z}^2)$; we denote such evaluation as $(L_1Q)(\bs{z}^1,\bs{z}^2)$.

Let $\{\bs{z}^\ell\}_{1\le\ell\le 2}$ and $\{\bs{y}^k\}_{1\le k\le 4}$ be six families of $d$-tuples of free semicircular variables freely independent from $\M_N(\C)$, and define
\[ \bs{z}^k_{s,t} = \sqrt{1-e^{-s}}\bs{y}^k + \sqrt{e^{-s}-e^{-t}}\bs{x}^{\lceil k/2 \rceil} + e^{-t/2}\bs{X}^N. \]
Notice that $\bs{z}^k_{0,t} \equaldist \bs{z}^k_{t,t} \equaldist \bs{z}_t$.  Ergo, we may just as well express the right-hand-side of \eqref{eq.Mehler.diff} as
\begin{align*} \E\left[e^{-t}\Big(\tau_N\otimes\tau_N(L_1Q(\bs{z}^1_{t,t},\bs{z}^2_{t,t})) - \tau_N\circ(\ts_N\otimes \ts_N)(L_1Q(\bs{z}^1_{0,t},\bs{z}^2_{0,t}))\Big)\right].
\end{align*}
The free independence of $\bs{z}^1_{s,t}$ and $\bs{z}^2_{s,t}$ allows the factoring of the $\tau_N$ on the right, putting the two terms on the same footing: expressed in terms of $\tau_N\otimes\tau_N$.  Combining this with \eqref{eq.Mehler.1} yields
\begin{align} \nonumber &\; \E\left[\tau_N(Q(\bs{X}^N))\right]-\tau_N(Q(\bs{x})) \\
=& \; -\int_0^\infty e^{-t} \int_0^t\E\left[\frac{d}{ds}\tau_N\left(\ts_N\otimes \ts_N(L_1Q(\bs{z}^1_{s,t},\bs{z}^2_{s,t}))\right)\right]\,ds\,dt \label{eq.Mehler.2} \end{align}

Since $\bs{z}^1_{s,t}$ and $\bs{z}^2_{s,t}$ are copies of another Ornstein--Uhlenbeck type process, the expected derivative in \eqref{eq.Mehler.2} will again be expressible in terms of the generator of this process, which is a second-order (non-commutative) differential operator $L_2$.  Hence, the difference we are estimating in \eqref{eq.Mehler.1} is expressed in term of the {\em fourth order} operator $L=L_2L_1$.  Indeed: again applying judicious use of the Schwinger--Dyson equations \eqref{SDeq-mat} and \eqref{SDeq-op}, \eqref{eq.Mehler.2} yields the explicit formula
\begin{equation} \label{eq.Mehler.3} \E\left[\tau_N(Q(\bs{X}^N))\right]-\tau_N(Q(\bs{x}))
= \frac{1}{N^2}\int_0^\infty\int_0^t e^{-s-t}\E\left[\tau_N\left(\Theta_{s,t}\circ L_2L_1 Q\right)\right] ds dt \end{equation}
where $L_2 = \sum_{j=1}^d \partial_j\otimes\partial_j$ and $\Theta_{s,t}$ is a twisted evaluation map defined as follows.  Noting that $L_2L_1\colon\PP_d\to\PP_d^{\otimes 4}$, $\Theta_{s,t}$ is defined by linearity and 
\[ \Theta_{s,t}(P_1\otimes P_2\otimes P_3\otimes P_4) := P_2(\bs{z}^1_{s,t})P_1(\bs{z}^3_{s,t})P_4(\bs{z}^4_{s,t})P_3(\bs{z}^2_{s,t}). \]
Crucially: the $\frac{1}{N^2}$ arises from \eqref{SDeq-mat} which employs $\partial_j$ to merge two normalized traces into one normalized trace, yielding an extra factor of $1/N$. This happens twice because of a separate merging of free traces $\tau$, which is where the extra interpolating processes $\bs{z}^3_{s,t}$ and $\bs{z}^4_{s,t}$ arise.  See \cite{Felix-Interpolation-Notes} for the full calculation.

Equation \eqref{eq.Mehler.3} gives an explicit $\mathcal{O}(\frac{1}{N^2})$ bound for the difference between the expected trace of $Q(\bs{X}^N)$ and its large-$N$ limit, in terms of the fourth-order non-commutative differential operator $L = L_2L_1 = \frac12\sum_{i,j=1}^d (\partial_j\otimes\partial_j)\circ\partial_iD_i$.  The resultant polynomial $LQ$ is evaluated at interpolated processes $\bs{z}^k_{s,t}$; by comparing these to $\lim_{t\to\infty}\bs{z}^k_{0,t} = \bs{x}$, the above calculations can be iterated, yielding an expansion in powers of $\frac{1}{N^2}$ for $\E[\tau_N(Q(\bs{X}^N))]$ up to any integer order desired.  This expansion coincides with the genus expansion in the case of monomial $Q$ (necessarily by uniqueness), but expresses the coefficients in terms of powers of the operator $L$.  We detail this expansion in \eqref{skjvbsknv} below.

The full details of the iterated expansion are notationally complicated.  In order to minimize the amount of notation, we will not derive it fully here; the interested reader can find all the details in \cite{Parraud2023}, or \cite{Jekel-ALEA} for a more combinatorial approach. We present only the features of the expansion needed ultimately for our proof of Theorem \ref{thm.2}, to be used as a black box as much as possible. We do however need to introduce the following three definitions, expanding on the constructions above.

\begin{defi}
\label{tech1}
    Let us assume that we are given
    \begin{itemize}
        \item $\bs{X}^N=(X^N_1,\dots, X^N_d)$ a $d$-tuple of independent GUE matrices,
        \item $(x^m_i)_{i\in [1,d], m\in \N }$ free semicircular variables,
        \item $I = (I_1,\dots,I_{2n})\in \mathbb{Z}_{>0}^{2n}$ a collection of $2n$ integers,
        \item $T_n = (t_1,\dots,t_{2n})$ a sequence of non-negative real numbers.
    \end{itemize}
    Then with $\widetilde{t}_0=0$ and $\widetilde{t}_1\le \dots\le \widetilde{t}_{2n}$ the elements of $T_n$ sorted in increasing order, we set 
	$$ X_{i,I}^{N,T_{n}} = \sum_{l=1}^{2n} (e^{-\widetilde{t}_{l-1}} -e^{-\widetilde{t}_l})^{1/2} x^{I_{l}}_i + e^{-\widetilde{t}_{2n} /2} X_i^N , \qquad 1\le i\le d\,,$$
	$$ x_{i,I}^{T_{n}} = \sum_{l=1}^{2n} (e^{-\widetilde{t}_{l-1}} -e^{-\widetilde{t}_l})^{1/2} x^{I_{l}}_i + e^{- \widetilde{t}_{2n}/2} x_i^0 \,,\qquad 1\le i\le d. $$
    Note that we used Lemma \ref{lem:bigspace} to define $X^{N,T_n}_{i,I}$.
\end{defi}

Note that for $I,T_n$ fixed, the family $(x_{i,I}^{T_{n}})_{1\leq i\leq d}$ is a free semicircular system.

\begin{defi}[Terms of the $m$-th kind]
    \label{def:terms}
    If $P\in\PP_d$, a ``term of the $m$-th kind associated to $P$''
    is a tuple $(P_1,\dots, P_{m+1})\in \mathcal P_d^{m+1}$ defined inductively as follows:
    \begin{itemize}
        \item For $m=0$, it is either $P$ or $P^*$.
        \item For $m>0$ and any term of the $m-1$-st kind $(P_1,\dots,P_m)$, for every $i,j$ we fix choices of $A_l^{i,j}, B_l^{i,j}\in \PP_d$ such that
        \begin{equation}
            \label{skjvnskvndigh}
            \partial_iP_j=\sum_lA_l^{i,j}\otimes B_l^{i,j},
        \end{equation}
        and declare $(P_1,\dots,P_{j-1}, A_l^{i,j},B_l^{i,j},P_{j+1},\dots, P_m)$ to be a term of the $m$-th kind for each $i,j,l$.
    \end{itemize}
    This allows us to define
    $$ c(P) \deq \sup_{\substack{m\geq 0, \\ i\in [1,d]}} \sup_{\substack{(P_1,\dots,P_{m+1}) \text{ is a} \\ \text{term of the $m$-th kind.}}} \left\{ \sum_{j=1}^{m+1}\kappa_j\ \middle|\ \forall j,\ \partial_i P_j = \sum_{l=1}^{\kappa_j} A_l^{i,j}\otimes B_l^{i,j} \right\}.$$
\end{defi}

Heuristically, $c(P)$ measures the growth in the number of terms between the successive differentials of a polynomial: $\partial_i P$ will be a sum of at most $c(P)$ simple tensors and $\left(\partial_j\otimes\id\circ \partial_i+\id\otimes\partial_j\circ \partial_i  \right)P$ will be a sum of at most $ c(P)^2$ simple tensors.
Besides, note that the definition of the terms of the $m$-th kind is not unique since there is no canonical way of writing $\partial_iP_j$ as a sum of simple tensors, which is why we \emph{fix} a way of doing so in the construction above. 

\begin{defi}[$P$-monomials]
\label{def:Pmonomial} 
    With the notations of Definition \ref{def:terms}, 
    given $P\in \mathcal P_d$, a \emph{$P$-monomial} is any non-commutative polynomial $M$ in variables $(X_{i,I})_{i\in[1,d], I\in \mathbb{Z}_{>0}^{2n}}$, for some fixed $n\ge0$, obtained in the following way:
    For some finite set $\mathcal{C}$ of terms $\underline P^j=(P^j_1,\dots, P^j_{m_j+1})$, $j=1,\dots, q$ associated to $P$, let 
            $$ M=\prod_{s=1}^p R_s( \bs X_{I^s})$$ 
    for some reordering $(R_1,\dots, R_p)\in \mathcal P_d^p$ of the concatenation $(\underline P^1,\dots, \underline P^q)\in \mathcal P_d^p$, and some $I^1,\dots, I^p\in \mathbb{Z}_{>0}^{2n}$ (possibly with repetition), where $\bs X_{I^s}= (X_{i,I^s})_{i\in [1,d]}$. 
    Letting $r_m$ be the number of terms of the $m$-th kind in $\mathcal C$, 
we set
    $$p(M) := p=\sum_{i\geq 0} (i+1)r_i,\quad q(M) := q=\sum_{i\geq 0} r_i. $$
\end{defi}

\begin{exe}
    With $P=(X_2+1)X_1X_2^2X_1+X_1\in \mathcal P_2$, we have
    $$ \partial_1 P = (X_2+1)\otimes X_2^2X_1 + (X_2+1)X_1X_2^2\otimes 1 + 1\otimes 1,  $$
    $$ \partial_2 P = 1\otimes X_1 X_2^2X_1 + (X_2+1)X_1\otimes X_2X_1 + (X_2+1)X_1X_2\otimes X_1,  $$
    so we may take the following six elements of $\mathcal P_2^2$:
    \begin{align}   
        &(X_2+1, X_2^2X_1), ((X_2+1)X_1X_2^2, 1), (1, 1), \notag\\
        &(1, X_1 X_2^2X_1), ((X_2+1)X_1, X_2X_1), ((X_2+1)X_1X_2, X_1),\label{ex:terms}
    \end{align}
    to be the terms of the first kind associated to $P$. 
    Since we may alternatively express
    \[
    \partial_1P = 
    X_2\otimes X_2^2X_1 +1\otimes X_2^2X_1+ (X_2+1)X_1X_2^2\otimes 1 + 1\otimes 1,
    \]
    we could have taken the two terms
    \[
    (X_2,X_2^2 X_1), \; (1, X_2^2 X_1)
    \]
    in place of the first term in \eqref{ex:terms} to get a list of seven terms of the first kind.
    
    Having fixed the list \eqref{ex:terms} of terms of the first kind, then by differentiating the second entry $X_2^2X_1$ of the first term in \eqref{ex:terms} with respect to $X_2$ as follows:
    $$\partial_2 X_2^2X_1 = 1\otimes X_2X_1 + X_2\otimes X_1, $$
    we obtain the following two terms of the second kind:
    \[
    (X_2+1,1,X_2X_1)\,,\qquad (X_2+1, X_2,X_1).
    \]
    
    From the collection $\mathcal C=\{ \underline P^1,\underline P^2,\underline P^3\}$ of terms
    \begin{equation}    \label{ex:terms2}
        \underline P^1=P^* \,, \quad  \underline P^2 =(1,X_1X_2^2X_1)\,,\quad \underline P^3=(X_2+1,X_2,X_1)
    \end{equation}
    and $I,I'\in (\mathbb N\setminus\{0\})^{2n}$ for some $n\ge0$, we can construct the following $P$-monomial:
    \begin{align*}
        M &= P^*(\bs X_I) P^3_1(\bs X_{I'})  P^3_3(\bs X_{I'}) P^2_1(\bs X_I) P^3_2(\bs X_I) P^2_2(\bs X_{I'})    \\
        &=P^*(\bs X_I) (X_{2,I'}+1) X_{1,I'} X_{2,I} X_{1,I'}X_{2,I'}^2X_{1,I'}  .
    \end{align*}
    For this $P$-monomial (or any $P$-monomial built from the selection \eqref{ex:terms2}) we have $q(M)= 3$ and $p(M) = 1+ 2+3=6$. 
\end{exe}

We can now state the lemma that we will use for estimating high order moments of the multiplicative Brownian motion.

\begin{lemma}
	\label{3apparition}
    There exist sequences $(A_i)_{i}, (J_i)_{i}, (\mathcal{L}_i)_{i}$ with the following properties:
    \begin{enumerate}[left=0pt,label=(\alph*)]
    \item\label{expand-Ai} $A_i\subset\R_+^{2i}$ is a Borel set (in fact a polyhedron) with
    \[
    \int_{A_i } e^{-t_{2i}-\dots-t_{1}} dt_1\dots dt_{2i} \leq \frac{1}{i!}\,.
    \]
    \item \label{expand-Ji} $J_i\subset \mathbb Z_+^{2i}$ is an index set.
    \item \label{expand-Li} $\mathcal{L}_i$ is a mapping
    \[
    \mathcal{L}_i: \mathcal P_d \times A_i \to \C\langle X_{j,I}:j\in[1,d], I\in J_i \rangle
    \]
    where 
        \begin{itemize}
            \item for each fixed $T\in A_i$, $\mathcal{L}_i(\cdot, T)$ is a linear map, 
            \item for each fixed $Q\in \mathcal P_d$, the coefficients of $\mathcal{L}_i(Q,\cdot)$ are measurable (in fact simple) functions on $A_i$,
            \item given $M$ a monomial with coefficient $1$, for each fixed $T\in A_i$, $\mathcal{L}_i(M, T)$ is a sum of monomials with  coefficient $1$.
        \end{itemize}
    \item\label{expand-N2} ($\frac1{N^2}$-Expansion).
    For any $Q\in \PP_d$ and $s\geq 1$,
	\begin{align} \label{skjvbsknv}
		&\E\left[ \ts_N\Big( Q(\bs{X}^N) \Big) \right] 
        = \tau(Q(\bs x^0)) + \sum_{i=1}^{s-1} \frac{\alpha_i(Q)}{N^{2i}}
        + \frac{\tilde \alpha^N_s(Q)}{N^{2s}}
	\end{align}
    with coefficients given by
    \begin{align*}
        \alpha_i(Q) &:= \frac{1}{2^i} \int_{A_i } e^{-t_{2i}-\dots-t_{1}} \tau\Big( \left[\mathcal{L}_i(Q,T_i)\right](\bs x^{T_i}) \Big)\ dt_1\dots dt_{2i} \\
        \tilde \alpha_s^N(Q) &:= \frac{1}{2^s} \int_{A_s} e^{-t_{2s}-\dots-t_{1}} \E\left[\tau_N\Big( \left[\mathcal{L}_s(Q,T_s)\right] (\bs X^{N,T_s}) \Big)\right]\ dt_1\dots dt_{2s}\,
    \end{align*}
    (using the notation $T_i=(t_1,\dots, t_{2i})$ of Definition \eqref{tech1}).
    \item\label{expand-complexity} (Complexity). 
        If $Q=(P^*P)^k$, then for any $T\in A_i$, $\mathcal{L}_i(Q,T)$ is the sum of at most 
           $ 
            d^{2i} \times \left( 2k\times  c(P) \right)^{4i} 
            $
        $P$-monomials $M$ (with $n=i$) such that $p(M) = 2k+4i$ and $q(M) = 2k$.
    \end{enumerate}
\end{lemma}

We refer the interested reader to \cite[Section 2.3]{Parraud2023} for the explicit definitions of the sets $J_i$ and mappings $\mathcal{L}_i$; their construction is quite lengthy and will not be used directly in this paper (the sets $A_i$ are recalled in the proof below). Lemma \ref{3apparition} distills all of the properties that we need for the proof of Theorem \ref{thm.2} in Section \ref{section.concentration}.

\begin{proof}

This lemma, and notably 
the construction of $A_i,J_i,\mathcal L_i$ and 
the $\frac1{N^2}$-expansion of part \ref{expand-N2},
is a direct corollary of Lemma 3.6 and Proposition 3.7 of \cite{Parraud2023}. To establish the esetimate in \ref{expand-Ai} we recall from \cite{Parraud2023} that
$$A_i = \{ t_{2i}\geq t_{2i-2}\geq \dots \geq t_2\geq 0 \}\cap\{\forall g\in [1,i], t_{2g} \geq t_{2g-1} \geq 0\} \subset \R^{2i}. $$ 
Since for every $g$, one has
$\int_0^{t_{2g}} e^{-t_{2g-1}} dt_{2g-1} \leq 1, $
we get that
$$\int_{A_i } e^{-t_{2i}-\dots-t_{1}} dt_1\dots dt_{2i} \leq \int_{0\leq t_1\leq\dots\leq t_i} e^{-t_{i}-\dots-t_{1}} dt_1\dots dt_{i} = \frac{1}{i!}.$$
This yields \ref{expand-Ai}.

For \ref{expand-Li}, one can set 
$$\mathcal{L}_1 : (Q,T)\in\PP_d\times A_1 \mapsto \sum_{1\leq i,j\leq d} \Theta\circ\ev_1\circ\left(\partial_{j}\otimes \partial_{j}\right)\circ \left(\partial_iD_i\right)(Q), $$
where $\Theta(A\otimes B\otimes C\otimes D) := BADC$ and $\ev_1$ is a map which evaluates each tensorand in specific subsets of $(X_{j,I})_{1\leq j\leq d, I\in J_1}$. Note that $\mathcal{L}_1(Q,T)$ does not depend on $T$. Besides, $\mathcal{L}_{s+1}(Q,T_{s+1})$ is defined by induction by applying the operator 
\[
\mathcal{K}_{l,J_s} : \C\langle (X_{i,I})_{1\leq i\leq d, I\in J_s}\rangle \ni R \mapsto \sum_{\substack{1\leq i,j\leq d \\ I,K\in J_s\\ \text{such that } I_l=K_l}} \ev_{s+1} \circ \Theta \circ \left(\partial_{j,I}\otimes \partial_{j,K}\right)\circ \left(\partial_iD_i\right) R
\]
to $\mathcal{L}_s(Q,T_s)$, for an integer $l=l(T_{s+1})$ determined by $T_{s+1}$ and $\ev_{s+1}$ a map which evaluates each tensorand in specific subsets of $(X_{j,I})_{1\leq j\leq d, I\in J_{s+1}}$. In this context, given $M$ a monomial,
$$ \partial_{i,I} M =  \sum_{M=AX_{i,I}B} A\otimes B, $$
$$ \partial_{i} M = \sum_{I\in J_s} \sum_{M=AX_{i,I}B} A\otimes B, $$
$$ D_i M = \sum_{I\in J_s} \sum_{M=AX_{i,I}B} BA. $$
In particular, when applied to a monomial with leading coefficient $1$, the operator $\mathcal{K}_{l,J_s}$ returns a sum of monomial also with leading coefficient $1$. Hence we get \ref{expand-Li}.

Now we prove assertion \ref{expand-complexity}.  Fix a list $\mathcal{T}_P$ of terms associated to $P$ as in Definition \ref{def:terms}. Thus, whenever we take the non-commutative differential of a $P$-monomial, we assume that the sum of simple tensors that appears is built with the help of the same equation \eqref{skjvnskvndigh} with which we built $\mathcal{T}_P$ by induction. Thus, if $M$ is a $P$-monomial with $q(M)=2k$, then by definition of $c(P)$ (computed with respect to $\mathcal{T}_P$), $\mathcal{K}_{l,J_s}(M)$ will be a sum of at most $d^2 (2k\times c(P))^4$ simple tensors. By induction we obtain the claim on the number of $P$-monomials. Besides, applying the operator $\mathcal{K}_{l,J_s}$ (which involves differentiating non-commutatively four times) to a $P$-monomial $M$ with $q(M)=2k$, then applying the operator $\Theta$, yields a sum of $P$-monomials, and given one such $P$-monomial $L$, the number of terms used to build it does not change, i.e. $q(L)=q(M)=2k$. However, by performing this operation, we proceed to replace a term by a new term one kind higher four times in a row. Hence $p(L)=p(M)+4$ and the last part of \ref{expand-complexity} follows by induction on $i$ and the fact that $p(Q) = 2k$. \end{proof}

Before concluding this subsection, let us prove the following proposition that allows us to control moments of semicircular variables and GUE matrices by the moments of certain jointly Gaussian random variables.

\begin{prop}
	\label{prop:monomial}
    In the setting of Definition \ref{tech1}, for any monomial $M$ (with coefficient 1) in non-commutative variables $(X_{i,I})_{i\in [1,d],I\in J_n}$ we have
	\begin{equation}
		\label{ovsfnsdkov}
		0 \leq \tau\left(M(\bs x^{T_n})\right) \leq \E[M(g)],
	\end{equation}
	\begin{equation}
		\label{lskdvnsl}
		0 \leq \E\left[\tau_N\left(M(\bs X^{N,T_n})\right)\right] \leq \E[M(g)],
	\end{equation}
	where $g=(g_{i,I})_{i\in[1,d],I\in J_n}$ are centered real Gaussian random variables of variance $1$ such that $g_{i,I}$ and $g_{j,J}$ are independent if $i\neq j$ and equal otherwise.
	
	In particular, given $c_M\in\C$ and $a_M\in\R^+$ such that $|c_M|\leq a_M$, then
	\begin{equation}
		\label{skjdvngfd}
		\left| \E\left[\tau_N\left(\sum_{M \text{ monomial}}c_M M(\bs X^{N,T_n})\right)\right] \right| \leq \E\left[\sum_{M \text{ monomial}}a_M M(g)\right],
	\end{equation}
	and similarly for $\bs x^{T_n}$ instead of $\bs X^{N,T_n}$.
\end{prop}

\begin{proof}
	Thanks to Proposition \ref{prop:SDeq}(b), assuming that $M = X_{i,I} R$, we have that
	$$ \tau\left(M(\bs x^{T_n})\right) = e^{-\widetilde{t}_{2n}} \tau\otimes\tau\left(\partial_{i,0}R(\bs x^{T_n})\right) + \sum_{l=1}^{2n} (e^{-\widetilde{t}_{l-1}} -e^{-\widetilde{t}_l}) \tau\otimes\tau\left(\partial_{i,I_l}R(\bs x^{T_n})\right), $$
	where $\partial_{i,I_l}$ is the non-commutative partial differential with respect to $x_{i,I_l}$. Thus, by a straightforward induction, $\tau\left(M(\bs x^{T_n})\right)\geq 0$. 
	
	Let us remark that the product of the $s$-th and $r$-th moment of a centered Gaussian random variable of variance $1$ is always bounded by its $s+r$-th moment. Thus for any monomials $A$ and $B$, $\E[A(g)]\E[B(g)] \leq \E[A(g)B(g)]$. Hence if we assume that \eqref{ovsfnsdkov} is satisfied for monomials of degree strictly smaller than $k$ and that $M$ has degree $k$, then
	$$ \tau\left(M(\bs x^{T_n})\right) \leq \left(e^{-\widetilde{t}_{2n}} + \sum_{l=1}^{2n} (e^{-\widetilde{t}_{l-1}} -e^{-\widetilde{t}_l}) \right) \E\left[\frac{\partial R(g)}{\partial g_i}\right] = \E\left[\frac{\partial R(g)}{\partial g_i}\right] = \E\left[g_i R(g)\right], $$
	where we used Gaussian integration by part in the last equality.  The bound \eqref{ovsfnsdkov} now follows by induction.
	
	Turning to \eqref{lskdvnsl}, thanks to Lemma \ref{3apparition} (or more precisely \cite[Lemma 3.6]{Parraud2023}), 
	$$ \E\left[\ts_N\left(M(\bs X^{N,T_n})\right)\right] = \sum_{i=0}^{k/4} \frac{\alpha_i}{N^{2i}}, $$
	where the coefficients $\alpha_i$ are linear combinations with non-negative coefficients of traces of monomials in $x^{T_l}$ for $l\geq n$. In particular, thanks to the first part of the proof, those coefficients are non-negative. Thus 
	$$ 0\leq \E\left[\ts_N\left(M(\bs X^{N,T_n})\right)\right] = \sum_{i=0}^{k/4} \frac{\alpha_i}{N^{2i}} \leq \sum_{i=0}^{k/4} \alpha_i = \E\left[\tau\left(M(\bs X^{1,T_n})\right)\right]. $$
	However, if we let $G=(G_{i,I}^{T_n})_{i\in[1,d],I\in J_n}$ be a family of Gaussian random variables, built exactly like $x_{i,I}^{T_n}$ in Definition \ref{tech1} but with independent Gaussian random variables of variance $1$ instead of free semicircular variables, then by using Equation \eqref{ovsfnsdkov},
	$$ 0\leq \E\left[\ts_N\left(M(\bs X^{N,T_n})\right)\right] \leq \E\left[M(g)\right]. $$
	Hence the conclusion by using the fact that if $g_1,\dots,g_l$ are centered Gaussian random variables of variance $1$, not necessarily independent but such that for all $i,j$, $\cov(g_i,g_j)\geq 0$, then
	$\E\left[g_1\dots g_l\right] \leq \E[g_1^l].$
\end{proof}

\section{
Small-time Affine Approximation \label{section.concentration}}

\subsection{Setup and overview}  For the duration of this section we let $W^N=W^N_{\rho,\zeta}$ be an elliptical Brownian motion as in Definition \ref{def:elliptical}, i.e. 
\[
W^N(t)=e^{\i \theta }(a X^N(t) +\i bY^N(t))\,,\quad t\ge0
\]
with $(X^N(t))_{t\geq 0},(Y^N(t))_{t\geq 0}$ independent Hermitian Brownian motions. Given $W^N$, we let $B^N$ be the multiplicative Brownian motion as in Definition \eqref{def.SDE}, driven by $W^N$, and for each $n\in \N$ we let $B^N_n$ be the discretized process 
\begin{equation}    \label{def:discrete}
    B_n^{N}(t) = \prod_{i=1}^{\lfloor t2^n \rfloor} \left(I_N + W^N\left(\frac{i}{2^n}\right) - W^N\left(\frac{i-1}{2^n}\right) +\frac{\zeta}{2^{n+1}} I_N\right). 
\end{equation}
Note that $B^N$ and $B^N_n$ are defined using the same driving process $W^N$. We also fix realizations of the corresponding free non-commutative processes $x,y,w,b,b_n$, defining the latter as in \eqref{def:discrete} with $W^N$ replaced by $w$. 
We set
\begin{equation}
    \label{def:tn}
    t_n:= \lfloor t2^n\rfloor /2^n\,.
\end{equation}
Note that $B_n^N(t) = B_n^N(t_n)$ for all $n,t$. 

We further let $(X_i^N,Y_i^N)_{i\ge1}$ be an i.i.d.\ sequence of independent pairs of GUE matrices,  independent of all other processes. 
For computations, and in particular for invoking Lemma \ref{3apparition}, it will be convenient to reexpress joint moments of increments of the processes $B^N,B_n^N,W^N$ in terms of the $X_i^N,Y_i^N$. 
Indeed, note that
\begin{equation}
    \label{def:delW}
    W^N\bigg(\frac i{2^n}\bigg) - W^N\bigg(\frac{i-1}{2^n}\bigg) 
    \stackrel{d}{=} 
    \frac{e^{\i \theta}}{2^{n/2}} ( a X_i^N + b Y_i^N)\,,\quad i\ge1.
\end{equation}

As in the previous section we use $X_i,Y_i$ (without the superscript $N$) to denote generic arguments of non-commutative polynomials, which may be evaluated at the GUE matrices $X_i^N, Y_i^N$ or semicircular variables.

In this section we use $L^p$ norms for the non-commutative probability space of random matrices equipped with $\E\ts_N$ as state, and also on a ``large-$N$ limit'' non-commutative probability space equipped with a tracial state $\tau$.  We use the same notation for both, which will always be clear from context: for an $N\times N$ random matrix $A$ and non-commutative random variable $a$, 
\[ \norm{A}_p := \left(\E\ts_N(|A|^p)\right)^{1/p} \quad \&  \quad \norm{a}_p := \left(\tau(|a|^p)\right)^{1/p} \]
recalling $|A|= (AA^\ast)^{1/2}$, $|a|= (a^\ast a)^{1/2}$.

We prove Theorem \ref{thm.2} using the moment method. Thus, we need to bound
\begin{equation}    \label{BIW-2k}
    \|B^N(t)-I_N-W_N(t)\|_{2k}^{2k}=\E \ts_N |B^N(t)-I_N-W_N(t)|^{2k}
\end{equation}
for large $k=k(N)$. This is accomplished with Proposition \ref{prop:2kmom}.

Towards the proof of Proposition \ref{prop:2kmom}, in Section \ref{sec:discrete} we show $B_n^N\to B^N$ in $L^{2k}$, uniformly on compact time intervals. 
This allows us to replace $B^N(t), W^N(t)$ in \eqref{BIW-2k} with $B_n^N(t_n), W^N(t_n)$. 
From \eqref{def:discrete}, \eqref{def:delW} we can then view $|B_n^N(t) - I_N - W_N(t)|^{2k}$ as a non-commutative polynomial in independent GUE matrices, allowing for the application of Lemma \ref{3apparition}.
The proof of Proposition \ref{prop:2kmom} is then concluded in Section \ref{sec:thm.2} by bounding each term in the resulting expansion. 

In Section \ref{sec:norm} we deduce tail bounds on the norm of $B^N(t)$ that will also be useful in subsequent sections.

\subsection{Random walk approximation}
\label{sec:discrete}

Recall the discretized processes $B_n^N,b_n$ from \eqref{def:discrete}.
In this subsection we show these processes converge to $B,b$ in $L^p$. 
The main results of this section are the following proposition and corollary.

\begin{prop} \label{prop:Lp.conv} For any $T\geq 0$ and $2\le p<\infty$, $B_n^N\to B^N$ 

in $L^p(\Omega\times [0,T],\ts_N)$. 
Moreover, there exists a constant $C_{\rho,T,p}$ such that
	\begin{align}  \label{bd:BBn-Lp}
		&\norm{B^N(t)-B^N_n(t)}_p \leq \frac{C_{\rho,T,p}}{2^{n/2p}} \qquad \forall\,t\in[0,T].
	\end{align}
	The same estimate holds in $L^p(\tau)$
    with $B^N_n,B^N$ replaced by their free counterparts.
\end{prop}

\begin{cor} \label{cor.B-I-W.moments}
    For $t\ge 0$ and $2\le p<\infty$,
    \[ B^N_n(t)-I_N-W^N\left(t_n\right)\text{ converges to }B^N(t)-I_N-W^N(t)\text{ in }L^p\text{ as }n\to\infty. \]
    The same holds with $B^N_n,B^N$ replaced by their free counterparts $b_n,b$.
\end{cor}

\begin{proof}[Proof of Corollary \ref{cor.B-I-W.moments}]
Note that the $L^p$-norm of the difference is equal to
\begin{align*} &\norm{B^N_n(t)-B^N(t)+W^N(t)-W^N\left(t_n\right)}_{p} \\
    \leq &\norm{B^N_n(t)-B^N(t)}_{p} +\norm{W^N(t)-W^N\left(t_n\right)}_{p}.
\end{align*}
The first term 
tends to zero by Proposition \ref{prop:Lp.conv}.  For the second term, note that $W^N(t)-W^N(s)$ has the same law as $\sqrt{t-s}W^N(1)$, and so
\[ \norm{W^N(t)-W^N\left(t_n\right)}_{p} = \sqrt{t-t_n} \norm{aX^N+\i bY^N}_p \]
where $X^N$ and $Y^N$ are GUE matrices.  This $L^p$-norm is bounded independently of $N$ by a $\rho$-dependent constant times a $p$th Gaussian moment.  Since $t_n=\lfloor 2^n t\rfloor/2^n\to t$, 
the result follows uniformly in $N$.  The argument for the free case is similar. \end{proof}

The remainder of this subsection is dedicated to the proof of Proposition \ref{prop:Lp.conv}. We need the following:

\begin{lemma}
	\label{sdovmkson}
	For any $k\in\N$ and 
    $0\le t\le T$
    there exists a constant 
    $C_{k,T}$ such that
	$$ \sup_{n,N\geq 1} 
    \|B^N_n(t)\|_{2k} 
    \leq C_{k,T}.$$
    The same holds with $B^N_n(t)$ replaced by $b_n(t)$.
\end{lemma}

\begin{proof}
    From \eqref{def:delW}, 
    \[
    \E\left[\ts_N\left(|B^N_n(t)|^{2k}\right)\right] 
        = \E \left[ \ts_N \Bigg| \prod_{i=1}^{\lfloor t2^n \rfloor} \left( \bigg( 1+ \frac{\zeta}{2^{n+1}} \bigg)I_N + \frac{e^{\i\theta}}{2^{n/2}}\left(aX_i^N+\i b Y_i^N\right) \right) \Bigg|^{2k} \right].
    \]
    Thanks to Proposition \ref{prop:monomial}, 
    with $g_i,g_i'$ i.i.d.\ standard Gaussians, the right hand side above is bounded by
	\begin{align*}
		&\E\left[\left| \prod_{i=1}^{\lfloor t2^n \rfloor} \left(1+  \frac{|a^2-b^2|}{2^{n+1}}+ \frac{1}{2^{n/2}}\left(|a|g_i+|b| g_i'\right) \right) \right|^{2k}\right] \\
		&= \E\left[\left| \prod_{i=1}^{\lfloor t2^n \rfloor} \left(1+ \frac{|a^2-b^2|}{2^{n+1}}+ \sqrt{\frac{a^2+b^2}{2^n}}g_i \right) \right|^{2k}\right]\\
         &= \E\left[\left(1+ \frac{|a^2-b^2|}{2^{n+1}} + \sqrt{\frac{a^2+b^2}{2^n}}g_1 \right)^{2k}\right]^{\lfloor t2^n \rfloor}.
    \end{align*}
	One can write 
	$$ \E\left[\left(1+ \frac{|a^2-b^2|}{2^{n+1}}+ \sqrt{\frac{a^2+b^2}{2^n}}g_1 \right)^{2k}\right] = 1 + \sum_{i=1}^{k}\frac{c_{i,k}}{2^{i n}},$$
	where the coefficients $c_{i,k}$ do not depend on $n$. Thus
	$$ \E\left[\ts_N\left(|B^N_n(t)|^{2k}\right)\right] \leq e^{t\sum_{i=1}^{k}\frac{c_{i,k}}{2^{(i-1)n}}}. $$
	Hence the conclusion. The claim for $b_n$ follows by similar lines.
\end{proof}

Next we can show $B_n(t)$ and $b_n(t)$ are Cauchy in $L^2$: 

\begin{lemma}
	\label{sljvndksn}
	There are constants $C_\rho,K_\rho<\infty$ so that, for any $n\geq 1$, $t\ge0$,
	$$ \norm{B^N_n(t)-B_{n-1}^N(t)}_2^2 \leq C_{\rho}\frac{t+1}{2^n} e^{K_{\rho}t}. $$
	The above inequality also stands if one replaces $B^N_n(t)$ by $b_n(t)$.
\end{lemma}

\begin{proof}
	Let $(X_i^N,Y_i^N)_{i\geq 1} $ be a family of independent GUE matrices, and define
	$$ \Upsilon_i^n \deq \left(1+\frac{\zeta}{2^{n+1}}\right)I_N + \frac{ae^{\i\theta}}{2^{n/2}}X^N_i +\i \frac{be^{\i\theta}}{2^{n/2}}Y^N_i. $$
	Then $B^N_n(t)$ has the same distribution as the product of the $\Upsilon_i^n$ for $1\le i\le \lfloor t2^n\rfloor$.  We proceed inductively, peeling away the $i=1$ terms in the following product: 
	\begin{align*}
		&\E\left[\ts_N\left(B^N_n(t)(B^N_n(t))^*\right)\right] \\
        &= \E\left[\ts_N\left(\prod_{i=2}^{\lfloor t2^n \rfloor} \Upsilon_i^n \left(\prod_{i=2}^{\lfloor t2^n \rfloor} \Upsilon_i^n\right)^*\right)\right] \left|1+\frac{\zeta}{2^{n+1}}\right|^2 \\
		&\quad + \E\left[\ts_N\left( \frac{ae^{\i\theta} X^N_1 +\i be^{\i\theta} Y^N_1}{2^{n/2}} \prod_{i=2}^{\lfloor t2^n \rfloor} \Upsilon_i^n \left(\prod_{i=2}^{\lfloor t2^n \rfloor} \Upsilon_i^n\right)^* \right)\right]\left(1+\frac{\zeta^*}{2^{n+1}}\right) \\
		&\quad + \left(1+\frac{\zeta}{2^{n+1}}\right)\E\left[\ts_N\left( \prod_{i=2}^{\lfloor t2^n \rfloor} \Upsilon_i^n \left(\prod_{i=2}^{\lfloor t2^n \rfloor} \Upsilon_i^n\right)^* \frac{ae^{-\i\theta} X^N_1 -\i be^{-\i\theta} Y^N_1}{2^{n/2}}\right)\right] \\
		&\quad + \E\left[\ts_N\left(\frac{ae^{\i\theta} X^N_1 +\i be^{\i\theta} Y^N_1}{2^{n/2}} \prod_{i=2}^{\lfloor t2^n \rfloor} \Upsilon_i^n \left(\prod_{i=2}^{\lfloor t2^n \rfloor} \Upsilon_i^n\right)^* \frac{ae^{-\i\theta} X^N_1 -\i be^{-\i\theta} Y^N_1}{2^{n/2}}\right)\right] \\
		&= \E\left[\ts_N\left(\prod_{i=2}^{\lfloor t2^n \rfloor} \Upsilon_i^n \left(\prod_{i=2}^{\lfloor t2^n \rfloor} \Upsilon_i^n \right)^* \right)\right] \left(\left|1+\frac{\zeta}{2^{n+1}}\right|^2+\frac{\rho}{2^n}\right) \\
		&= \E\left[\ts_N\left(\prod_{i=2}^{\lfloor t2^n \rfloor} \Upsilon_i^n \left(\prod_{i=2}^{\lfloor t2^n \rfloor} \Upsilon_i^n\right)^* \right)\right] \left(1+\frac{\rho+\Ree\zeta}{2^n}+\frac{|\zeta|^2}{4^{n+1}}\right),
	\end{align*}
	where we used Proposition \ref{prop:SDeq} to go from the second to third equality, wherein the cross terms vanish (alternatively, this step follows from the fact that  $(X_i^N,Y_i^N)_{i\ge1}$ are independent standardized Wigner matrices). By induction we get that
	\begin{equation}
		\E\left[\ts_N\left(B^N_n(t)(B^N_n(t))^*\right)\right] = \left(1+\frac{\rho+\Ree\zeta}{2^n}+\frac{|\zeta|^2}{4^{n+1}}\right)^{\lfloor t2^n \rfloor}.
	\end{equation}
	Now to compute $\E\left[\ts_N\left(B^N_n(t)(B_{n-1}^N(t))^*\right)\right]$, we set
	$$ \Gamma_i^n \deq \left(1+\frac{\zeta}{2^{n}}\right)I_N + \frac{ae^{\i\theta}}{2^{n/2}}(X^N_{2i-1}+X_{2i}^N) +\i \frac{be^{\i\theta}}{2^{n/2}}(Y^N_{2i-1}+Y_{2i}^N) $$	
	and further note that $\lfloor t2^n\rfloor = 2\lfloor t2^{n-1}\rfloor+\chi$ for $\chi\in\{0,1\}$. Then 
	\begin{align*}
		&\E\left[\ts_N\left(B^N_n(t)(B_{n-1}^N(t))^*\right)\right] \\
		&= \E\Bigg[\ts_N\Bigg(\prod_{i=1}^{\lfloor t2^{n-1} \rfloor} \Upsilon_{2i-1}^n\Upsilon_{2i}^n \prod^{\lfloor t2^{n} \rfloor}_{i= 2\lfloor t2^{n-1}\rfloor+1} \Upsilon_i^n \times \Bigg(\prod_{i=1}^{\lfloor t2^{n-1} \rfloor} \Gamma_i^{n}\Bigg)^* \Bigg)\Bigg] \\
		&= \E\Bigg[\ts_N\Bigg(\prod_{i=1}^{\lfloor t2^{n-1} \rfloor} \Upsilon_{2i-1}^n\Upsilon_{2i}^n \Bigg(\prod_{i=1}^{\lfloor t2^{n-1} \rfloor} \Gamma_i^{n}\Bigg)^* \Bigg)\Bigg] \left(1+\frac{\zeta}{2^{n+1}}\right)^{\chi}
	\end{align*}
	where we used the fact that 
	$\prod^{\lfloor t2^{n} \rfloor}_{i= 2\lfloor t2^{n-1}\rfloor+1} \Upsilon_i^n$ (which is either the identity or $\Upsilon_{\lfloor t2^n\rfloor}^n$, as $\chi\in\{0,1\}$)
	is independent of other random matrices. Now since
	\begin{align*}
		\Upsilon_{2i-1}^n\Upsilon_{2i}^n =&\ \left(1+\frac{\zeta}{2^{n+1}}\right) \left(\frac{ae^{\i\theta}}{2^{n/2}}(X^N_{2i-1}+X^N_{2i}) + \i \frac{be^{\i\theta}}{2^{n/2}}(Y^N_{2i-1}+Y^N_{2i})\right) \\
		&+ \left(\frac{ae^{\i\theta}}{2^{n/2}}X^N_{2i-1} +\i \frac{be^{\i\theta}}{2^{n/2}}Y^N_{2i-1}\right)\left(\frac{ae^{\i\theta}}{2^{n/2}}X^N_{2i} +\i \frac{be^{\i\theta}}{2^{n/2}}Y^N_{2i}\right) \\
        &+ \left(1+\frac{\zeta}{2^{n+1}}\right)^2I_N,
	\end{align*}
	we can apply Proposition \ref{prop:SDeq} once again (or alternatively that $(X_i^N,Y_i^N)_{i\ge1}$ are independent standardized Wigner matrices), to get that
	\begin{align*}
		&\E\left[\ts_N\left(B^N_n(t)(B_{n-1}^N(t))^*\right)\right] \left(1+\frac{\zeta}{2^{n+1}}\right)^{-\chi}\\
		&= \E\left[\ts_N\left(\prod_{i=2}^{\lfloor t2^{n-1} \rfloor} \Upsilon_{2i-1}^n\Upsilon_{2i}^n \left(\prod_{i=2}^{\lfloor t2^{n-1} \rfloor} \Gamma_i^n\right)^*\right)\right] \\
        &\quad\quad \times\left(\left(1 + \frac{\zeta}{2^{n+1}}\right)^2\left(1 + \frac{\overline{\zeta}}{2^{n}}\right) + \left(1 + \frac{\zeta}{2^{n+1}}\right)\frac{\rho}{2^{n-1}} \right) \\
		&= \E\left[\ts_N\left(\prod_{i=2}^{\lfloor t2^{n-1} \rfloor} \Upsilon_{2i-1}^n\Upsilon_{2i}^n \left(\prod_{i=2}^{\lfloor t2^{n-1} \rfloor} \Gamma_i^n\right)^*\right)\right] \left(1+\frac{\rho + \Ree\zeta}{2^{n-1}} +\frac{C^{n,\rho,\zeta}}{4^n}\right),
	\end{align*}
    where $C^{n,\rho,\zeta}$ is uniformly bounded by a constant that only depends on $\rho$ (since $|\zeta|\le\rho$). Hence, by induction, we get 
	\begin{align*}
		\E\left[\ts_N\left(B^N_n(t)(B_{n-1}^N(t))^*\right)\right] &= \left(1+\frac{\rho + \Ree\zeta}{2^{n-1}} +\frac{C^{n,\rho,\zeta}}{4^n}\right)^{\lfloor t2^{n-1} \rfloor}  
        \left(1+\frac{\zeta}{2^{n+1}}\right)^{\chi}.
	\end{align*}
	Thus we have
	\begin{align*}
		&\norm{B^N_n(t)-B_{n-1}^N(t)}_2^2 \\
        &= \E\left[\ts_N\left(|B^N_n(t)|^2\right)\right] + \E\left[\ts_N\left(|B_{n-1}^N(t)|^2\right)\right] - 2\Ree \E\left[\ts_N\left(B^N_n(t)(B_{n-1}^N(t))^*\right) \right] \\
		&= \left(1+\frac{\rho+\Ree\zeta}{2^n}+\frac{|\zeta|^2}{4^{n+1}}\right)^{\lfloor t2^n \rfloor} + \left(1+\frac{\rho+\Ree\zeta}{2^{n-1}}+\frac{|\zeta|^2}{4^{n}}\right)^{\lfloor t2^{n-1} \rfloor} \\
		&\quad - 2\Ree \Bigg( \left(1+\frac{\rho+\Ree\zeta}{2^{n-1}} +\frac{C^{n,\rho,\zeta}}{4^{n}}\right)^{\lfloor t2^{n-1} \rfloor} \left(1+\frac{\zeta}{2^{n+1}}\right)^{\chi} \Bigg).
	\end{align*}
    Since $|\zeta|\le\rho$, the claimed bound for $B_n$ follows. The claim for $b_n$ follows from the same argument, using Proposition \ref{prop:SDeq}(b) in place of Proposition \ref{prop:SDeq}(a). 
\end{proof}

\begin{proof}[Proof of Proposition \ref{prop:Lp.conv}] 

    We provide the proof for $B^N_n \to B^N$; the free version is essentially identical.  
    
    We first reduce to the case $p=2$: we claim that for any $t\in\R$, there exists a constant $C_{\rho,T}$ such that
	\begin{align}  \label{goal:Lp-L2}
		&\norm{B^N(t)-B^N_n(t)}_2 \leq \frac{C_{\rho,T}}{2^{n/2}} \qquad \forall\,t\in[0,T].
	\end{align}
    (Note the bound is stronger than the $p=2$ case of \eqref{bd:BBn-Lp}.)
    Before establishing \eqref{goal:Lp-L2} we use it to conclude the proof of the proposition. 
Note that, for $p\ge 1$, and any random matrix $A$,
\[ |A|^p = (A^\ast A)^{\frac{p}{2}} = (A^\ast A)^{\frac{p-1}{2}}(A^\ast A)^{\frac{1}{2}} = |A|^{p-1}\cdot|A|, \]
and hence taking expected traces $\|A\|_p^p = \E\ts_N[|A|^{p-1}|A|]$.  Thus, by the Cauchy-Schwarz inequality,
\begin{align} \nonumber
\norm{B^N_n(t)-B^N(t)}_{p}^p &= \E\ts_N[|B^N_n(t)-B^N(t)|^{p-1}\cdot|B^N_n(t)-B^N(t)|] \\ \nonumber
&\le \norm{|B^N_n(t)-B^N(t)|^{p-1}}_2\norm{B^N_n(t)-B^N(t)}_2 \\
&= \norm{B^N_n(t)-B^N(t)}_{2(p-1)}^{p-1}\cdot \norm{B^N_n(t)-B^N(t)}_2. \label{e.Holder.final}
\end{align}
Now, $\norm{B^N_n(t)-B^N(t)}_{2(p-1)} \le \norm{B^N_n(t)}_{2(p-1)} + \norm{B^N(t)}_{2(p-1)}$.  By Lemma \ref{sdovmkson} (and the fact that $\|A\|_p \le \|A\|_{\lceil p\rceil}$), there is a constant $C^1_{\rho,T,p}$ with $\norm{B^N_n(t)}_{2(p-1)} \le C^1_{\rho,T,p}$ for all $n,N$ and $t\in[0,T]$.  

Similarly, $\norm{B^N(t)}_{2(p-1)} \le \norm{B^N(t)}_{2k}$ for $k = \lceil p-1\rceil$.  By \cite[Theorem 5.1]{Kemp2016},
\[ \norm{B^N(t)}_{2k}^{2k} \le \norm{b(t)}_{2k}^{2k} + \frac{C(k,\rho,\zeta,t)}{N^2} \]
where the constant $C(k,\rho,\zeta,t)$ is continuous in all arguments; moreover, by \cite[Proposition 1.8]{Kemp2016}, $\|b(t)\|_{2k}^{2k}$ is continuous in $t$.  (These results are proved in \cite{Kemp2016} in the case $\zeta\in(-\rho,\rho)$ is real; tracking through the calculations shows they hold for general $|\zeta|\le\rho$ complex, with formulas replacing $\zeta$ by $\mathrm{Re}\,\zeta$.)

It follows by choosing $k$ with $2k>p$ that there is a constant $C^2_{\rho,p,T}$ with $\sup_N \norm{B^N(t)}_{2(p-1)} \le C^2_{\rho,T,p}<\infty$ for $t\in[0,T]$.
Hence, from \eqref{e.Holder.final}, we have
\[ \norm{B^N_n(t)-B^N(t)}_{p}^p \le (C^1_{\rho,T,p}+C^2_{\rho,t,p})^{p-1}\cdot \norm{B^N_n(t)-B^N(t)}_2. \]
Combining this with 
\eqref{goal:Lp-L2} and taking $p$th roots, the proposition is thus proved with $C_{\rho,T,p}:=(C_{\rho,T})^{1/p}(C^1_{\rho,T,p}+C^2_{\rho,t,p})^{1-1/p}$.

    It remains to establish \eqref{goal:Lp-L2}.
	Thanks to Lemma \ref{sljvndksn}, $(B_n^N)_{n\geq 0}$ is a Cauchy sequence in $L^2(\Omega\times [0,T];\M_N(\C))$, and hence converges towards a stochastic process $\wt B^N$. By construction, for any $t\in\R$
	\begin{equation}
		\label{skjdvnsknl}
		\norm{\wt B^N(t)-B^N_n(t)}_2 \leq \sum_{s\geq n+1} \sqrt{C_{\rho}\frac{t+1}{2^s} e^{K_{\rho}t}} \leq  \frac{\sqrt{C_{\rho}(t+1)}e^{K_{\rho}t/2}}{2^{n/2} (\sqrt{2}-1)}.
	\end{equation}
    It only remains to show that $\wt B^N=B^N$.
    By strong uniqueness for the SDE defining $B^N$ (see Definition \ref{def.SDE}) it suffices to show
    \begin{equation}    \label{Lp-goal}
        \wt B^N(t) = I_N + \int_{0}^{t} \wt B^N(s) dW^N(s) + \frac{\zeta}{2} \int_{0}^{t} \wt B^N(s) ds \qquad \forall t\ge0.
    \end{equation}
    We compute
	\begin{align*}
		&\int_0^t B_n^N(s) dW^N(s) \\
        &= \sum_{j=1}^{\lfloor t2^n \rfloor} \int_{\frac{j-1}{2^n}}^{\frac{j}{2^n}} B_n^N\left(\frac{j-1}{2^n}\right) dW^N(s) + \int_{t_n}^{t} B_n^N\left(t_n\right) dW^N(s) \\
		&= \sum_{j=1}^{\lfloor t2^n \rfloor} B_n^N\left(\frac{j-1}{2^n}\right) \left(W^N\left(\frac{j}{2^n}\right) - W^N\left(\frac{j-1}{2^n}\right)\right) 
        + B_n^N\left( t_n \right) \left(W^N(t)-W^N\left(t_n\right)\right) \\
		&= \sum_{j=1}^{\lfloor t2^n \rfloor} \left(B_n^N\left(\frac{j}{2^n}\right)  - B_n^N\left(\frac{j-1}{2^n}\right)\right) - \frac{\zeta}{2^{n+1}} \sum_{j=1}^{\lfloor t2^n \rfloor} B_n^N\left(\frac{j-1}{2^n}\right) \\
        &\quad\quad\quad\quad\quad\quad\quad\quad\quad\quad\quad\quad\quad\quad + B_n^N\left( t_n \right) \left(W^N(t)-W^N\left(t_n\right)\right) \\
		&= B_n^N\left( t_n \right)  - I_N - \frac{\zeta}{2} \int_0^{t_n} B_n^N(s)ds 
        + B_n^N\left( t_n \right) \left(W^N(t)-W^N\left(t_n\right)\right) \\
		&= B_n^N(t)  - I_N - \frac{\zeta}{2} \int_0^{t} B_n^N(s)ds 
        + B_n^N\left( t_n \right) \left(W^N(t)-W^N\left(t_n\right) + \frac{\zeta}{2}\left(t-t_n\right)I_N\right) .
 	\end{align*}
    Thanks to Lemma \ref{sdovmkson} and H\"older's inequality, we have
	\begin{align}
		&\norm{ B_n^N(t)  - I_N- \int_0^t B_n^N(s) dW^N(s) - \frac{\zeta}{2} \int_0^{t} B_n^N(s)ds }_2^2 \nonumber \\
		&\le
        \norm{B_n^N\left(t_n\right)}_4^2 \cdot 
        \norm{W^N(t)-W^N\left(t_n\right) + \frac{\zeta}{2} \left(t-t_n\right)I_N}_4^2 \nonumber \\
        &\le C_{2,t}\cdot \norm{W^N\left(t-t_n\right) + \frac{\zeta}{2} \left(t-t_n\right)I_N}_4^2    \label{ekvnsl}
	\end{align}
    wherein we used the fact that $C_{2,t}$ is nondecreasing in $t$ (and $t_n \le t$) and that $W^N$ has stationary centered increments.  Now, $e^{-\i\theta}W^N(s)$ has the same distribution as $\sqrt{s}(aX^N+\i bY^N)$ where $X^N,Y^N$ are independent GUEs.  The expected normalized trace moments of a GUE are well known from the genus expansion, see \cite[Section 13.4]{Kemp-RMT}; in particular, $\norm{X^N}_{4}^4 = 2+\frac{1}{N^2}\le 3$.   By the triangle inequality,
    \begin{align*} \eqref{ekvnsl} &\le C_{2,t}\cdot
    (t-t_n)\cdot \left( |a|\norm{X^N}_{4} + |b|\norm{Y^N}_{4}+|\zeta|\right)^2 =\mathcal{O}(2^{-n})
    \end{align*}
    so
    \begin{equation} \label{eq.L2.SDE.est}
    \norm{ B_n^N(t)  - I_N- \int_0^t B_n^N(s) dW^N(s) - \frac{\zeta}{2} \int_0^{t} B_n^N(s)ds }_{2} = \mathcal{O}(2^{-n/2}). \end{equation} 
	Now we also compare the above integrals of $B_n^N$ with the same integrals of $\wt B^N$:
	\begin{align*}
		&\norm{ \int_0^t B_n^N(s) dW^N(s) - \int_0^t \wt B^N(s) dW^N(s)}_{2}^2 \\
		&= \norm{ \int_0^t (B_n^N(s) - \wt B^N(s)) dW^N(s) }_{2}^2 \\
		&= \E\left[ \ts_N\left( \int_0^t (B_n^N(s) - \wt B^N(s)) dW^N(s) \left(\int_0^t (B_n^N(s) - \wt B^N(s)) dW^N(s) \right)^* \right) \right] \\
		&= \E\Bigg[ \sum_{1\leq i,j,l,a,b,c\leq N} \ts_N\Bigg( \int_0^t (B_n^N(s) - \wt B^N(s))_{i,j} d(W^N(s))_{j,l} E_{i,l} \\
        &\quad\quad\quad\quad\quad\quad\quad\quad\quad\quad\quad \times \left(\int_0^t (B_n^N(s) - \wt B^N(s))_{a,b} d(W^N(s))_{b,c} E_{a,c} \right)^* \Bigg) \Bigg] \\
		&= \frac{1}{N}\sum_{1\leq i,j,l,b \leq N} \E\Bigg[ \int_0^t (B_n^N(s) - \wt B^N(s))_{i,j} d(W^N(s))_{j,l} \\
        &\quad\quad\quad\quad\quad\quad\quad\quad \times\int_0^t \overline{(B_n^N(s) - \wt B^N(s))_{i,b}}\ d\overline{(W^N(s))_{b,l}} \Bigg] \\
		&= \frac{2\rho}{N^2}\sum_{1\leq i,j,l \leq N} \E\left[ \int_0^t \left|(B_n^N(s) - \wt B^N(s))_{i,j}\right|^2 ds \right] \\
		&= 2\rho\int_0^t \norm{B_n^N(s) - \wt B^N(s)}_{2}^2 ds \\
		&\leq \frac{\int_0^t C_{\rho}(s+1)e^{K_{\rho}s} ds}{2^n (\sqrt{2}-1)^2},
	\end{align*}
	where we used \eqref{skjdvnsknl} in the last line. In combination with \eqref{eq.L2.SDE.est}, and the fact that
	$$ \norm{\int_0^{t} (B_n^N(s)-\wt B^N(s))ds}_{2} \leq \int_0^{t} \norm{B_n^N(s)-\wt B^N(s)}_{2}ds \,,$$
    upon taking $n\to\infty$ we obtain \eqref{Lp-goal} to complete the proof.
\end{proof}

\subsection{Proof of Theorem \ref{thm.2}}
\label{sec:thm.2}

We now come to the main technical result of this section, 
from which
Theorem \ref{thm.2} quickly follows.

\begin{prop}
	\label{prop:2kmom}
	Given a $(\rho,\zeta)$-Brownian motion $B^N$, driven by $W^N$, there exists a constant $C_{\rho, T}$ such that for all $k\geq 0$ and $t\leq T$,
	\begin{align*}
		\E\left[\ts_N\left(\left|B^N(t) - I_N - W^N(t)\right|^{2k}\right)\right] \leq (t C_{\rho,T})^{2k} e^{C_{\rho,T} \frac{k^4}{N^2}}.
	\end{align*}
\end{prop}

    \begin{proof}[Proof of Theorem \ref{thm.2}]
    Since the probability is trivially bounded by 1, by modifying the constant $C$ in the theorem statement we may assume $t\le c\delta$ for any fixed $c=c(\rho,T)>0$.
    From Markov's inequality and Proposition \ref{prop:2kmom}, 
	\begin{align}
		\P\left( \norm{B^N(t) - I_N - W^N(t)} \geq \delta \right) &\le \P\left( \norm{\left|B^N(t) - I_N - W^N(t)\right|^{2k}} \geq \delta^{2k} \right) \notag\\
		&\leq \P\left( \tr_N\left(\left|B^N(t) - I_N - W^N(t)\right|^{2k}\right) \geq \delta^{2k} \right) \notag\\
		&\leq N \frac{\E\left[ \ts_N\left(\left|B^N(t) - I_N - W^N(t)\right|^{2k}\right) \right]}{\delta^{2k}} \notag\\
		&\leq N\left(\frac{t C_{\rho,T}}{\delta}\right)^{2k}e^{C_{\rho,T} \frac{k^4}{N^2}}\notag\\
        &= ( e^{f(N,k)} t/\delta)^{2k}  \label{efNk}
	\end{align}
    for $f(N,x):= \log C_{\rho,T} + \frac1{2x}\log N + \frac12C_{\rho,T} x^3/N^2$. Now for any $x\in [N^{2/3},2N^{2/3}]$, 
    \[
    f(N,x) \le C_0:= \log C_{\rho,T} + 4C_{\rho,T} + \sup_N N^{-2/3}\log N<\infty
    \]    
    so \eqref{efNk} is 
    bounded by
    $
    (e^{C_0} t/\delta)^{2\lceil N^{2/3}\rceil},
    $
    which in turn is at most $(e^{C_0}t/\delta)^{N^{2/3}}$ for $t\le e^{-C_0}\delta $. The claim follows.
    \end{proof}

It remains to prove Proposition \ref{prop:2kmom}. We need the following:

\begin{lemma}
	\label{lem:bnorm}
	Given $(b(t))_{t\geq 0}$ a $(\rho,\zeta)$-free Brownian motion driven by $(w(t))_{t\geq 0}$, there exists constants $C_{\rho},K_{\rho}<\infty$, such that
	$$  \norm{b(t)} \leq C_{\rho} e^{K_{\rho} t},\quad \norm{b(t) - 1} \leq \sqrt{t}\ C_{\rho} e^{K_{\rho} t},\quad \norm{b(t) - 1 - w(t)} \leq t\ C_{\rho} e^{K_{\rho}t}. $$
\end{lemma}

\begin{proof}
	Note that
	$$ b(t) = 1 + \int_0^t b(s)\ dw(s) + \frac{\zeta}{2} \int_{0}^{t} b(s) ds, $$
	$$ b(t) - 1 = \int_0^t b(s)\ dw(s) + \frac{\zeta}{2} \int_{0}^{t} b(s) ds, $$
	$$ b(t) - 1 - w(t) = \int_0^t (b(s)-1)\ dw(s) + \frac{\zeta}{2} \int_{0}^{t} b(s) ds. $$
	Besides $w = e^{\i\theta}(a x + \i b y)$ where $x,y$ are usual self-adjoint free Brownian motions. Thus by using the free Burkholder--Gundy inequality (see \cite[Theorem 3.2.1]{BianeSpeicher1998}) coupled with the fact that
	$$ \norm{\int_{0}^{t} b(s) ds} \leq \left(\int_{0}^{t} \norm{b(s)}^2 ds\right)^{1/2}, $$
	and $|\zeta|\leq \rho$, we get that 
	$$ \norm{b(t)} \leq 1 + \left(4\sqrt{2\rho}+\frac{\rho}{2}\right) \left(\int_0^t \norm{b(s)}^2 ds\right)^{1/2}, $$
	$$ \norm{b(t) - 1} \leq \left(4\sqrt{2\rho}+\frac{\rho}{2}\right) \left(\int_0^t \norm{b(s)}^2 ds\right)^{1/2}, $$
	$$ \norm{b(t) - 1 - w(t)} \leq \sqrt{32\rho}\left(\sqrt{32\rho}+\frac{\rho}{2}\right) \left(\int_0^t \int_0^s \norm{b(u)}^2 du ds\right)^{1/2} + \frac{\rho}{2} \int_0^t \norm{b(s)} ds. $$
	By squaring the first equation we get that
	$$ \norm{b(t)}^2 \leq 2 + 2\left(4\sqrt{2\rho}+\frac{\rho}{2}\right)^2 \int_0^t \norm{b(s)}^2 ds, $$
	and by Gronwall's inequality,
	$$ \norm{b(t)}^2 \leq 2 e^{2\left(4\sqrt{2\rho}+\frac{\rho}{2}\right)^2t}. $$
	Consequently, we have that
    \begin{align*}
        \norm{b(t) - 1} &\leq \left(4\sqrt{2\rho}+\frac{\rho}{2}\right) \left(\int_0^t 2 e^{2\left(4\sqrt{2\rho}+\frac{\rho}{2}\right)^2s} ds\right)^{1/2} \\
        &\leq \left(4\sqrt{2\rho}+\frac{\rho}{2}\right) \sqrt{2t}\ e^{\left(4\sqrt{2\rho}+\frac{\rho}{2}\right)^2t},
    \end{align*}
	\begin{align*}
		\norm{b(t) - 1 - w(t)} &\leq 4\sqrt{2\rho}\left(4\sqrt{2\rho}+\frac{\rho}{2}\right) \left(\int_0^t \int_0^s 2 e^{2\left(4\sqrt{2\rho}+\frac{\rho}{2}\right)^2u} du ds\right)^{1/2} \\
        &\quad + \frac{\rho}{2} \int_0^t \sqrt{2} e^{\left(4\sqrt{2\rho}+\frac{\rho}{2}\right)^2s} ds \\
		&\leq \left(8\sqrt{\rho}\left(4\sqrt{2\rho}+\frac{\rho}{2}\right) e^{\left(4\sqrt{2\rho}+\frac{\rho}{2}\right)^2t} + \frac{\rho}{\sqrt{2}} e^{\left(4\sqrt{2\rho}+\frac{\rho}{2}\right)^2t}\right) t.
	\end{align*}
	Hence the conclusion.
	\end{proof}

\begin{proof}[Proof of Proposition \ref{prop:2kmom}]
	To begin with, we define
    \begin{align*}
        &\bs{X}^N=\left(2^{n/2} \left(X^N\left(\textstyle{\frac{j}{2^n}}\right) - X^N\left(\textstyle{\frac{j-1}{2^n}} \right)\right)\right)_{1\leq j\leq \lfloor 2^n t \rfloor}, \\
        &\bs{Y}^N = \left(2^{n/2}\left(Y^N\left(\textstyle{\frac{j}{2^n}}\right) - Y^N\left(\textstyle{\frac{j-1}{2^n}}\right)\right)\right)_{1\leq j\leq \lfloor 2^n t \rfloor},
    \end{align*}
	where $(X^N(t))_{t\geq 0}$ and $(Y^N(t))_{t\geq 0}$ are the Hermitian Brownian motions used to define $W^N$ in Definition \ref{def:elliptical}. In this proof, we work with $\PP_d$ where $d=2\lfloor 2^n t \rfloor$. Besides, those $2\lfloor 2^n t \rfloor$ indeterminates will be sorted in the following order: $X_1,\dots,X_{\lfloor 2^n t \rfloor},$ $ Y_1,\dots,Y_{\lfloor 2^n t \rfloor}$. In particular, $\partial_p$ will be the non-commutative differential with respect to $X_p$ if $p\in [1,\lfloor 2^n t \rfloor]$, and with respect to $Y_{p-\lfloor 2^n t \rfloor}$ if $p\in [\lfloor 2^n t \rfloor+1,2\lfloor 2^n t \rfloor]$. We now define polynomials
    \begin{align*}
        S_n &:= \prod_{1\leq j\leq \lfloor 2^n t \rfloor}\left(1+\frac{\zeta}{2^{n+1}}+e^{\i\theta}\frac{a X_j+ b Y_j}{2^{n/2}}\right),\\
        Q_n &\deq \bigg|S_n - \sum_{1\leq j\leq \lfloor 2^n t \rfloor} e^{\i\theta} \frac{aX_j+\i bY_j}{2^{n/2}} - 1 \bigg|^{2k}
    \end{align*}
    in non-commuting indeterminates $X_j,Y_j$. Note that $Q_n$ is such that $Q_n(\bs{X}^N,\bs{Y}^N) = \left|B^N_n(t) - I_N - W^N\left(t_n\right)\right|^{2k}$. Besides, since $(\bs{X}^N,\bs{Y}^N)$ is a family of independent GUE random matrices, Lemma \ref{3apparition} yields an expansion
	\begin{align}
		\label{cite1}
		&\E\left[ \ts_N\left( \left|B^N_n(t) - I_N - W^N\left(t_n\right)\right|^{2k} \right) \right]
        = \sum_{i=0}^{s-1} \frac{\alpha_i(Q_n)}{N^{2i}} 
        + \frac{\tilde \alpha_s^N(Q_n)}{N^{2s}}
	\end{align}
    for any integer $s\ge1$,
    with the convention
    $\alpha_0(Q_n)=\tau(Q_n(\bs x^0,\bs y^0))$.
	From there on, we will divide the proof in two steps where we upper bound each of the terms in the equation above, before eventually letting $s$ go to infinity.
	
	\textbf{Step 1:}
    We first bound the coefficient for the remainder term in \eqref{cite1}:
    \begin{equation*}
        \tilde \alpha_s^N(Q_n) = \frac{1}{2^s} \int_{A_s} e^{-t_{2s}-\dots-t_{1}} \E\left[\tau_N\Big( \left[\mathcal{L}_s(Q_n,T_s) \right] (\bs X^{N,T_s},\bs Y^{N,T_s}) \Big)\right]\ dt_1\dots dt_{2s}\,.
    \end{equation*}
    First, let us remark that if we set
	$$ U_n \deq \prod_{1\leq j\leq \lfloor 2^n t \rfloor}\left(1+\frac{|\zeta|}{2^{n+1}}+\frac{|a|X_j+ |b| Y_j}{2^{n/2}}\right),\quad P_n \deq \left|U_n\right|^{2k},$$
	then one can write $Q_n$ and $P_n$ as linear combinations of monomials $M$, with coefficients $c_M$ and $a_M$ respectively, such that 
    $|c_M|\leq a_M$. This property of the polynomials $Q_n$ and $P_n$ is preserved under application of the operator $\mathcal{L}_s$. Indeed, thanks to Lemma \ref{3apparition}\ref{expand-Li}, we know that applying the operator $\mathcal{L}_s$ to a monomial $M$ yields a sum of monomials (with coefficients 1). Therefore, let $g = \left(g_{i,I}\right)_{i\in [1,\lfloor 2^n t\rfloor], I\in J_s}$ 
    be a family of centered Gaussian random variables of variance $1$ where $g_{i,I}=g_{j,J}$ if $i=j$ and otherwise $g_{i,I}$ and $g_{j,J}$ are independent. If we also set $\widetilde{g}$ an independent copy of $g$, thanks to Proposition \ref{prop:monomial},
		\begin{align*}
		&\int_{A_s} e^{-t_{2s}-\dots-t_{1}} \E\left[\tau_N\Big( \left[\mathcal{L}_s(Q_n,T_s) \right] (\bs X^{N,T_s},\bs Y^{N,T_s}) \Big)\right]\ dt_1\dots dt_{2s} \\
		&\leq\int_{A_s} e^{-t_{2s}-\dots-t_{1}} dt_1\dots dt_{2s} \E\left[\left[\mathcal{L}_s(P_n,T_s) \right] (g,\widetilde{g}) \right].
	\end{align*}
	Besides, if $q_1,q_2\in [1,\lfloor 2^n t \rfloor]$ and $p\in [l,m]$,
    \begin{align*}
        & \partial_p \left(  \prod_{q_1\leq j\leq q_2}\left(1+\frac{|\zeta|}{2^{n+1}}+\frac{|a|X_j+ |b| Y_j}{2^{n/2}}\right) \right) \\
        &= \frac{|a|}{2^{n/2}}\prod_{q_1\leq j< p}\left(1+\frac{|\zeta|}{2^{n+1}}+\frac{|a|X_j+ |b| Y_j}{2^{n/2}}\right) \otimes \prod_{p< j\leq q_2}\left(1+\frac{|\zeta|}{2^{n+1}}+\frac{|a|X_j+ |b| Y_j}{2^{n/2}}\right).
    \end{align*}
    And we have the exact same formula for $ \partial_{p+\lfloor 2^n t \rfloor} $ but with $|a|$ replaced by $|b|$ in the first instance on the right hand side. Thus, a term of the $m$-th kind associated to $U_n$ will be of the form 
    $$ \left( \frac{P_1}{2^{n/2}},\dots,\frac{P_m}{2^{n/2}}, P_{m+1} \right), $$
    where
    $$ P_l = |c_l| \prod_{p_l\leq j< p_{l+1}}\left(1+\frac{|\zeta|}{2^{n+1}}+\frac{|a|X_j+ |b| Y_j}{2^{n/2}}\right), $$
    for $c_{m+1}=1$, $c_l\in\{a,b\}$ else and $p_l\in [1,\lfloor 2^nt\rfloor]$. Therefore, with the notations of Definition \ref{def:terms},  we have $c(U_n)=1$. We also set $g_j:=g_{j,J}$ since this does not depend on $J$. Then, thanks to Lemma \ref{3apparition} used with $d=2\lfloor 2^n t \rfloor$, for any $T_s\in A_s$, we get that $ \left[\mathcal{L}_s(P_n,T_s)\right] (g,\widetilde{g})$ is a sum of at most $(2k)^{4s} (2\lfloor 2^n t \rfloor)^{2s}$ $U_n$-monomials evaluated in $(g,\widetilde{g})$. Each of those quantities can then be upper bounded (since both $|a|$ and $|b|$ are smaller than $\sqrt{\rho}$) by
	$$ \left(\frac{\rho}{2^n}\right)^{\frac{1}{2}\sum_{m\geq 0} mr_m} \prod_{j=1}^{\lfloor 2^n t \rfloor} \left|1+\frac{|\zeta|}{2^{n+1}}+\frac{|a|g_j+|b|\widetilde{g}_j}{2^{n/2}}\right|^{2k-z_j}, $$
	where $z_j\in [0,2k]$ is such that $z_1+\dots+z_{\lfloor 2^n t \rfloor} = 4s$, and $r_m$ is the number of terms of the $m$-th kind used to build the $U_n$-monomial. But thanks to Lemma \ref{3apparition}\ref{expand-complexity},
    \begin{equation}
    \label{sdkvjnsldnvs}
        \sum_{m\geq 0} mr_m = \sum_{m\geq 0} (m+1)r_m - \sum_{m\geq 0} r_m = 2k+4s - 2k = 4s.
    \end{equation}
    Consequently, for some constant $C_{k,t,\rho}$,
	\begin{align*}
		\E\left[ \left[\mathcal{L}_s(P_n,T_s)\right] (g,\widetilde{g}) \right] 
        &\leq (2k)^{4s} (2\lfloor 2^n t \rfloor)^{2s} \left(\frac{\rho}{2^n}\right)^{2s} \E\left[\left|1+\frac{|\zeta|}{2^{n+1}}+\frac{|a|g_j+|b|\widetilde{g}_j}{2^{n/2}}\right|^{2k}\right]^{\lfloor 2^n t \rfloor} \\
		&\leq (2k)^{4s} (2\rho t)^{2s}\ C_{k,t,\rho},
	\end{align*}
	where for the last line we argue as in the proof of Lemma \ref{sdovmkson}.
    Combining with Lemma \ref{3apparition}\ref{expand-Ai}, we conclude
	\begin{align}
		\label{cite2}
        \sup_n
        \tilde \alpha_s^N(Q_n) 
        &\leq \frac{(2k)^{4s} (2\rho t)^{2s}}{s!}\ C_{k,t,\rho}. 
	\end{align}
	
	\textbf{Step 2:}
    We turn to bound the coefficients $\alpha_i(Q_n)$ in \eqref{cite1}, where
    \begin{equation*}
        \alpha_i(Q_n) = 
        \frac{1}{2^i }\int_{A_i } e^{-t_{2i}-\dots-t_{1}} \tau\Big( \left[\mathcal{L}_i(Q_n,T_i) \right] (\bs x^{T_i},\bs y^{T_i}) \Big)\ dt_1\dots dt_{2i}\,.
    \end{equation*}
	Note that if we set
	$$ R_n \deq \prod_{1\leq j\leq \lfloor 2^n t \rfloor} \left(1+\frac{\zeta}{2^{n+1}}+e^{\i\theta} \frac{aX_j+\i bY_j}{2^{n/2}}\right) - \sum_{1\leq j\leq \lfloor 2^n t \rfloor} e^{\i\theta} \frac{aX_j+\i bY_j}{2^{n/2}} - 1, $$
    so that $Q_n=|R_n|^{2k}$, as well as
    $$ S_n^{q_1,q_2} = \prod_{q_1\leq j\leq q_2} \left(1+\frac{\zeta}{2^{n+1}}+e^{\i\theta} \frac{aX_j+\i bY_j}{2^{n/2}}\right), $$
    then for $p\in [1,\lfloor 2^n t \rfloor]$,
	\begin{align*}
		\partial_p R_n &= \frac{ae^{\i\theta}}{2^{n/2}} \left( S_n^{1,p-1} \otimes S_n^{p+1, \lfloor 2^n t \rfloor} - 1\otimes 1\right) \\
		&= \frac{ae^{\i\theta}}{2^{n/2}} \left( \left(S_n^{1,p-1} - 1\right) \otimes S_n^{p+1, \lfloor 2^n t \rfloor} + 1\otimes \left(S_n^{p+1, \lfloor 2^n t \rfloor}-1\right)\right).
	\end{align*}
	And similarly for $p\in [q_1,q_2]$, we get that
	\begin{align*}
		\partial_p S_n^{q_1,q_2} = \frac{e^{\i\theta}a}{2^{n/2}} S_n^{q_1,p-1}\otimes S_n^{p+1,q_2}.
	\end{align*}
	
	\noindent We also have similar formulas for $p\in [\lfloor 2^n t \rfloor+1,2\lfloor 2^n t \rfloor]$. Consequently, we have $c(R_n)=2$ and terms of the first kind associated to $R_n$ are of the following form:
    $$ \left( \frac{c e^{\i\theta}}{2^{n/2}} \left(S_n^{1,p-1} - 1\right), S_n^{p+1, \lfloor 2^n t \rfloor}\right) \quad \text{or}\quad \left( \frac{ce^{\i\theta}}{2^{n/2}} , S_n^{p+1, \lfloor 2^n t \rfloor}- 1\right),$$
    where $c\in\{a,b\}$. Besides, terms of the $m$-th kind for $m\geq 2$ are of the form
    $$ \left( \frac{c_1 e^{\i\theta}}{2^{n/2}} S_n^{1,p_1-1},\frac{c_2 e^{\i\theta}}{2^{n/2}} S_n^{p_1+1,p_2-1},\dots,\frac{c_m e^{\i\theta}}{2^{n/2}} S_n^{p_{m-1}+1,p_m-1}, S_n^{p_m+1,\lfloor 2^nt\rfloor} \right), $$
    $$ \text{or}\quad \left( \frac{c_1 e^{\i\theta}}{2^{n/2}},\frac{c_2 e^{\i\theta}}{2^{n/2}} S_n^{p_1+1,p_2-1},\dots,\frac{c_m e^{\i\theta}}{2^{n/2}} S_n^{p_{m-1}+1,p_m-1}, S_n^{p_m+1,\lfloor 2^nt\rfloor} \right), $$
    
    \noindent with $c_l\in \{a,b\}$. Thus, thanks to Lemma \ref{3apparition}, $\left[\mathcal{L}_i(Q_n,T_i)\right](\bs x^{T_i},y^{T_i})$ is a sum of at most $ (4k)^{4i} (2\lfloor 2^n t \rfloor)^{2i}$ $R_n$-monomials which are themselves products of $2k+4i$ polynomials of the following form (or of its adjoints):
    \begin{align*}
        R_n,\quad \frac{c e^{\i\theta}}{2^{n/2}}\left(S_n^{q_1,q_2} -1\right),\quad \frac{c e^{\i\theta}}{2^{n/2}} S_n^{q_1,q_2},\quad \frac{c e^{\i\theta}}{2^{n/2}},
    \end{align*}
    each evaluated in a given free semicircular system and with $c\in\{a,b\}$. Besides, since both $|a|$ and $|b|$ are smaller than $\sqrt{\rho}$, given $(\bs x, \bs y) =(x_i,y_i)_{1\leq i\leq \lfloor 2^n t \rfloor}$ a free semicircular system, thanks to H\"older inequality, one can upper bound the trace of such an expression by 
    \begin{align}
    \label{sdjvnslvmls}
        \left( \frac{\rho}{2^{n}}\right)^{\frac{1}{2}\sum_{m\geq 0}mr_m}& \norm{R_n(\bs x,\bs y)}_{2k+4i}^{r_0} \times \left( \sup_{q_1,q_2} \norm{S_n^{q_1,q_2}(\bs x,\bs y) - 1}_{2k+4i} \right)^{r_1} \\
        \times& \left( \sup_{q_1,q_2} \norm{S_n^{q_1,q_2}(\bs x,\bs y)}_{2k+4i} \right)^{2k+4i-r_1-r_0}. \nonumber
    \end{align}
    Let us now remark that thanks to Proposition \ref{prop:Lp.conv} and Lemma \ref{lem:bnorm}, there exist constants $C_{\rho,t,k,i}$ such that for $1\leq q_1\leq q_2\leq \lfloor 2^n t \rfloor$,
    $$\norm{R_n(\bs x,\bs y)}_{2k+4i} \leq tC_{\rho} e^{K_{\rho}t} + C_{\rho,t,k,i}\ 2^{-\frac{n}{4k+8i}},$$
    \begin{align*}
        \norm{S_n^{q_1,q_2}(\bs x,\bs y) - 1}_{2k+4i}  &= \norm{b_n\left( \frac{q_2-q_1+1}{2^n} \right) -1}_{2k+4i} 
        \leq \sqrt{t} C_{\rho} e^{K_{\rho}t} + \frac{C_{\rho,t,k,i}} {2^{n/(4k+8i)}},
    \end{align*}
    $$\norm{S_n^{q_1,q_2}(\bs x,\bs y)}_{2k+4i} = \norm{b_n\left( \frac{l-m+1}{2^n} \right) }_{2k+4i} \leq C_{\rho} e^{K_{\rho}t} + \frac{C_{\rho,t,k,i}} {2^{n/(4k+8i)}}.$$
    
    \noindent Thus, by using again Lemma \ref{3apparition}\ref{expand-complexity} as in Equation \eqref{sdkvjnsldnvs}, we get that \eqref{sdjvnslvmls} is upper bounded by 
	$$ \frac{\rho^{2i}}{2^{2ni}}t^{r_0+\frac{r_1}{2}} \left(C_{\rho}e^{K_{\rho }t}\right)^{2k+4i} + K_{\rho,t,k,i} 2^{-\left(2i+c_{k,i}\right)n}, $$
    for some constant $K_{\rho,t,k,i}$. Besides, thanks to Lemma \ref{lem:simple} below coupled with Lemma \ref{3apparition}\ref{expand-complexity}, one can always lower bound $r_1+\frac{r_2}{2}$ by $2k-2i$. Thus, one can upper bound $t^{r_1+\frac{r_2}{2}}$ by $t^{2k-2i}$ (up to a constant $T^{2i}$ if $t\geq 1$). Finally, we get that for some constant $C_{\rho,T}$,
	\begin{align*}
		&\limsup_{n\to\infty} \left|\int_{A_i } e^{-t_{2i}-\dots-t_{1}} \tau\Big( \left[\mathcal{L}_i(Q_n,T_i)\right] (\bs x^{T_i},\bs y^{T_i}) \Big)\ dt_1\dots dt_{2i}\right| \\
		&\leq \limsup_{n\to\infty} \int_{A_i } e^{-t_{2i}-\dots-t_{1}} dt_1\dots dt_{2i} \times k^{4i} \lfloor 2^n t \rfloor^{2i} \times \frac{1}{2^{2ni}} t^{2k-2i} (C_{\rho,T})^{2k+i}\\
		&\leq \frac{k^{4i}}{i!} t^{2k} (C_{\rho,T})^{2k+i},
	\end{align*}
    where we used \ref{3apparition}\ref{expand-Ai} in the last line.
	
    From \eqref{cite1}, \eqref{cite2} and Corollary \ref{cor.B-I-W.moments}, we have shown that for all integers $k,s\ge 1$,
	\begin{align*}
		&\E\left[\ts_N\left(\left|B^N(t) - I_N - W^N(t)\right|^{2k}\right)\right] 
        =\lim_{n\to\infty}\E\left[\ts_N\left(\left|B^N_n(t) - I_N - W^N\left(t_n\right)\right|^{2k}\right)\right] \\        
        &\qquad\qquad\qquad\qquad\qquad\leq \sum_{i=0}^{s-1} \frac{k^{4i}(C_{\rho,T})^{i}}{N^{2i}i!} (t C_{\rho,T})^{2k} + \frac{(2k)^{4s} (2\rho t)^{2s}}{s!}\ C_{k,t,\rho} \\
		&\qquad\qquad\qquad\qquad\qquad\leq  (t C_{\rho,T})^{2k} e^{C_{\rho,T} \frac{k^4}{N^2}} + \frac{(2k)^{4s} (2\rho t)^{2s}}{s!}\ C_{k,t,\rho}
	\end{align*}
	The result now follows by letting $s\to\infty$.  \end{proof}

\begin{lemma} \label{lem:simple}
    If a sequence of non-negative integers $(r_m)_{m\geq 0}$ satisfies
    $$ \sum_{m\geq 0} (m+1)r_m = 2k+4i,\quad \text{and}\quad \sum_{m\geq 0} r_m = 2k, $$
    then $r_0+\frac{r_1}{2}\geq 2k-2i$.
\end{lemma}

\begin{proof}
    We have
    $$ \sum_{m\geq 0} (m+1)r_m = \sum_{m\geq 0} r_m+4i,$$
    thus
    $$ \frac{r_1}{2} + \sum_{m\geq 2} r_m \leq \frac{1}{2} \sum_{m\geq 1} m r_m = 2i. $$
    Therefore,
    \[
    r_0+\frac{r_1}{2} = 2k - \frac{r_1}{2} - \sum_{m\geq 2} r_m \geq 2k - 2i. \qedhere
    \]
\end{proof}

\subsection{Tail bound for the norm}
\label{sec:norm}

We will also need a concentration estimate on the norm of $B^N(t)$. One could obtain it by using Proposition \ref{prop:2kmom} coupled with an estimate on the norm of a GUE matrix and the triangle inequality. However, following the strategy developed in the proof of Proposition \ref{prop:2kmom} yields the desired bound with no need for another reference. Hence, we have the following result.

\begin{theorem}
	\label{ldnvsdkn2}
	There exists a constant $c_{\rho,T}$ such that for any $t\in [0,T]$ and $K>0$,
	$$ \P\left( \norm{B^N(t) } \geq K \right) \leq \left(\frac{c_{\rho,T}}{K}\right)^{N^{2/3}}. $$
\end{theorem}

\begin{proof}
    One has
   	\begin{align*}
		\P\left( \norm{B^N(t)} \geq K \right) 
        \le \P\left( \norm{\left|B^N(t)\right|^{2k}} \geq K^{2k} \right) 
        &\leq \P\left(  \tr_N\left(\left|B^N(t)\right|^{2k}\right) \geq K^{2k} \right) \\
		&\leq \frac{N}{K^{2k}} \E\left[\ts_N\left(\left|B^N(t)\right|^{2k}\right) \right].
	\end{align*}
	The rest of the proof then follows exactly like those of Proposition \ref{prop:2kmom} and \ref{thm.2}.
\end{proof}

Letting $K>c_{\rho,t}$ so that $r=c_{\rho,t}/K\in(0,1)$, an application of the Borel--Cantelli lemma yields the following:

\begin{cor} \label{cor.bounded} There is some constant $K = K_{\rho,t}<\infty$ so that
\[ \mathbb{P}\{\|B^N(t)\|\le K\text{ for all large }N\}=1. \]
\end{cor}

\section{Estimates on Small Singular Values}
\label{sect:invertibility}

To lighten notation, in this section we drop the superscript $N$ from matrices $B^N(t),Z^N$, etc.

As discussed in Section \ref{sect.Brown.meas}, a key step in the Hermitization approach is to establish quantitative lower bounds on small singular value of shifts of $B(t)$. 
Specifically, for an arbitrary fixed $N\times N$ matrix $A$ we seek bounds of the form
\begin{equation}
    \label{bd:smallest}
    \P( \sigma_N( B(t)+ A) \le \delta) \ls N^C\delta^c \qquad \forall \delta\ge0
\end{equation}
for the smallest singular value, and
\begin{equation}
    \label{bd:smallish}
    \sigma_{N-\ell} ( B(t) + A) \ge c'(\ell/N)^{C'} \qquad \forall\, \ell\in[k, \kappa N]
\end{equation}
holding with high probability for some $k=o(N)$ and $\kappa>0$ independent of $N$.
Such control is provided by the following two propositions.

\begin{prop}[Smallest singular value]
\label{prop:smallest}
    Let $(B(t))_{t\geq 0}$ be a $(\rho,\zeta)$-Brownian motion. 
    Let $A$ be an $N\times N$ matrix, deterministic or random and independent of $(B(t))_{t\geq 0}$.
    Then the following hold for $0<t\le T<\infty$.
    \begin{enumerate}
        \item[(a)] If $|\zeta|<\rho$, then for all $\alpha>0$, \eqref{bd:smallest} holds for all $N\ge (4/\alpha)^{3/2}+ t^{-1/2}$, with $C=2$, $c=1-\alpha$ and the implicit constant depending only on $\alpha,\rho,\zeta,T$. 
        \item[(b)] (Degenerate case). If $|\zeta|=\rho$, then \eqref{bd:smallest} holds for universal constants $C,c\in (0,\infty)$ and all $N\ge t^{-c/C}$, with the implicit constant in \eqref{bd:smallest} depending only on $\rho$ and $T$.
    \end{enumerate}
\end{prop}

\begin{prop}[Mesoscopic singular values]
\label{prop:meso}
    In the setting of Proposition \ref{prop:smallest} (with no assumptions on $\zeta,\rho$)
    there are universal constants $C,c\in(0,\infty)$ such that for any $k\ge C\log N$, \eqref{bd:smallish} holds with $C'=2$,  $c'=c'(\rho,\zeta,T)>0$, $\kappa=\sqrt{t}$ with probability at least $1-\exp( -\min\{ ck, N^{2/3}\})$. 
\end{prop}

The key to the proof of both Propositions is Theorem \ref{thm.2}, which enables a reduction to the study of shifted Gaussian matrices, for which we have the following two lemmas.

\begin{lemma}
\label{lem:smallest}
    Let $G$ be an $N\times N$ complex Ginibre or GUE matrix (having the distribution of $Z^N(1)$ or $X^N(1)$ in Definition \ref{def:GinibreBM}),
    and let $A$ be an arbitrary complex $N\times N$ matrix, deterministic or random and independent of $G$.
    Then
    \begin{equation}
        \P( \sigma_N(G+A) \le \delta) \le N^C\delta^c\qquad \forall\,\delta\ge0
    \end{equation}
    for universal constants $C,c\in(0,\infty)$.
\end{lemma}

For the Ginibre case see \cite[Lemma 3.3]{sinmin}, where it is established with optimal constants $C=c=2$. (For a short argument with sub-optimal constants see \cite[$\mathsection$4.4]{Bordenave-Chafai-circular}.)
    The GUE case appears to be new (note the perturbation $A$ is not assumed to be Hermitian). We defer the proof and a discussion of related literature to Appendix \ref{app:smallest}.

\begin{lemma}\label{lem:meso}
    There is a universal constant $c>0$ such that the following holds. 
    Let $G$ be an $N\times N$ matrix with standard real Gaussian entries, and assume the $N(N+1)/2$ pairs $\{(G_{ij},G_{ji}): 1\le i\le j\le N\}$ are jointly independent. 
    Then for any $N\times N$ complex matrix $A$, deterministic or random and independent of $G$, and any $k\ge c^{-1}\log N$,
    \[
    \sigma_{N-\ell} ( G + A) \ge c\ell/N\qquad \forall \ell\, \ge k
    \]
    with probability at least $1-e^{-ck}$.
\end{lemma}

The proof of Lemma \ref{lem:meso} is deferred to Appendix \ref{app:meso}.

We note that the unshifted matrix $B(t)$ is always well invertible with high probability:

\begin{lemma}
	\label{lem3}
    For any $0<t\le T<\infty$ and $\delta>0$,
    \[
    \P( \sigma_N(B(t)) \le \delta ) = \mathcal{O}_{\rho,T}(\delta)^{N^{2/3}}.
    \]
\end{lemma}

\begin{proof}
	For any $t>0$ we have
	\begin{align*}
		\P\left(\sigma_N(B(t)) \leq \delta \right) &= \P\big( \big\| B(t)^{-1}\big\| \geq 1/\delta \big).
	\end{align*}
	However, following \eqref{eq.SDE.Binv}, if we define $(\Binv(t))_{t\geq 0}$ as the solution of the following stochastic differential equation: 
	\[
    \Binv(t) = I - \int_{0}^{t} dW(s) \Binv(s) + \frac{\zeta}{2} \int_{0}^{t} \Binv(s) ds
    \]
	we get thanks to classical stochastic calculus that $B(t) \Binv(t) = I$ for all $t\geq 0$ (see \cite[Section 4.2]{Kemp2016} for further details). In particular,
    $$ \P(\sigma_N\left(B(t)) \leq \delta \right) = \P\left( \norm{\Binv(t)} \geq 1/\delta \right)= \P\left( \norm{\Binv(t)^*} \geq 1/\delta \right). $$
    But if ${B}^\dagger(t) = ((B(t))^{-1})^*$, and $W^\dagger(t) = - \left(W(t)\right)^*$, we have that
    \begin{equation}
        \label{eq:inverse}
        B^\dagger(t) = I + \int_{0}^{t} B^\dagger(s) d\widetilde{W}(s)  + \frac{\overline{\zeta}}{2} \int_{0}^{t} B^\dagger(s) ds.
    \end{equation}
	Since $(W^\dagger(t))_{t\geq 0}$ has the same law as $(W_{\rho,\overline{\zeta}}(t))_{t\geq 0}$, we have that $(B^\dagger(t))_{t\geq 0}$ is a $(\rho,\overline{\zeta})$-Brownian motion. Hence the conclusion follows thanks to Theorem \ref{ldnvsdkn2}.
\end{proof}

We turn to prove Propositions \ref{prop:smallest} and \ref{prop:meso}. 
We will first handle the case that $|\zeta|<\rho$ and later note the modifications needed for the degenerate case. 

With $|\zeta|<\rho$, and recalling the parametrization \eqref{parametrization}, we may assume without loss of generality that $0< b\le |a|$.

Given independent $(\rho,\zeta)$-Brownian motions $(B(t))_{t\geq 0}$, $(\wt B(t))_{t\geq 0}$  driven respectively by $(W(t))_{t\geq 0}$, $(\widetilde{W}(t))_{t\geq 0}$, we have for any $0\le \varepsilon\le t$ the equality in law
\begin{equation}    \label{BtBe}
    B(t)\eqd B(t-\varepsilon)\wt B(\varepsilon). 
\end{equation}
Next we show that up to a small error we can replace $\wt B(\varepsilon)$ with a shift $cG+M$ of a scaled Ginibre matrix $G$ for some $c=c(\varepsilon)>0$, where $G$ and $M$ are independent. 
Indeed, note that if $\alpha,\beta\in\R$ and $X,Y$ are independent GUE matrices of size $N$, then $\alpha X+\beta Y$ has the same law as $\sqrt{\alpha^2+\beta^2}Z$ where $Z$ is also a GUE matrix. Furthermore, $\widetilde{W}(\varepsilon)$ has the same law as
\begin{equation}
    \label{eq:decomp0}
    \sqrt{\varepsilon} e^{\i \theta} \left(aX + \i b Y\right). 
\end{equation}
Thus, since $|a|\geq b\ge 0$, we get that $\widetilde{W}(\varepsilon)$ has the same law as
\begin{equation}
\label{eq:decomp}
    b \sqrt{\varepsilon} e^{\i \theta} \left(X + \i Y\right) +  \sqrt{\vert \zeta \vert} \sqrt{\varepsilon} e^{\i \theta} \widetilde{X},
\end{equation}
where $\widetilde{X}$ is an independent copy of $X$. 
Consequently, 
\begin{equation}    \label{tWG}
    \widetilde{W}(\varepsilon)\eqd
    b \sqrt{2\varepsilon} G + \sqrt{\vert \zeta \vert} \sqrt{\varepsilon} e^{\i \theta} \widetilde{X}
\end{equation}
where $G$ is a Ginibre random matrix independent of $\widetilde{X}$. 

For $\varepsilon,\delta_0,\delta_1\ge0$, denote the event
\begin{equation}        \label{ssv.good}
    \cG=\cG(\delta_0,\delta_1,\eps):= \big\{\, \sigma_N(B(t-\varepsilon))\ge \delta_0,\, \| \wt B(\varepsilon)-I-\wt W(\varepsilon)\|\le\delta_1\,\big\}\,.
\end{equation}
From Lemma \ref{lem3} and Theorem \ref{thm.2},
\begin{equation}    \label{bd:ssv.good}
    \P( \cG(\delta_0,\delta_1,\varepsilon)^{\mathrm{c}}) \le \mathcal{O}_{\rho,T}(\delta_0)^{N^{2/3}} + \mathcal{O}_{\rho,T}(\varepsilon/\delta_1)^{N^{2/3}}.
\end{equation}
On $\cG$, for any $0\le \ell\le N-1$ we have
\begin{align*}
    \sigma_{N-\ell}( B(t-\varepsilon)\wt B(\varepsilon) + A) 
    &\ge \delta_0\, \sigma_{N-\ell}\big(\,\wt B(\varepsilon) + B(t-\varepsilon)^{-1}A\,\big)
\end{align*}
and
\begin{align*}
    \sigma_{N-\ell}\big(\,\wt B(\varepsilon) + B(t-\varepsilon)^{-1}A\,\big)
    &\ge \sigma_{N-\ell}\big(\, \wt W(\varepsilon) + I + B(t-\varepsilon)^{-1}A \,\big) - \delta_1\,.
\end{align*}
Combining with \eqref{BtBe} and \eqref{tWG} we have that for any $\delta\ge0$,
\begin{align}
    \P\Big(\, \sigma_{N-\ell}( B(t) + A) \le \delta\,,\, \cG\, \Big) &= \P\Big(\, \sigma_{N-\ell}( B(t-\varepsilon) \wt B(\varepsilon)+ A) \le \delta\,,\,\cG\,\Big)    \notag\\
    &\le \P \bigg(\, \sigma_{N-\ell}\big( b\sqrt{2\varepsilon} G + M \big) \le \frac{\delta}{\delta_0} + \delta_1 \bigg) \label{bd:sigNell}
\end{align}
where
\[
M:= I+ \sqrt{\vert \zeta \vert} \sqrt{\varepsilon} e^{\i \theta} \widetilde{X} +  B(t-\varepsilon)^{-1}A
\]
is independent of the Ginibre matrix $G$. 

To conclude the proof of Proposition \ref{prop:smallest} for the case $|\zeta|<\rho$ we use Lemma \ref{lem:smallest}, along with \eqref{bd:ssv.good} and \eqref{bd:sigNell} with $\ell=0$ and $\varepsilon,\delta_0,\delta_1>0$ to be chosen, to get
\begin{align*}
    &\P\Big(\, \sigma_{N}( B(t) + A) \le \delta\Big)
    \le \frac{N^2}{2b^2\varepsilon}\bigg( \frac {\delta}{\delta_0}+ \delta_1 \bigg)^2 + \mathcal{O}_{\rho,T}( \delta_0+ \varepsilon\delta_1^{-1})^{N^{2/3}}.
\end{align*}
The claim follows upon taking $\varepsilon:=\delta$, $\delta_0:=\delta^{\alpha/2}$ and $\delta_1:= \delta^{1-\alpha/2}$.
(Note that we may assume $\delta\le N^{-2}$ as the claimed bound is trivial otherwise, so to ensure $\eps=\delta\le t$ it suffices to have $N\ge t^{-1/2}$.) 

Turning to conclude the proof of Proposition \ref{prop:meso} for the case $|\zeta|<\rho$, from a union bound over the choices of $\ell$, it suffices to show that for any fixed $\ell \ge C\log N$, 
\begin{equation}
    \label{ssv-goal}
    \P( \sigma_{N-\ell}(B(t) + A) \le c(\ell/N)^2) \le e^{-c\ell} + \exp( -2N^{2/3}). 
\end{equation}
Taking $\delta_0=c$, $\delta=\delta_1 =c^2(\ell/N)^2$, and $\varepsilon=c\delta$ for sufficiently small $c=c(T,\rho,\zeta)>0$, from \eqref{bd:sigNell} and \eqref{bd:ssv.good} we can bound the left hand side above by 
\[
\P\bigg( \sigma_{N-\ell}\Big( G+ \frac1{b\sqrt{2\varepsilon}}M\Big) \le \frac{2\sqrt{c}}{ b}\frac\ell N \bigg) + \exp( - 2N^{2/3}).
\]
Taking $c$ smaller if necessary, Proposition \ref{prop:meso} now follows from Lemma \ref{lem:meso} after conditioning on the imaginary part of $G$ (rescaling by a factor $\sqrt{2}$ to standardize the real parts of the entries). From the constraint $\eps\le t$, with this choice of parameters we can take $\ell$ as large as $C_0\sqrt{t}N$ for any fixed $C_0>0$ by shrinking $c$. 

We turn to the degenerate case $|\zeta|=\rho$. 
Now we may assume $b=0<a=\sqrt{\rho}$. 
Returning to \eqref{eq:decomp0}, we see that
$\widetilde{W}(\varepsilon)$ has the same law as
\begin{equation}
\label{eq:decomp1}
    \sqrt{a\varepsilon} e^{\i \theta} X,
\end{equation}
where $X$ is a GUE matrix.
Following the same lines we get that for any $\ell\ge0$
\begin{align}
    \P\Big(\, \sigma_{N-\ell}( B(t) + A) \le \delta\,,\, \cG\, \Big) 
    &\le \P \bigg(\, \sigma_{N-\ell}\big( \sqrt{a\varepsilon}  X + M' \big) \le \frac{\delta}{\delta_0} + \delta_1 \bigg) \label{bd:sigNell1}
\end{align}
where now $M':= e^{-\i\theta}(I+   B(t-\varepsilon)^{-1}A)$
is independent of the GUE matrix $X$. 
The proof now concludes along the same lines as for the case $|\zeta|<\rho$, using the GUE case of Lemma \ref{lem:smallest}, and applying Lemma \ref{lem:meso} to $\sqrt{2}$ times the real part of $X$. 
\qed

\section{An Analytic Approach for Mesoscopic Singular Values\label{sect.analytic.approach}}

In this section we provide an alternative argument for a lower bound on small singular values of $B_0^NB^N(t)-zI$, as in Proposition \ref{prop:meso}.
That proposition implies that the empirical cumulative distribution function of shifted singular values is essentially H\"older-$\frac12$ continuous near $0$, at shrinking scales, for all $z\in\C$.  This is enough (in conjunction with the bound on polynomial decay of the smallest singular value in \ref{prop:smallest}) to yield convergence of the empirical law of eigenvalues; in fact, it would be enough to have a H\"older continuity bound (of varying exponent) just at Lebesgue almost every $z\in\C$.

In this section, we show that: at least for the Brownian motion $B^N(t)$ started at the identity (i.e.\ $B^N_0=I$), the expected empirical CDF of singular values $\sigma_j(B^N(t)-zI)$ is actually Lipschitz near $0$, for Lebesgue almost all $z\in\C$.  In particular: the only $z$ where this may fail (and instead we get H\"older-$\frac12$ continuity) are on the boundary of the the support set of the Brown measure of the limit free multiplicative Brownian motion.

To state and prove this precisely, we make the same parameter reduction as in Notation \ref{notat.ovweparam}: Since $B^N_{\rho,\zeta}(t) \equaldist B^N_{t\rho,t\zeta}(1)$, we henceforth consider the random matrices and operators
\begin{equation} \label{eq.reparam.sect5}
 B^N(t,\zeta) = B^N_{t,\zeta}(1) \quad \& \quad b(t,\zeta) = b_{t,\zeta}(1).
 \end{equation}
Recall that the Brown measure of $b(t,\zeta)$ was identified in \cite{HallHo2023}; in particular, its support set is a compact region $\overline{\Sigma(1,t,\zeta)}$ defined in Theorem \ref{thm.Brown.meas.rhozeta}.  The boundary $\partial\Sigma(1,t,\zeta)$ is a simple smooth closed curve if $t<4$, a smooth ``figure-eight'' curve if $t=4$, and the union of two disjoint simple smooth closed curves if $t>4$.  In particular, the Lebesgue measure of $\partial\Sigma(1,t,\zeta)$ is $0$ for all $t>0$ and $|\zeta|\le t$.

\begin{prop} \label{prop.Wegner.bounded}
Let $T>0$ and $\zeta\in\C$ with $|\zeta|\le t$.  Let $z\in\C\setminus\partial\Sigma(1,t,\zeta)$. Then there is a constant $\overline{C}_{z,T}$ so that, for all $N\in\N$, all $t\in[0,T]$, and all $\eta> N^{-2/11}$,
\begin{equation} \label{eq.Wegner.bounded} \left|\mathrm{Im}\,\E\ts_N[(i\eta-|B^N(t,\zeta)-z I|)^{-1}]\right| \le \overline{C}_{z,T}. \end{equation}
\end{prop}

Proposition \ref{prop.Wegner.bounded} is a strong {\em Wegner anti-concentration estimate} (see \cite{Wegner}).  The (conclusion of the) proof is in Section \ref{sect.proof.Wegner} below.  It follows that the expected CDF of the empirical law of singular values of $B^N_{\rho,\zeta}(t)-zI$ is uniformly Lipschitz on mesoscopic-to-macroscopic scales.

\begin{cor}
\label{cor:wegner}
    For $0<t\le T<\infty$ and $z\in\mathbb{C}\setminus\partial\Sigma(1,t\rho,t\zeta)$, we have
    \[
    \E \mu_{|B_{\rho,\zeta}^N(t)-zI|}([0,\eta]) =\mathcal{O}_{\rho,z,T}( \eta)
    \]
    uniformly for $\eta\in[N^{-2/11},1]$.
\end{cor}

\begin{proof} 
Since $B^N_{\rho,\zeta}(t) \equaldist B^N_{t\rho,t\zeta}(1) = B^N(t\rho,t\zeta)$, it suffices to prove the corollary for $B^N(t,\zeta)$; i.e.\ for $z\in\C\setminus\partial\Sigma(1,t,\zeta)$,
\begin{equation} 
\label{eq.Lipschitz.1} 
\E\mu_{|B^N(t,\zeta)-zI|}([0,\eta]) = \mathcal{O}_{z,T}(\eta). 
\end{equation}
Now since $\mathrm{Im} (i\eta-x)^{-1} = -\eta/(\eta^2+x^2)$ for $x\in \R$, the claim follows from Proposition \ref{prop.Wegner.bounded} and the pointwise bound $\eta/(\eta^2+x^2)\ge (2\eta)^{-1}\mathbbm{1}_{x\in[0,\eta]}$.
\end{proof}

Corollary \ref{cor:wegner} can be substituted for Proposition \ref{prop:meso} in the proof of Theorem \ref{thm.main} in Section \ref{sect.final} (for the case $B_0^N=I$), following the argument from \cite{GuionnetKZ-single-ring}.

\begin{rem} It is worth noting that, for any matrix $A^N\in\M_N(\C)$, the spectral measure $\mu_{|A^N|}$ is exactly the empirical law of singular values $\sigma_j(A^N)$; thus
\[ \mu_{|A^N|}([0,\eta]) = \frac1N\sum_{j=1}^N \delta_{\sigma_j(A^N)}([0,\eta]) = \frac{\#\{j\le N\colon \sigma_j(A^N)\le \eta\}}{N}. \]
This is indeed the empirical CDF of singular values.  Hence, Corollary \ref{cor:wegner} says that the expected empirical CDF of singular values of $B^N_{\rho,\zeta}(t)-zI$ is Lipschitz near $0$, with a Lipschitz constant that is is locally uniform in $t$, for all $z$ except possibly on the boundary of the domain $\Sigma(1,t\rho,t\zeta)$.
\end{rem}

The quantity on the left-hand-side of \eqref{eq.Wegner.bounded} is the (imaginary part of the) expected Cauchy transform $G_{|B^N(t,\zeta)-zI|}(i\eta)$.  Our approach to prove Proposition \ref{prop.Wegner.bounded} is to compare this to the Cauchy transform of the large-$N$ limit operator, $G_{|b(t,\zeta)-z|}(i\eta)$.  As shown in Section \ref{sect.C5} below, this difference is $\mathcal{O}(\frac{1}{N^2})$ with constants that depend polynomially on $1/\eta$; then by choosing $\eta>N^{-\alpha}$ for an appropriate power $\alpha>0$, the bound \eqref{eq.Wegner.bounded} reduces to a similar bound on the free multiplicative Brownian motion; i.e.\ it reduces to showing $G_{|b(t,\zeta)-z|}(i\eta)$ is bounded for $\eta>0$.  This is accomplished in Section \ref{sect.PDE}, primarily using analytic results derived from the PDE methods in \cite{DHKBrown,HallHo2023} together with spectral bounds from \cite{HallHo2025Spectrum}.

\subsection{Smooth Functional Comparison to the Large-$N$ Limit\label{sect.C5}}

Throughout this section, our estimates will be formulated in terms of the overparametrized $B^N(t) = B^N_{\rho,\zeta}(t)$; they will be converted to the reduced parametrization $B^N(t,\zeta)$ when needed in Section \ref{sect.proof.Wegner}.  Additionally, the tools in this section apply equally well for a deterministic initial condition $B_0^N$, so we prove them in that level of generality for potential future use.

We begin with the following lemma.

\begin{lemma}
	\label{skjdbcnks}
    Let $B_0^N$ be a deterministic matrix. There exists a constant $C_{\rho,t}$ such that, for any $y\in\R$, we have that
    \begin{align*}
        &\left| \E\left[\ts_N\left(e^{\i y \left|B^N_0B^N(t)-z\right|^2}\right)\right] - \tau_N\left(e^{\i y |B^N_0b(t)-z|^2}\right) \right| \\
        & \leq C_{\rho,t} \frac{(1+y^4)(1+|z|^4)(1+\norm{B^N_0}^4)}{N^2},
    \end{align*}
    where we used the construction of the space $\A_N$ of Lemma \ref{lem:bigspace} to define the product of $B^N_0$ and $b(t)$. 
\end{lemma}

\begin{proof}

    Given two operators $A$ and $B$, one has that
    $$ e^A-e^B = \int_0^1 e^{\alpha A} (A-B) e^{(1-\alpha)B} d\alpha.$$
    Consequently,
    \begin{align*}
        &\left| \E\left[\ts_N\left(e^{ \i y \left|B^N_0B^N(t)-z\right|^2}\right)\right] - \E\left[\ts_N\left(e^{\i y \left|B^N_0B^N_n(t)-z\right|^2}\right)\right] \right| \\
        &= \Bigg| y \int_0^1 \E\Big[\ts_N\Big(e^{ \i\alpha y \left| B^N_0B^N(t) - z\right|^2 } \Big( \left| B^N_0B^N(t) - z\right|^2 \\
        &\quad\quad\quad\quad\quad\quad\quad\quad - \left| B^N_0 B^N_n(t) - z\right|^2 \Big) e^{\i y (1-\alpha) \left| B^N_0 B^N_n(t) - z\right|^2}\Big)\Big] d\alpha \Bigg| \\
        &\leq \left| y \right| \norm{B^N_0} \left( \norm{B^N_0B^N(t)-z}_2 + \norm{B^N_0B^N_n(t)-z}_2 \right) \norm{B^N(t) - B^N_n(t)}_2,
    \end{align*}
    where we used that for any self-adjoint operator $X$, the operator norm of $e^{\i y X}$ is equal to $1$. Thus, thanks to Lemma \ref{sdovmkson} and Proposition 
    \ref{prop:Lp.conv},
    $$ \E\left[\ts_N\left(e^{\i y B^N_0 B^N(t)}\right)\right] = \lim_{n\to\infty} \E\left[\ts_N\left(e^{\i y B^N_0B^N_n(t)}\right)\right] $$
    and similarly,
    $$ \tau_N(e^{\i y B^N_0 b(t)}) = \lim_{n\to\infty} \tau_N(e^{ \i y B^N_0 b_n(t)}). $$
    Finally, thanks to \cite[Lemma 3.6]{Parraud2023}, one has that
    \begin{align*}
        &\E\left[\ts_N\left(e^{\i y \left|B^N_0B^N(t)-z\right|^2}\right)\right] - \tau_N\left(e^{\i y |B^N_0b(t)-z|^2}\right) \\
        &= \frac{1}{N^2} \int_0^{\infty}\int_0^u e^{-v-u}\int_{[0,1]^4} \E\left[\tau_N\left( R_{\alpha_1,\alpha_2,\alpha_3,\alpha_4,}^{n,N,y} \right)\right] \ d\alpha_1d\alpha_2d\alpha_3d\alpha_4 dvdu,
    \end{align*}
    where 
    $$\frac{\left|\E\left[\tau_N\left( R_{\alpha_1,\alpha_2,\alpha_3,\alpha_4,}^{n,N,y} \right)\right]\right|}{(1+y^4)(1+|z|^4)(1+\norm{B^N_0}^4)}$$
    can, up to a universal constant, be upper bounded by 
    $$ \sup_{1\leq l\leq m\leq \lfloor 2^n t \rfloor} \E\left[\ts_N\left( \left|\prod_{i=l}^m \left(I_N + \frac{e^{\i\theta}}{2^{n/2}}\left(aX_{i,\{1,2\}}^{N,\{v,u\}}+\i b Y_{i,\{1,2\}}^{N,\{v,u\}}\right) +\frac{\zeta}{2^{n+1}} I_N\right) \right|^{12} \right)\right], $$
    where the variables $X_{i,\{1,2\}}^{N,\{v,u\}}$ and $Y_{i,\{1,2\}}^{N,\{v,u\}}$ are defined as in Lemma \ref{3apparition}. Besides, by the same proof as in Lemma \ref{sdovmkson}, the quantity above is bounded by a constant which only depends on $t$ and $\rho$ since $|\zeta| \leq \rho$. Hence the conclusion.
\end{proof}

The following corollary is the main tool to handle the global regime.

\begin{cor}
    \label{ksdvkjvbdw}
	There exists a constant $C_{\rho,t}$ such that, for any functions $f\in\CC^5(\R)$ with compact support bounded by $K$,
    \begin{equation} \label{eq.comparison.C5}
    \begin{aligned}
        &\left| \E\left[\ts_N\left(f\left(\left|B^N_0B^N(t)-z\right|^2\right)\right)\right] - \tau_N\left(f\left( |B^N_0b(t)-z|^2\right)\right) \right| \\
        &\leq C_{\rho,t} \sqrt{K} \frac{(1+|z|^4)(1+\norm{B^N_0}^4)\norm{f}_{\mathcal{C}^5(\R)}}{N^2},
    \end{aligned}
    \end{equation}
    where $\norm{f}_{\mathcal{C}^5(\R)} = \sum_{k=0}^5\sup_{x\in\R} |f^{(k)}(x)|$. 
\end{cor}

\begin{proof}
    Let $f\in L^1(\R)$, and set
    $$ \widehat{f}(y) \deq \frac{1}{2\pi} \int_{\R} e^{-\i yx} f(x) dx.$$
    Then if $\widehat{f}\in L^1(\R)$, by Fourier inversion theorem,
    $$ f(x) = \int_{\R} e^{\i x y} \widehat{f}(y) dy.  $$
    Besides, if $f$ is sufficiently smooth, then $\left|y^k\widehat{f}(y) \right|= \left|\widehat{f^{(k)}}(y)\right|$, where $f^{(k)}$ is the $k$-th derivative of $f$. Consequently,
    \begin{align*}
        &\int_{\R} |y^k| \left|\widehat{f}(y)\right| dy
        = \int_{\R} \frac{(|y^k| + |y^{k+1}|) \left|\widehat{f}(y)\right| }{1+|y|} dy 
        \leq \int_{\R} \frac{\left|\widehat{f^{(k)}}(y)\right| +\left|\widehat{f^{(k+1)}}(y)\right| }{1+|y|} dy \\
        &\leq \left( \left( \int_{\R} \left|\widehat{f^{(k)}}(y)\right|^2 dy \right)^{1/2} + \left(\int_{\R} \left|\widehat{f^{(k+1)}}(y)\right|^2 dy \right)^{1/2} \right) \left(\int_{\R}\frac{1 }{(1+|y|)^2} dy\right)^{1/2} \\
        &\leq \frac{1}{\sqrt{2\pi}} \left( \left( \int_{\R} \left|f^{(k)}(y)\right|^2 dy \right)^{1/2} + \left(\int_{\R} \left|f^{(k+1)}(y)\right|^2 dy \right)^{1/2} \right) \left(\int_{\R}\frac{1 }{(1+|y|)^2} dy\right)^{1/2} \\
        &\leq \sqrt{K}\left(\frac{1}{2\pi}\int_{\R}\frac{1 }{(1+|y|)^2} dy\right)^{1/2} \left( \sup_{y\in\R} \left|f^{(k)}(y)\right|+ \sup_{y\in\R}\left|f^{(k+1)}(y)\right| \right),
    \end{align*}
    where we used Plancherel theorem on the second line. Thus, we have
    \begin{align*}
        \int_{\R} (1+y^4) \left|\widehat{f}(y)\right| dy &\leq \sqrt{K}\left(\frac{1}{2\pi}\int_{\R}\frac{1 }{(1+|y|)^2} dy\right)^{1/2} \norm{f}_{\mathcal{C}^5(\R)}.
    \end{align*}
    Hence the conclusion by using Lemma \ref{skjdbcnks}.
\end{proof}

\subsection{Free Wegner Estimate\label{sect.PDE}}

Here we prove a version of the desired Wegner estimate in Proposition \ref{prop.Wegner.bounded} but for the free multiplicative Brownian motion already in the large-$N$ limit.  The reader may wish to review Section \ref{sect.Brown(ian).meas} for the definitions of the regions $\Sigma(1,t)=\Sigma(1,t,0)$ and more generally $\Sigma(1,t,\zeta)$; here the $1$ is the initial condition $b_0=u_0=1$ for the Brownian motion.

\begin{lemma} \label{lem.free.Wegner} Let $T>0$ and $\zeta\in\C$ with $|\zeta|\le t$.  Let $z\in\C\setminus\partial\Sigma(1,t,\zeta)$. Then there is a constant $\overline{C}_{z,T}$ so that, for  all $t\in[0,T]$, and all $\eta>0$,
\begin{equation} \label{eq.free.Wegner.bounded} \left|\mathrm{Im}\,\tau[(i\eta-|b(t,\zeta)-z |)^{-1}]\right| \le \overline{C}_{z,T}. \end{equation}
\end{lemma}

\begin{proof} The potentially bad set $\partial\Sigma(1,t,\zeta)$, which is of Lebesgue measure $0$, separates $\C$ into two regions. We treat separately the two cases: where $z\in\Sigma(1,t,\zeta)$, and where $z\notin\overline{\Sigma(1,t,\zeta)}$.  We begin with the latter.

In \cite[Theorem 3.7]{HallHo2025Spectrum}, it was shown that $\overline{\Sigma(1,t,\zeta)}$ is precisely equal to the spectrum of $b(t,\zeta)$.  Hence, for $z\notin\overline{\Sigma(1,t,\zeta)}$, $b(t,\zeta)-z$ has a bounded inverse, and therefore so does $|b(t,\zeta)-z|$ (by the spectral theorem).  For readability, set $x=|b(t,\zeta)-z|$.  We can therefore compute that
\begin{align*} (i\eta-x)^{-1} = -x^{-1}(1-i\eta x^{-1})^{-1} = -x^{-1}\sum_{k=0}^\infty (ix^{-1})^k \eta^k
\end{align*}
which converges in $\A$ provided $|\eta|< \|x^{-1}\|^{-1}$.  Thus
\[ \tau[(i\eta-|b(t,\zeta)-z|)^{-1}] = -\sum_{k=0}^\infty i^k \tau[(x^{-1})^{k+1}]\eta^k \]
has a convergent power series expansion for $|\eta|<\| |b(t,\zeta)-z|^{-1} \|^{-1}$, and hence it is bounded for $0<\eta<\||b(t,\zeta)-\lambda|^{-1} \|^{-1}$.  Since this norm is continuous in $t$, the bound is locally uniform in $T$, as stated.

Now, consider the case that $z\in\Sigma(1,t,\zeta)$.  Define
\[ S(t,z,\zeta,\eta) = \tau[\log(|b(t,\zeta)-z|^2+\eta^2)]. \]
(Then $\lim_{\eta\downarrow0}S(t,z,\zeta,\eta) = 4\pi U_{b(t,\zeta)}(z)$ where $U_{b(t,\zeta)}$ is the log potential and the Brown measure of $b(t,\zeta)$ is the Laplacian  of this limit, see Definition \ref{def.BrownMeasure}.)
\[ \left.\frac{\partial}{\partial\eta}S(t,z,\zeta,\eta)\right|_{\eta=\eta_0} =2\eta_0\,\tau[(|b(t,\zeta)-z|^2+\eta_0^2)^{-1}] = -2\mathrm{Im}\,\tau[(i\eta_0-|b(t,\zeta)-z |)^{-1}] \]
where the second equality follows just as in Corollary \ref{cor:wegner} from the identity $\mathrm{Im}(i\eta-x)^{-1} = -\eta/(\eta^2+x^2)$ for $x\in\R$.

Now, \cite[Theorem 8.2]{HallHo2023} (or more precisely its proof) showed that there is a diffeomorphism $\Phi_{\zeta}\colon\Sigma(1,t)\to\Sigma(1,t,\zeta)$ such that
\[ S(t,z,\zeta,\eta) = S(t,\Phi_\zeta^{-1}(z),0,\eta). \]
Moreover, the diffeomorphism does not depend on $t$ or $\eta$.  Hence
\begin{align*}
\mathrm{Im}\,\tau[(i\eta_0-|b(t,\zeta)-z |)^{-1}] = -\frac12 \left.\frac{\partial}{\partial\eta}S(t,z,\zeta,\eta)\right|_{\eta=\eta_0} = \left.\frac{\partial}{\partial\eta}S(t,\Phi_{\zeta}^{-1}(z),0,\eta)\right|_{\eta=\eta_0}
\end{align*}

By \cite[Theorem 6.4]{DHKBrown}, for $z'\in\Sigma(1,t)$ the function $(t,\eta)\mapsto S(t,z',0,\eta)$ has real analytic continuation to $\R_+\times (-\delta,\delta)$ where $\delta$ may depend on $z$, but not $t$.  (The statement of that theorem suggests the analytic continuation is only local in $t$ as well; however, in the proof, the actual extension implicitly computed in \cite[Eq.\ (6.20)]{DHKBrown}, is analytic on $\R_+$ and continuous on $[0,\infty)$.)  Since $z'=\Phi_{\zeta}^{-1}(z)\in\Sigma(1,t)$ precisely when $z\in\Sigma(1,t,\zeta)$, it therefore follows that the $\frac{\partial}{\partial\eta}$ derivative above is continuous and bounded, locally uniformly in $t$, for $0<\eta_0<\delta/2$ (where $\delta$ may depend on $z$).

Hence, for such $z$, the derivative $\frac{\partial}{\partial\eta} S(t,z,\zeta,\eta)$ is continuous and bounded for $0<\eta<\delta/2$.  What's more, the Cauchy transform $|[(i\eta-|b(t,\zeta)-z |)^{-1}]|$ is $\le 2/\delta$ for all $\eta\ge\delta/2$.  This proves the bound for all $\eta>0$, concluding the proof.
\end{proof}

\subsection{Proof of Proposition \ref{prop.Wegner.bounded}\label{sect.proof.Wegner}}

We now combine Corollary \ref{ksdvkjvbdw} with Lemma \ref{lem.free.Wegner} to prove Proposition \ref{prop.Wegner.bounded}.

Taking $B_0^N=I$ and reparametrizing per \eqref{eq.reparam.sect5}, \eqref{eq.comparison.C5} yields
\begin{equation} \label{eq.Wegner.final.1}
\begin{aligned} & \left|\E\left[\ts_N\left(f(|B^N(t,\zeta)-zI|^2)\right)\right] -\tau\left(f(|b(t,\zeta)-z|^2)\right)\right| \\
&\le  2C_{t\rho,t}\sqrt{K}\frac{(1+|z|^4)\|f\|_{\mathcal{C}^5(\R)}}{N^2}
\end{aligned}
\end{equation}
for any function $f\in\mathcal{C}^5(\R)$ that is compactly supported with support of diameter $\le K$.  Let $M=M(t,z,\zeta)<\infty$ be a constant large enough that $\||B^N(t,\zeta)-zI|^2\|\le M$ a.s.\ for all large $N$ (see Corollary \ref{cor.bounded}) and $\| |b(t,\zeta)-z|^2\|\le M$. For given $\eta\in(0,1]$, let $f_\eta$ be a smooth, compactly supported function satisfying
\begin{itemize}
    \item $f_\eta(x) = -\frac{\eta}{\eta^2+x}$ for $x\in[0,M]$
    \item For each $k\le 5$, the maximum derivative $\max_{x\in\R}|f_\eta^{(k)}(x)|$ does not exceed the maximum derivative $\max_{x\in[0,M]}|f_\eta^{(k)}(x)| = |f_\eta^{(k)}(0)| = \frac{k!\eta}{(\eta^2)^{k+1}} = k!\eta^{-2k-1}$.
\end{itemize}
This can be accomplished with a smooth function $f_\eta$ whose support diameter $K$ is uniform in $\eta\in(0,1]$: the smaller $\eta>0$ is, the larger the maximum derivative of the approximation is allowed to be.  Hence, $K=K(M)$ can be taken to depend on $M$ independent of $\eta$.

Then, since the spectra of $|B^N(t,\zeta)-zI|^2$ and $|b(t,\zeta)-z|^2$ are contained in the interval $[0,M]$, it follows that
\begin{align*} \ts_N\left(f(|B^N(t,\zeta)-zI|^2)\right) &= -\lambda \ts_N\left((\eta^2+|B^N(t,\zeta)-zI|^2)^{-1}\right) \\
&= \mathrm{Im}\,\ts_N[(i\eta-|B^N(t,\zeta)-zI|)^{-1}] \\
\& \qquad \tau\left(f(|b(t,\zeta)-z|^2)\right) &= -\lambda \tau\left((\eta^2+|b(t,\zeta)-z|^2)^{-1}\right) \\
&= \mathrm{Im}\,\tau[(i\eta-|b(t,\zeta)-z|)^{-1}].
\end{align*}
What's more, the second condition on $f_\eta$ shows that
\[ \|f_\eta\|_{\mathcal{C}^5} = \sum_{k=0}^5 k!\eta^{-2k-1} \le 154 \eta^{-11}. \]
Hence, \eqref{eq.Wegner.final.1} shows that
\begin{align*} &\left| \E\left(\mathrm{Im}\,\ts_N[(i\eta-|B^N(t,\zeta)-zI|)^{-1}]\right) - \mathrm{Im}\,\tau[(i\eta-|b(t,\zeta)-z|)^{-1}]\right| \\
&\le 2C_{t\rho,t}\sqrt{K(M)}\frac{(1+|z|^4)\cdot 154\eta^{-11}}{N^2}.
\end{align*}
That is to say: for all $N\in\N$ and all $\eta\in(0,1]$,
\begin{align} \nonumber &\left|\E\left(\mathrm{Im}\,\ts_N[(i\eta-|B^N(t,\zeta)-zI|)^{-1}]\right)\right| \\ \label{eq.Wegner.final.2}
\le& \left|\mathrm{Im}\,\tau[(i\eta-|b(t,\zeta)-z|)^{-1}]\right| + 308C_{t\rho,t}\sqrt{K(M)}(1+|z|^4)\frac{\eta^{-11}}{N^2}.
\end{align}
By Lemma \ref{lem.free.Wegner}, the first term is bounded above by a constant $\overline{C}_{z,T}$ for all $\eta>0$ and for all $t\in[0,T]$.  For the second term, if we take $\eta\ge N^{-2/11}$ then $\eta^{-11}/N^2\le 1$.  This proves boundedness.

Lastly, we consider the dependence of the bound on $t$.  The constant $K(M)$ depends on $z,\zeta,t$ only through the uniform upper bound on $\||B^N(t,\zeta)-zI|^2\| \le \|B^N(t,\zeta)\|^2 + 2|z|\|B^N(t,\zeta)\| + |z|^2$.  The strong convergence result of \cite{BCC2025} shows that $\|B^N(t,\zeta)\|$ is bounded uniformly in $N$ by a smooth exponential function of $t$.  Finally, the constant $C_{t\rho,t}$ is increasing in $t$, as is clear from the last display equation in the proof of Lemma \ref{skjdbcnks} defining it.  This shows, therefore, that the bound in Proposition \ref{prop.Wegner.bounded} is indeed locally uniform in $t$, concluding the proof.

\section{
Convergence of Eigenvalues\label{sect.final}}

In this final section, we briefly outline again the Hermitization procedure for proving convergence of empirical eigenvalue distributions, and put it together with the estimates bounding the decay rates of shifted singular values we proved in Sections \ref{sect:invertibility} and \ref{sect.analytic.approach} to prove our main Theorem \ref{thm.main}.  We then prove one more result: combined with the established strong convergence of the processes $B_{\rho,\zeta}$, and recent results on the spectrum of the free limit $b_{\rho,\zeta}$, we prove that the (random) spectrum of $B_{\rho,\zeta}(t)$ converges a.s.\ in Hausdorff metric to the spectrum of $b_{\rho,\zeta}(t)$ as $N\to\infty$.

\subsection{Proof of Theorem \ref{thm.main}: Convergence of the Empirical Law}
\label{sec:cvg-esd}

We begin by recalling the following lemma encapsulating the Hermitization procedure outlined in Section \ref{sect.Brown.meas}.

\begin{lemma}[{\cite[Lemma 4.3]{Bordenave-Chafai-circular}}]    \label{lem:BoCh}
    Let $(A^N)_{N\ge1}$ be a sequence of random matrices where $A^N$ is $N\times N$. 
    Suppose there exists a family of non-random probability measures $(\nu_z)_{z\in\C}$ on $\R_+$ such that for a.e.\ $z\in\C$, almost surely:
    \begin{enumerate}
        \item[(a)]\label{BoCh.a} $\mu_{|A^N-zI_N|}\to \nu_z$, and 
        \item[(b)]\label{BoCh.b} $\log$ is uniformly integrable 
        for $(\mu_{|A^N-zI_N|})_{N\ge1}$.
    \end{enumerate}
    Then there exists a probability measure $\mu$ on $\C$ such that
    \begin{enumerate}
        \item[(a')]\label{BoCh.a'} $\mu_{A^N}\to\mu$ weakly 
        almost surely, and
        \item[(b')]\label{BoCh.b'} for a.e.\ $z\in\C$, we have
        \[
        \int_\C \log|\lambda-z| d\mu(\lambda) =\int_0^\infty\log(s)d\nu_z(s).
        \]
    \end{enumerate}
\end{lemma}

We proceed to prove Theorem \ref{thm.main}.
Fix $t>0$. 
Lemma \ref{lem.init.cond.*-conv} verifies condition \ref{BoCh.a} of the above lemma for $A^N=B_0B(t)=B_0^NB^N(t)$, with $\nu_z = \nu_{b_0b(t)-z}$ as in Definition \ref{def.BrownMeasure}, i.e.\ $\nu_z$ is the spectral measure of $|b_0b(t)-z|$.  By definition, these measures satisfy \ref{BoCh.b'} with $\mu=\mu_{b_0,t}$. 

It only remains to verify condition \ref{BoCh.b} of the lemma.
From the Borel--Cantelli lemma it suffices to show that, for every $\varepsilon>0$ and a.e.\ $z\in\C$, there exists $H=H(\varepsilon, z)<\infty$ such that 
\begin{equation}
    \label{log-goal1}
    \P\bigg( \int_{\{|\log s|>H\}} |\log s| d\mu_{|B_0B(t)-z|}(s)>\varepsilon\bigg)\le N^{-2}
\end{equation}
for all $N$ sufficiently large. 
We show this holds for all $z\in\C$. 

Fix $z\in\C\setminus\{0\}$ arbitrary. 
We henceforth suppress from the notation the dependence of constants on the fixed parameters $\rho,\zeta,t,z$, and we assume without comment that $N$ is sufficiently large. 
From Theorem \ref{ldnvsdkn2}, we have with probability at least $1-\exp(-N^{2/3})$ that $\|B(t)\|\leq K$ for some constant $K$, and hence that
\begin{align*}
    \int_{\{\log s>H\}} |\log s| d\mu_{|B_0B(t)-z|}(s) &\leq e^{-H} \int_{\{s>\exp(L)\}} s^2 d\mu_{|B_0B(t)-z|} (s) \\
    \leq e^{-H} &\ts_N\left(|B_0B(t)-z|^2\right) \\
    \leq e^{-H} &\left(K^2\ts_N\left(|B_0|^2\right) +2K|z|\ts_N\left(|B_0|^2\right)^{1/2} + |z|^2\right).
\end{align*}
Since we assumed that $B_0$ converges in $*$-distribution, in particular $\ts_N\left(|B_0|^2\right)$ converges when $N$ goes to infinity and therefore its supremum over $N$ is bounded. Thus, with probability at least $1-\exp(-N^{2/3})$, we have that 
\[
\int_{\{\log s>H\}} |\log s| d\mu_{|B_0B(t)-z|}(s) = \mathcal{O}( e^{-H}). 
\]
It thus suffices to show that, with sufficiently high probability,
\begin{equation}
    \label{log-goal2}
     \int_0^{\delta} |\log s| d\mu_{|B_0B(t)-z|}(s) \lesssim o(1) + f(\delta)
\end{equation}
for some function $f(\delta)=o_{\delta\to0}(1)$. Note that for any $\eps>0$, 
\begin{align*}
    \P\left( \sigma_N(B_0B(t)-z) \leq \varepsilon \right) &= \P\left( \sigma_N\left((B_0-z(B(t))^{-1}) \cdot B(t)\right)  \leq \varepsilon \right) \\
    \leq \P&\left( \sigma_N(B(t))  \leq \sqrt{\varepsilon} \right) + \P\left( \sigma_N(B_0-zB(t)^{-1})  \leq \sqrt{\varepsilon} \right),
\end{align*}
where we used in the last line that for any matrices $X,Y$, $\sigma_{N}(XY) \geq \sigma_{N}(X)\sigma_{N}(Y)$. Since we know that $\widetilde{B}^N(t) \deq ((B(t))^{-1})^*$ is a solution of the SDE \eqref{eq:inverse}, $(B(t))^{-1}$ is also a multiplicative Brownian motion. Thus, since we assumed $z\ne 0$, we can apply Proposition \ref{prop:smallest} to $B(t)$ and $(B(t))^{-1}$ along with the union bound to conclude
\begin{equation}    \label{N.LB}
    \sigma_N(B_0B(t)-z) \ge N^{-\mathcal{O}(1)}
\end{equation}
with probability at least $1-N^{-10}$, say.
Similarly, 
from Proposition \ref{prop:meso} we have that with probability at least $1-N^{-10}$,
\[
\min\Big\{\sigma_{N-k}(B(t))\,,\; \sigma_{N-k}(B_0-z(B(t))^{-1})\Big\} \gs (k/N)^2\quad \forall\, C\log N\le k\le \sqrt{t}N
\]
for a sufficiently large constant $C$. 
Then from the fact that for any matrices $X,Y$, $\sigma_{N-2k}(XY) \geq \sigma_{N-k}(X)\sigma_{N-k}(Y)$, we conclude that
\begin{equation}    \label{N-k.LB}
    \sigma_{N-k}(B_0B(t)-z) \gs \min\{(k/N)^4,t^2\} \qquad \forall \;  k\ge 2C\log N
\end{equation}
with probability at least $1-2N^{-10}$.

For brevity we denote the singular values of $B_0B(t)-z$ by $\sigma_i:=\sigma_i(B_0B(t)-z)$.
For any $\delta\in(0,1)$, on the probability $1-\mathcal{O}(N^{-10})$ event $\cG$ that \eqref{N.LB}, \eqref{N-k.LB} hold,
\begin{align*}
    \int_0^\delta \log(1/s) d\mu_{|B_0B(t)-z|}(s) 
    &= \frac1N \sum_{k=0}^{N-1} \log\frac1{\sigma_{N-k}} 1_{\sigma_{N-k}\le\delta}\\
    &\le 2C\frac{\log N}{N}\log\frac1{\sigma_{N}} + \frac1N \sum_{k\ge 2C\log N} \log \frac1{\sigma_{N-k}} 1_{\sigma_{N-k} \le \delta}\\
    &\lesssim \frac{\log^2N}N 
    + \Big((\delta/c)^{1/4}-\sqrt{t}\Big)_+
    + \frac1N \sum_{k\le (\delta/c)^{1/4}N} \log\frac Nk \\
    &\lesssim \frac{\log^2N}N 
    + \Big((\delta/c)^{1/4}-\sqrt{t}\Big)_+
    + \int_0^{(\delta/c)^{1/4}}\log(1/x)dx
\end{align*}
where in the final line we recognized the sum over $k$ as a Riemann sum for the integral of $\log$ near the origin. 
Since $\log$ is locally integrable under the Lebsegue measure, we hence obtain \eqref{log-goal2} to conclude the proof of Theorem \ref{thm.main}.\\

\subsection{Convergence of the Spectrum}
\label{sec:cvg-Hausdorff}

In this final section, we prove that the spectrum of the Brownian motion, with a unitary initial condition, converges a.s.\ as $N\to\infty$ in {\em Hausdorff distance}.  Precisely: for any subset $H\subset\C$ and $\varepsilon>0$, we denote by $H+\varepsilon$ the Euclidean $\varepsilon$-neighborhood of $H$,
\[ H+\varepsilon \deq \{ \lambda\in\C\ |\ \exists z\in H\colon |z-\lambda| < \varepsilon \}. \]
The Hausdorff distance between two compact sets $H,K$ is 
\[ \mathrm{d}_{\mathrm{Haus}}(H,K) = \inf\{\varepsilon>0\colon H\subset K+\varepsilon \;\&\; K\subset H+\varepsilon\}. \]
Ergo, given compact sets $\Sigma,\Sigma^N$, to stay $\mathrm{d}_{\mathrm{Haus}}(\Sigma^N,\Sigma)\to 0$ is precisely to say that, for any $\varepsilon>0$, for all sufficiently large $N$ it holds true that $\Sigma \subset \Sigma^N+\varepsilon$ and $\Sigma^N\subset \Sigma+\varepsilon$.

Our main result in Theorem \ref{thm.main}, weak convergence of Brown(ian) measures, is not enough on its own to prove convergence of the {\em support} of the measures, i.e.\ the spectrum.  But coupled with strong convergence and some limit spectral results, we prove the following.

\begin{theorem} \label{thm.Hausdorff} Let $\rho>0$ and $\zeta\in\C$ with $|\zeta|\le\rho$. Let $B_{\rho,\zeta}$ be a Brownian motion on $\mathrm{GL}(N,\C)$, and let $b_{\rho,\zeta}$ be a free multiplicative Brownian motion.  Let $u_0$ be a unitary operator freely independent from $b_{\rho,\zeta}$, and suppose that for each $N$ $U_0\in\mathrm{U}(N)$ is a random unitary matrix independent from $B_{\rho,\zeta}$ with the property that, a.s., $U_0$ converges strongly to $u_0$ as $N\to\infty$.

Then for any fixed $t\ge 0$ the spectrum of $U_0B_{\rho,\zeta}(t)$ converges a.s.\ in Hausdorff distance to the spectrum of $u_0b_{\rho,\zeta}(t)$, which coincides with the support $\overline{\Sigma(u_0,t\rho,t\zeta)}$ of the Brown measure of $u_0b_{\rho,\zeta}(t)$.
\end{theorem}

We prove Theorem \ref{thm.Hausdorff} as a consequence of the following more general result which may be of independent interest.  We state it here for fixed (deterministic) matrices; in its application to prove Theorem \ref{thm.Hausdorff}, all {\em almost sure} statements refer to single probability $1$ events on which they hold true, and so we may treat all matrices as fixed for this purpose.

\begin{prop} \label{prop.Hausdorff.conv} Let $a$ be an operator in a $W^\ast$-probability space $(\A,\tau)$, and for each $N\in\N$ let $A^N$ be a matrix in $\M_N(\C)$.
\begin{enumerate}
    \item\label{strong.1} If $A^N$ converges strongly to $a$, then for each $\varepsilon>0$ and all sufficiently large $N\in\N$ the spectrum of $A^N$ is contained in the $\varepsilon$-neighborhood of the spectrum of $a$: $\sigma(A^N)\subset\sigma(a)+\varepsilon$.
    \item\label{strong.2} If the ESD $\mu_{A^N}$ of $A^N$ converges weakly to the Brown measure $\mu_a$ of $a$, then for each $\varepsilon>0$ and all sufficiently large $N\in\N$ the support of $\mu_a$ is contained in the $\varepsilon$-neighborhood of the spectrum of $A^N$: $\supp\mu_a\subset \sigma(A^N)+\varepsilon$.
    \item\label{strong.3} Ergo: if $A^N\to a$ strongly and $\mu_{A^N}\to\mu_a$ weakly, and in addition $\supp\mu_a = \sigma(a)$, then $\sigma(A^N)\to\sigma(A)$ in Hausdorff distance.
\end{enumerate}   
\end{prop}
In the case that $A^N$ and $a$ are selfadjoint, \ref{strong.1} was proved in \cite[Prop.\ 2.1(3)]{CollinsMale2014}; the non-selfadjoint setting has not been covered in any literature known to us.

\begin{proof} We begin with \ref{strong.2}. Fix $\varepsilon>0$.  Since $\supp\mu_a\subseteq \sigma(a)$ it is compact, there are finitely many points $z_1,\ldots,z_m\in\supp\mu_a$ so that the union of the open balls $\{B(z_j,\varepsilon/2)\}_{1\le j\le m}$ covers $\supp\mu_a$. For each $j$, let $\psi_j$ be a $C_c(\C)$ function which is strictly positive on $B(z_j,\varepsilon/2)$ and $0$ elsewhere.  By assumption in \ref{strong.2},
\[ \int_{\C} \psi_j\,d\mu_{A^N} \to \int_{\C} \psi_j\,d\mu_a \quad a.s. \]
Since $\psi_j>0$ on a neighborhood of $z_j\in \supp\mu_a$, $\int_{\C} \psi_j\,d\mu_a > 0$; hence it follows that $\int_{\C} \psi_j\,d\mu_{A^N} >0$ for all large $N$.  This implies that there is an element $\lambda_j^N\in\sigma(A^N)$ in $B(z_j,\varepsilon/2)$ for all large $N$.  Now, if $\zeta\in B(z_j,\epsilon/2)$, it follows that $|\zeta-\lambda_j^N|\le |\zeta-z_j|+|z_j-\lambda_j^N| < \varepsilon$ and so $\zeta\in\sigma(A^N)+\varepsilon$ for all large $N$.  As $\supp\mu_a$ is contained in the union of the $B(z_j,\varepsilon/2)$, it follows that $\supp\mu_a\subset\sigma(A^N)+\varepsilon$ for all large $N$, as claimed.

We now turn to \ref{strong.1}.  If $\lambda\notin \sigma(a)$, then $a-\lambda$ is invertible, and thus so is $|a-\lambda|^2$. In particular, since the spectrum of $|a-\lambda|^2$ is compact, $c_\lambda:=\inf\sigma(|a-\lambda|^2)>0$.  Since $A^N$ converges strongly to $a$, it follows that $|A^N-\lambda I|^2$ converges strongly to $|a-\lambda|^2$.  As $|A^N-\lambda I|^2$ and $|a-\lambda|^2$ are selfadjoint, it thence follows from \cite[Prop.\ 2.1(3)]{CollinsMale2014} that $\sigma(|A^N-\lambda I|^2)$ converges to $\sigma(|a-\lambda|^2)$ in Hausdorff distance.  In particular $\inf\sigma(|A^N-\lambda I|^2) > \frac12\inf\sigma(|a-\lambda|^2) = c_\lambda/2$ for all sufficiently large $N$, say $N\ge N_\lambda$.  Note that $\inf \sigma(|A^N-\lambda|^2) = (\sigma_N(A^N-\lambda I))^2$ is the square of the smallest singular value of $A^N-\lambda I$.

Now, for any $z\in\C$,
\[ A^N - zI = A^N-\lambda I + (\lambda-z)I = (A^N-\lambda I)\left(I+(\lambda-z)(A-\lambda I)^{-1}\right). \]
Since the smallest singular value $\sigma_N$ is a supermultiplicative function on $\M_N(\C)$,
\begin{equation} \label{eq.strong.conv.1} \sigma_N(A^N-z I) \ge \sigma_N(A^N-\lambda I)\cdot\sigma_N\left(I+(\lambda-z)(A-\lambda I)^{-1}\right). \end{equation}
If $|\lambda-z|<\sqrt{c_\lambda/2}$, then for $N\ge N_\lambda$, $\sigma_N(A^N-\lambda I) > \sqrt{c_\lambda/2}>|\lambda-z|$.  But then
\[ |\lambda-z| \|(A^N-\lambda I)^{-1}\| = |\lambda-z|\cdot\frac{1}{\sigma_N(A^N-\lambda I)} < 1 \]
and it follows that $I+(\lambda-z)(A-\lambda I)^{-1}$ is invertible.  Combined with \eqref{eq.strong.conv.1}, this shows that $A^N-zI$ is invertible whenever $|z-\lambda|<\sqrt{c_\lambda/2}$, for $N\ge N_\lambda$.  In summary: for any $\lambda\notin\sigma(a)$, the open ball $B(\lambda,\sqrt{c_\lambda/2})$ is disjoint from $\sigma(A^N)$ for all $N\ge N_\lambda$.

Now, since $A^N$ converges strongly to $a$, $R = \sup_N \norm{A^N}$ is finite, and so $\sigma(A^N)$ is contained in the closed ball $\overline{B(0,R)}$ for all $N$.  For any $\varepsilon>0$, $\overline{B(0,R)}\setminus(\sigma(a)+\varepsilon)$ is compact; hence, there is a finite subset $\Lambda$ of $\lambda$ in this set so that
\[ \overline{B(0,R)}\setminus(\sigma(a)+\varepsilon) \subset \bigcup_{\lambda\in\Lambda} B(\lambda,\sqrt{c_\lambda/2}). \]
Thus, for $N\ge \max\{N_\lambda\colon \lambda\in\Lambda\}$, $A^N-zI$ is invertible for any $z\in\overline{B(0,R)}\setminus(\sigma(a)+\varepsilon)$.  In other words: for all large $N$, $\overline{B(0,R)}\setminus(\sigma(a)+\varepsilon)$ is contained in $\C\setminus\sigma(A^N)$.  As $\sigma(A^N)\subset\overline{B(0,R)}$, it therefore follows taking complements that $\sigma(A^N)\subseteq \sigma(a)+\varepsilon$ as claimed.

Finally, \ref{strong.3} follows immediately from \ref{strong.1} and \ref{strong.2} together with the stated assumption that $\supp\mu_a$ is equal to the full spectrum $\sigma(a)$.
\end{proof}

\begin{proof}[Proof of Theorem \ref{thm.Hausdorff}] From \cite[Corollary 1.2]{BCC2025}, it follows that $B_0B_{\rho,\zeta}(t)$ converges strongly to $b_0b_{\rho,\zeta}(t)$ a.s.\ for any sequence of initial conditions $B_0$ converging strongly to $b_0$; hence from Proposition \ref{prop.Hausdorff.conv}\ref{strong.1} we have $\sigma(B_0B_{\rho,\zeta}(t))$ eventually contained in any neighborhood of $\sigma(b_0b_{\rho,\zeta}(t))$ a.s.  Conversely, from our Theorem \ref{thm.main}, the ESD of $B_0B_{\rho,\zeta}(t)$ converges weakly a.s.\ to the Brown measure of $b_0b_{\rho,\zeta}(t)$; hence, from Proposition \ref{prop.Hausdorff.conv}\ref{strong.2} it follows that the support of the Brown measure of $b_0b_{\rho,\zeta}(t)$ is eventually contained in any neighborhood of the spectrum of $B_0B_{\rho,\zeta}(t)$.

Now restricting to the case that $B_0=U_0$ and $b_0=u_0$ are unitary as in the theorem statement, we have that the support of the Brown measure of $u_0b_{\rho,\zeta}(t)$ is equal to $\overline{\Sigma(u_0,t\rho,t\zeta)}$ (see our Theorem \ref{thm.Brown.meas.rhozeta}).  Moreover, \cite[Theorem 3.7]{HallHo2025Spectrum} proves that the Brown measure support is exactly equal to the spectrum $\sigma(u_0b_{\rho,\zeta}(t))$; hence, we conclude the statement of the theorem.
\end{proof}

As a final remark: we expect it to hold true in full generality that the spectrum of $B_0B_{\rho,\zeta}(t)$ converges in Hausdorff distance to the spectrum of $b_0b_{\rho,\zeta}(t)$ a.s., for any initial conditions $B_0$ converging strongly to $b_0$ a.s.  The only technical obstacle to this is extending the results of \cite{HallHo2025Spectrum} beyond the unitary initial condition case.

\appendix

\section{Small Singular Values of Shifted Gaussian Matrices}
\label{app:small}

\subsection{Proof of Lemma \ref{lem:smallest}}
\label{app:smallest}

In this section we prove the GUE case of Lemma \ref{lem:smallest}, restated below. 
As we noted, the Ginibre case was established in \cite[Lemma 3.3]{sinmin}.

It will be convenient to work with a GUE matrix $G$ scaled to have entries of unit variance.
Thus, letting $G$ have the distribution of $\sqrt{N}X^N(1)$, with $X^N$ as in Definition \ref{2HBdef}, it suffices to prove the following:

\begin{lemma}[Non-normal shift of a GUE matrix]
\label{lem:ssv.GUE}
    Let $G$ be a scaled $N\times N$ GUE matrix as above, and let $A$ be an arbitrary complex $N\times N$ matrix, deterministic or random and independent of $G$. For any $\delta\ge0$,
        \[
        \P( \sigma_N(G+A) \le \delta) \le  CN^2\delta^{1/8}\,.
        \]
        for a universal constant $C$. 
\end{lemma}

Previous works have controlled the smallest singular value for \emph{Hermitian} shifts $H+A$ of Hermitian random matrices $H$ with independent centered entries on and above the diagonal satisfying some tail or anti-concentration assumptions. The first results of this type were by Nguyen  \cite{Nguyen:symmetric}
and Vershynin \cite{Vershynin:symmetric}.
An optimal bound on the lower tail for unshifted Hermitian random matrices $H$ was recently obtained in \cite{CJMS:symmetric}; we refer to that work for a detailed review of the literature for unshifted Hermitian random matrices.

For matrices with Gaussian entries, it is common to obtain bounds that are completely uniform in $A$, with no constraint on the norm. This was done in the non-normal case of shifted Ginibre matrices in the well-known work of Shankar--Spielman--Teng on the smoothed analysis of Gaussian elimination \cite{SST}. Sharp bounds for the case that $H$ is GOE were obtained by Bourgain in \cite{Bourgain:symmetric}; another argument that also covers GUE matrices was found in \cite{APSSS}.

In the present work we need 
to allow the deterministic shift $A$ to be non-normal. Hence the matrix is generically non-normal, and the above-mentioned sharp results for Hermitian  $A$ do not apply. 

For the proof we follow the approach from \cite{Vershynin:symmetric} reducing the problem to bounding the lower tail for the distance of a column to the span of remaining columns, which can be expressed in terms of entries of the first row and column of $G+A$ and the inverse of the complementary $N-1$-dimensional principal submatrix -- see \eqref{def:Q}. The difference here from \cite{Vershynin:symmetric} is that there is some additional case analysis due to the fact that the first row and column of $A$ are completely arbitrary (and may be very large). On the other hand, we do not need hard results from the Littlewood--Offord theory for anticoncentration of random walks as in \cite{Vershynin:symmetric}, due to the Gaussianity of entries of $G$.

\begin{proof}
Without loss of generality we take $A$ to be deterministic. We may also assume $\delta\le\frac12$.  Let $S_i\subset S^{N-1}$ be the set of complex unit vectors $x$ such that $|x_i|\ge 1/\sqrt{N}$. 
Then $S^{N-1}\subset \bigcup_{i=1}^N S_i$, so by the union bound,
\[
    \P ( \sigma_N(G+A) \le \delta) 
    =\P\bigg( \inf_{x\in S^{N-1}} \|(G+A)x\|_2\le \delta\bigg)
    \le \sum_{i=1}^N \P\bigg( \inf_{x\in S_i} \|(G+A)x\|_2\le \delta\bigg).
\]
By permuting coordinates we see it suffices to show
\begin{equation}    \label{GUE-goal1}
    \P( \cE) \ls N\delta^{1/8}\,,\qquad \cE:= \bigg\{ \inf_{x\in S_1} \|(G+A)x\|_2\le \delta\bigg\}.
\end{equation}

Let $Y\in\C^{N-1}$ be the first column of $G$ with first entry removed, and let $B,\wt B^\tran$ be respectively the first column and row of $A$ with first entry removed. Thus,
\begin{equation}
    G+A = 
    \begin{pmatrix}
        G_{1,1}+ A_{1,1} & Y^* + \wt B^\tran\\
        Y + B & M^\tran 
    \end{pmatrix}
\end{equation}
for an $(N-1)\times (N-1)$ matrix $M$ that is independent of $G_{1,1},Y$. We will only use that $M$ is invertible almost surely. ($M$ is a.s. invertible because its distribution is absolutely continuous with respect to the Lebesgue measure on $\mathbb{C}^{N^2}$, and the collection of non-invertible matrices viewed as a subset of $\mathbb{C}^{N^2}$ has measure 0. The latter property is a consequence of the fact that the zero set of a polynomial has Lebesgue measure 0.)

From \cite[Lemma 3.2]{RudelsonVershynin2014} (see also \cite[Proposition 5.1]{Vershynin:symmetric}) we have
\begin{equation}    \label{EinQ}
    \cE\subset  \big\{ Q\le \delta \sqrt{N}\big\}
\end{equation}
where 
\begin{equation}\label{def:Q}
    Q:= \frac{|G_{1,1}+A_{1,1}- (Y+B)^\tran M^{-1}(\overline Y + \wt B)|}{\sqrt{1+\|M^{-1}(\overline Y + \wt B)\|_2^2}}\,.
\end{equation}
(This is the distance of the first column of $G+A$ to the span of the other $N-1$ columns.)
Letting $K_1,K_2,K_3>0$ to be chosen later, define events
\begin{equation*}
    \cE_1:= \big\{ \|M^{-1}(\overline Y + \wt B)\|_2\le K_1\big\}
    \,,\quad
    \cE_2:= \big\{ \|M^{-1}\|\ge K_2\big\}
    \,,\quad
    \cE_3:= \big\{ \|Y\|_2 \ge K_3\sqrt{N}\big\}.
\end{equation*}
The rest of the proof divides into four claims:

\begin{claim}   \label{claimE1}
    $\P( \{Q\le \delta \sqrt{N}\}\cap \cE_1) \ls \delta(1+K_1^2)^{1/2}\sqrt{N}$.
\end{claim}

\begin{claim}   \label{claimE2}
    $\P(\cE\cap \cE_2) \ls (\delta^2 + K_2^{-2})N$.
\end{claim}

\begin{claim}
    \label{claimE3}
    $\P( \cE_3) \le K_3^{-2}$.
\end{claim}

\begin{claim}
    \label{claimE4}
    With $K_1=\delta^{-1/2}$, $K_2=K_3\sqrt{N}=\delta^{-1/8}$, 
    \[
    \P( \{Q\le \delta\sqrt{N} \}\cap \cE_1^c\cap\cE_2^c\cap\cE_3^c) \ls \delta^{1/8}N.
    \]
\end{claim}

Combining the claims with \eqref{EinQ} we obtain \eqref{GUE-goal1} as desired. It only remains to prove the four claims.

\begin{proof}[Proof of Claim \ref{claimE1}]
For this claim we use the randomness of $G_{1,1}$.
On $\cE_1$ we have
\[
Q \ge (1+K_1^2)^{-1/2} |G_{1,1} + R|
\]
where $R:= A_{1,1}- (Y+B)^\tran M^{-1}(\overline Y + \wt B)$ is independent of $G_{1,1}$. 
Conditioning on $R$, the claim follows from the fact that the standard real Gaussian variable $G_{1,1}$ has bounded density. 
\end{proof}

\begin{proof}[Proof of Claim \ref{claimE2}]
For this claim we use the randomness of $Y$.
Condition on a realization of $M$ satisfying $\cE_2$. 
Let $u\in S^{N-2}$ such that $\|M^{-1}u\|_2\ge K_2$. We may select such a vector $u$ deterministically under the conditioning on $M$.
On $\cE$, fix $x=(x_1, \tilde x)\in S_1$ such that $\|(G+A)x\|_2\le \delta$.
With $v:= M^{-1}u/\|M^{-1}u\|_2$, on $\cE\cap \cE_2$ we thus have
\begin{align*}
    \delta \ge \|(G+A)x\|_2 
    &\ge |(0\; v^\tran)(G+A)x|\\
    &= |x_1v^\tran(Y+B) + v^\tran M^\tran \wt x|\\
    &\ge \frac1{\sqrt{N}}|v^\tran(Y+B)| - \|Mv\|_2\\
    &\ge\frac1{\sqrt{N}}|v^\tran(Y+B)| - \frac1{K_2}\,.
\end{align*}
Since $v$ is a fixed unit vector and $Y\in \C^{N-1}$ is standard Gaussian, $g:=v^\tran Y$ is standard Gaussian in $\C$ and in particular has bounded density.
Hence,
\[
\P( \cE \cap \cE_2) \le \P \big( |g + v^\tran B|\le \sqrt{N}(\delta+ K_2^{-1})\big)\ls (\delta^2 + K_2^{-2}) N\,.\qedhere
\]
\end{proof}

\begin{proof}[Proof of Claim \ref{claimE3}]
Since $Y$ is standard Gaussian in $\C^{N-1}$, the claim is immediate from Markov's inequality.
\end{proof}

\begin{proof}[Proof of Claim \ref{claimE4}]
Note that on $\cE_1^c\cap \cE_2^c\cap \cE_3^c$,
\begin{equation}
    \label{A1-top}
    K_1<\|M^{-1}(\overline Y + \wt B)\|_2 < K_2 K_3\sqrt{N} + \|M^{-1}\wt B\|_2\,.
\end{equation}
Combining these bounds, we get that the denominator in \eqref{def:Q} is bounded by
\begin{align}
    1+\|M^{-1}(\overline Y + \wt B)\|_2 
    &< \|M^{-1}\wt B\|_2 \bigg( 1+ \frac{1+K_2K_3\sqrt{N}}{\|M^{-1}\wt B\|_2}\bigg) \notag\\
    &\le \|M^{-1}\wt B\|_2 \bigg( 1+ \frac{1+K_2K_3\sqrt{N}}{K_1-K_2K_3\sqrt{N}}\bigg)\,.\label{Q-denom}
\end{align}
On $\cE_2^c\cap\cE_3^c$ the numerator in \eqref{def:Q} is bounded below by 
\begin{align}   \label{Q-num}
    |Y^\tran M^{-1}\wt B + B^\tran M^{-1}\overline Y + R'| - K_2K_3^2N 
\end{align}
where $R':= B^\tran M^{-1}\wt B - G_{1,1}-A_{1,1}$ is independent of $Y$. 
Decomposing the standard Gaussian vector $Y\in \C^{N-1}$ as $Y=\frac1{\sqrt{2}}( U+ iV)$ for $U,V$ standard real Gaussian vectors, we have
\[
Y^\tran M^{-1}\wt B + B^\tran M^{-1}\overline Y
= \sqrt{2}( a+ b) U + i \sqrt{2}(a-b) V
\]
where we denote the deterministic row vectors
\[
a:= (M^{-1}\wt B)^\tran \,,\quad b:= B^\tran M^{-1}\,.
\]
Combining with \eqref{Q-denom}, \eqref{Q-num}, we have that on $\{Q\le\delta\sqrt{N}\}\cap\cE_1^c\cap \cE_2^c\cap \cE_3^c$,
\begin{align*}
    \delta\sqrt{N}\ge Q &\ge 
    \frac{|\sqrt{2}(a+b)U + i\sqrt{2}(a-b)V + R'| - K_2K_3^2N}{\|a\|_2 \Big( 1+ \frac{1+K_2K_3\sqrt{N}}{K_1-K_2K_3\sqrt{N}}\Big)}\,.
\end{align*}
Rearranging and substituting the bound $\|a\|_2 \le \|a+b\|_2+\|a-b\|_2$, we get
\begin{align}   \label{zdot}
    \big| z \cdot (U, V) + & R'/\sqrt{2}(\|a+b\|_2+\|a-b\|_2) \big| \notag \\
   &  \le \delta\sqrt{N} \Big( 1+ \frac{1+K_2K_3\sqrt{N}}{K_1-K_2K_3\sqrt{N}}\Big) + \frac{K_2K_3^2N}{K_1-K_2K_3\sqrt{N}}
\end{align}
where for the second term on the right hand side we again used \eqref{A1-top} to bound $\|a\|_2 \ge K_1-K_2K_3\sqrt{N}$, and we set
\[
z:=\bigg( \frac{a+b}{\|a+b\|_2+\|a-b\|_2}, \frac{i(a-b)}{\|a+b\|_2+\|a-b\|_2}\bigg) \in \C^{2N-2}.
\]
With $K_1=\delta^{-1/2}$, $K_2=K_3\sqrt{N}=\delta^{-1/8}$, \eqref{zdot} becomes
\begin{equation}\label{zdot'}
   \big| z \cdot (U, V) + R'/\sqrt{2}(\|a+b\|_2+\|a-b\|_2) \big| \ls \delta^{1/8}N
\end{equation}

Clearly $\|z\|_2\gs 1$, so at least one of its real or imaginary parts has norm $\gs 1$. 
It follows that at least one of the real or imaginary parts of $z\cdot (U,V)$ is Gaussian of variance $\gs 1$, since $(U,V)$ is standard Gaussian in $\R^{2N-2}$. 
Hence,
\[
\sup_{w\in \C} \P ( |z\cdot (U,V) + w| \le \eps) \ls \eps
\]
for any $\eps\ge0$. 
Together with \eqref{zdot'} this yields the claim.
\end{proof}

The proof of Lemma \ref{lem:ssv.GUE} is complete.
\qedhere
\end{proof}

\subsection{Proof of Lemma \ref{lem:meso}}
\label{app:meso}

In this appendix we prove Lemma \ref{lem:meso}, restated below. 

\begin{lemma}\label{lem:meso.app}
    There is a universal constant $c>0$ such that the following holds. 
    Let $G$ be an $N\times N$ matrix with standard real Gaussian entries, and assume the $N(N+1)/2$ pairs $\{(G_{ij},G_{ji}): 1\le i\le j\le N\}$ are jointly independent. 
    Then for any $N\times N$ complex matrix $A$, deterministic or random and independent of $G$, and any $k\ge c^{-1}\log N$, the event that
    \[
    \sigma_{N-\ell} ( G + A) \ge c\ell/N\qquad \text{for all $k\le \ell\le N-1$}
    \]
    holds with probability at least $1-e^{-ck}$.
\end{lemma}

We prove Lemma \ref{lem:meso.app} via the geometric argument of Tao and Vu \cite{TaoVu2010-aop}, 
with a small modification to allow dependency between entries $G_{ij}$ and $G_{ji}$. The first step in this approach is to establish the following lower bound for $\sigma_{N-j}(A)$ for an arbitrary $N \times N$ matrix $A$. Informally, the lemma below provides a lower bound on $\sigma_{N-j}(A)$ in terms of the extent to which the rows of $A$ are linearly independent. 

\begin{lemma}\label{lem:interSVlb}
    For any $N \times N$ matrix $A$ with complex entries, we have that 
    \begin{equation}
        \label{LB-drows}
        \sigma_{N-j}(A) \gtrsim \sqrt{\frac{j}{N}} \min_{i \in [m]} \text{\emph{dist}}(R_i, V_i)
    \end{equation}
    where $m = N - \lceil j/2 \rceil$, $R_i$ denotes the $i$th row of $A$, and $V_i = \text{\emph{span}}(R_{\ell}: \ell \neq i, \ell \leq m)$.
\end{lemma}

\begin{proof}
Let $M$ denote the matrix consisting of the top $m$ rows of $A$. 
We may assume the the right hand side of \eqref{LB-drows} is positive, i.e.\ the rows of $M$ are linearly independent, as otherwise the claim holds trivially.
By Cauchy interlacing, we have that
\begin{equation}    \label{M-A}
    \sigma_{N-j}(A) \geq \sigma_{N-j}(M).
\end{equation}
    Combining with the inverse second moment identity (see  \cite[Lemma A.4]{TaoVu2010-aop}):\footnote{Proof: By projecting the rows to their span we reduce to the case that $M$ is square and invertible. Then one sees that the expressions in \eqref{inv2mom} are two ways of writing the Frobenius norm of $M^{-1}$.}
    \begin{equation}
        \label{inv2mom}
        \sum_{i=1}^m \text{dist}(R_i, V_i)^{-2} = \sum_{i=1}^m \sigma_i(M)^{-2}, 
    \end{equation}
    we get that
    \begin{align*}
        N \max_{i \in [m]} \text{dist}(R_i, V_i)^{-2} \geq \sum_{i=1}^m \sigma_i(M)^{-2} 
        \geq \sum_{i=N-j}^m \sigma_i(M)^{-2} \geq & \frac{j-1}{2} \sigma_{N-j}(M)^{-2} \,.
    \end{align*}
    Combining with \eqref{M-A} and rearranging, we obtain the desired conclusion.     
\end{proof}

With Lemma \ref{lem:interSVlb} in hand, we complete the proof of Lemma \ref{lem:meso.app} using concentration of measure.

\begin{proof}[Proof of Lemma \ref{lem:meso.app}]
    We begin by noting that it suffices to show that, by union bound, there exists a constant $c > 0$ (possibly different from the one in the lemma statement) such that for every $j \geq k$
    \begin{equation}
    \label{eq:sing_val_anticonc}
        \mathbb{P}\left(\sigma_{N-j}(X) \leq cj/\sqrt{N}\right) \leq e^{-cj}.
    \end{equation}
    Moreover, applying Lemma \ref{lem:interSVlb} with $A = \frac{1}{\sqrt{N}} X$, we see it suffices to show there exists a universal constant $c$, possibly different from the one in \eqref{eq:sing_val_anticonc}, such that
\begin{equation}
\label{eq:dist_conc_target}
  \mathbb{P}(\text{dist}(R_i, V_i) \leq c\sqrt{j}) \leq e^{-cj}. 
\end{equation}
    For an $N$-dimensional row vector $y$ we write $y^{(i)}$ for the $(N-1)$-dimensional vector obtained by removing the $i$th entry of $y$.
    Set 
    \begin{align*}
        V_i' = \text{span}\left(R_{\ell}^{(i)}: \ell \neq i, \ell \leq m\right)\,,\qquad
        V_i'' = \text{span}\left(V_i', Y^{(i)}_i\right)
    \end{align*}
    which are subspaces of $\C^{N-1}$.
    Then, 
    $$\text{dist}(R_i, V_i) \geq \text{dist}\left(R_i^{(i)}, V_i'\right) \geq \text{dist}\left(G_{i}^{(i)}, V_i''\right).$$
    By our assumptions, $G_{i}^{(i)}$ is a standard Gaussian vector and $V_i''$ is an independent subspace of dimension $\leq N-\lceil \frac{j}{2 }\rceil +1$. Additionally, notice that $x \mapsto \text{dist}(x, V_i'')$ is a 1-Lipshitz function of $x$. Thus, we can apply Gaussian concentration, conditional on $V_i''$, to obtain
    \begin{equation}
\label{eq:gauss_conc_dist}
\mathbb{P}\left(\mathbb{E}\left[\text{dist}\left(G_{i}^{(i)}, V_i''\right) \,\Big\vert\, V_i''\right] - \text{dist}\left(G_{i}^{(i)}, V_i''\right) \geq t \,\Big\vert\, V_i''\right) \leq e^{-t^2/2}.
    \end{equation}
It remains to estimate $\mathbb{E}\left[\text{dist}\left(G_{i}^{(i)}, V_i''\right) \big\vert V_i''\right]$. First, observe that
$$\mathbb{E}\left[\text{dist}\left(G_{i}^{(i)}, V_i''\right)^2 \,\Big\vert\, V_i''\right] = N - 1 -\dim(V_i'').$$
By the Gaussian Poincar\'e inequality, 
$$\text{Var}\left(\text{dist}\left(G_{i}^{(i)}, V_i''\right) \,\Big\vert\, V_i''\right) \leq 1.$$
(One also sees the variance is $\mathcal{O}(1)$ from integrating the tail bound \eqref{eq:gauss_conc_dist}, along with the same bound for the lower tail.)
Therefore,
$$\mathbb{E}\left[\text{dist}\left(G_{i}^{(i)}, V_i''\right) \,\Big\vert\, V_i''\right] \geq (N-2-\dim(V_i''))^{1/2} \gtrsim \sqrt{j}.$$
Combining the above lower bound with \eqref{eq:gauss_conc_dist}, we can see that \eqref{eq:dist_conc_target} holds and the desired conclusion follows.
\end{proof}

\begin{ack}
\phantomsection
\addcontentsline{toc}{section}{Acknowledgments}
We are grateful to Brian Hall and Ching Wei Ho, who provided help (and code) to create the figures in Sections \ref{Intro} and \ref{sect.Brown(ian).meas}, as well as insightful discussion related to Conjecture \ref{conj.b0}.
\end{ack}

\phantomsection
\bibliographystyle{amsplain}
\bibliography{GLBM-Eig}

\end{document}